\documentclass[a4paper,10pt,reqno]{article}
\usepackage[utf8x]{inputenc}
\usepackage[leqno]{amsmath}
\usepackage{amsfonts}
\usepackage{amssymb}
\usepackage{amsthm}
\usepackage[T1]{fontenc}
\usepackage{graphicx}
\usepackage{srcltx}

\newtheorem{thm}{Theorem}
\newtheorem{cor}{Corollary}
\newtheorem{lem}{Lemma}
\newtheorem{prop}{Proposition}

\theoremstyle{definition}

\theoremstyle{remark}
\newtheorem*{rmk}{Remark}

\newcommand{\R}{\mathbb R}

\newcommand{\N}{\mathbb N}

\newcommand{\T}{\mathbb T}
\newcommand{\rot}{\operatorname{curl}}
\newcommand{\ka}[1]{\mathcal{C}^{#1}}
\newcommand{\Lip}[1]{\mathcal{\text{Lip}}{\, #1}}
\newcommand{\Div}{\operatorname{div}}

\newcommand{\dmt}{diam}

\title{Analyticity of the Flow for the Aggregation Equation.}
\author{J. M. Burgués \and J. Mateu}

\begin{document}
	
	\maketitle
	
	\begin{abstract} Let $\Omega$ be a bounded domain in $\R^n$ whose boundary is $\ka{1,\,\gamma}$ for $\gamma\in(0,\,1)$. Consider the aggregation equation in the case of the initial condition being a positive multiple of the characteristic function of $\Omega$. In this paper we prove global in time analyticity of the flow generated by the velocity field which propagates the density solution of this equation. 
	\end{abstract}
	
	\section{Introduction}
	
	The aggregation equation is a classical partial differential equation that can be formulated in $\R^n$ for any $n\in\N\setminus\{0,\, 1\}$. Namely 	
	\begin{equation}\label{eq(A)}
		\begin{cases}
			\partial_t\rho(x,\, t)+\Div(\rho\, v)(x,\, t)=0, \ \ \ (x,t)\in \R^n\times(0,\, T^*)\\
			v(x,\, t)=K*\rho(\ ,\, t)(x)\\
			\rho(x,\, 0)=\rho_0(x),
		\end{cases}
		\tag{$A$}
	\end{equation} where $v$ represents a velocity field and $\rho$ the density of mass of an irrotational inviscid and compressible fluid. The vector-valued kernel $K$ is
	$$
	K(x)=\frac{x}{|x|^n}.
	$$
	
	It is is a continuity equation that in the cases of $n=2,\, 3$ can be related with biological systems and also with many other phenomena, as pointed in \cite{BeLaLe}.
	
	The velocity field $v$ provides a flow through the equation \begin{equation}\label{fla} 
		\begin{cases}
			\frac{\partial\psi}{\partial t}(x,t)=v(\psi(z,\, t),\, t)\\*[5pt]
			\psi(x,0)=x.
		\end{cases}
	\end{equation}
	
	The main result of this paper is
	
	\begin{thm}\label{Main} Let $\Omega\subset\R^n$ be a bounded domain 
		with $\partial\Omega\in\ka{1,\, \gamma}$ for $\gamma\in(0,\, 1)$. Assume $\psi(x,\, t)$ be the flow corresponding to the solution $(v,\, \rho)$ of the equation \eqref{eq(A)} when the initial condition $\rho_0(x)$ is  a constant multiple of the characteristic function of $\Omega$. Then, for each $x\in\R^n$, the function 
		$\psi(x,\, t)$
		is real analytic in the time variable $t$ in the interval of existence.
	\end{thm}
	\medskip

	An important observation is that after a non-linear rescaling of the time variable and the values of the function $\rho$ in the theorem, equation  \eqref{eq(A)} becomes a transport equation (see \cite{BeLaLe}):
	\begin{equation}\label{eqtildeA}
		\begin{cases}
			\partial_t\rho(x,\, t)+v(x,\, t)\cdot\nabla\rho(x,\, t)=0, \ \ \ (x,t)\in \R^n\times(0,\, \infty)\\
			v(x,\, t)=K*\rho(\ ,\, t)(x)\\
			\rho(x,\, 0)=\rho_0(x),	
		\end{cases}
		\tag{$\tilde A$}
	\end{equation}
	which is dual to the equation \eqref{eq(A)} in the weak sense.

	\subsection{Historical} 
	The equations (\ref{eq(A)}) and (\ref{eqtildeA}) are samples of the so-called active scalar equations, the class of partial differential equations where the evolution in time of a scalar quantity is governed by the motion of a fluid whose velocity itself varies with this scalar quantity. These equations include Euler equations for the dynamics of the vorticity of a perfect inviscid fluid and have been largely studied in several aspects. The most important of these are the existence and uniqueness of classical and weak solutions for different classes of initial data, their regularity in the space and time variables and the same topics for the associated flow (See \cite{Che}, \cite{BurMat}, \cite{BeLaLe}, \cite{BeGaLaVe}).
	
	Concerning the analyticity of the flow, in \cite{Her} a conceptual argument is developed for proving it in the case of some equations or systems related to very general fluid models. The method relies on ideas already introduced by Serfaty \cite{Ser} and consists in complexifying the time variable and considering an operator describing the time derivative of the flow in terms of the initial condition and the flow itself. Then the local boundedness and preservation of the analyticity of this operator in appropriated Banach spaces implies the analyticity of the flow. The method covers among other situations the vortex patch problem for the Euler equation in $\R^2$ in the case of $\ka{2}$ boundary but, as far as we know, it does not cover an identical problem for the aggregation equation. 
	
	In \cite{Shn} the analyticity of the flow for the Euler equation is proven in the case of the space domain being the torus $\T^3$ and the initial datum in a Sobolev space contained in $\ka{1}$. The key point consists in considering the flow as a geodesic, image of the exponential map, in the group of volume preserving diffeomorphisms of the space.
	
	In \cite{FrZh} the authors prove the analyticity of the flow in a sufficiently small interval of time in the case of $3D$ Euler equations for the vorticity in a torus $\T^3$ and for the initial vorticity in a Hölder class. The (a priori) method consists in the study of the size of the Taylor coefficients of the time-development of an hypothetical analytic solution. It is elementary enough to make possible the use of a relatively simple technology allowing a great generality. It has inspired our procedure in this paper. 
	
	Finally in the paper \cite{BurMat} the global analyticity in time of the trajectories of particles in the patch problem has been proven  for aggregation equation in dimension 2. The method covers automatically the patch problem for (\ref{eqtildeA}) in dimension 2, via a change of variables from \cite{BeLaLe}. Also the case of the patch problem for Euler equation in dimension $2$ is proven by this method, observing that in this dimension it is formally related to the aggregation.

	\subsection{Idea of the proof}
	
	Let us briefly sketch the procedure of the proof. It follows the lines of the one provided in \cite{BurMat} for the analogous theorem in dimension 2 . However, the general case needs developing and some extra ideas.
	
	It is worth highlighting the existence and uniqueness of a solution to the equation in the general dimension of the problem for the case of patches having bounded support. This was already proved in \cite{BeLaLe}. In this paper, the 
	key fact that the solution is a patch and also has bounded support at any instant of its evolution is also proved.
	
	As it is well known (see \cite{Che}), the velocity field associated with the solution is regular enough to have associated a flow $\psi$.
	We begin by proving that this flow is related to the solution $\rho(x, t)$ of the 
	equation (\ref{eq(A)}) by a new system of equations that turns out to be rather complicated. 
	To bypass this complexity we show that this system of equations can
	be reformulated as an operator between spaces of differential
	forms
	\begin{equation}\label{pflow}\begin{cases}
			
			P [\psi](x, t) = \rho(x,\, t)\,  V(x)\\
			
			\psi(x, 0) = x,
		\end{cases}
	\end{equation}
	where $P$ is a non-linear second order differential operator (see \cite{Sch}) from a space of maps to a space of forms, $V (x) = dx_1\wedge\dots\wedge
	dx_n$ is the volume form of the space $\R^n$ and $\rho(x,\, t) =\frac{\rho(x,\, 0)}{1-\rho(x,\, 0)\, t}$ . 
	This reformulation reduces considerably the complexity of the problem.
	
	The next step consists of using a standard "a priori" method to study equation (\ref{pflow}) in the case
	of the datum $\rho(x,\, 0)$ being a patch.
	In this case we prove that (\ref{pflow}) has an analytic solution at some interval too, not necessarily equal to the previous one. It is important to underscore here
	that in general the solution of (\ref{pflow}) is not unique. Nevertheless,  
	we obtain as a solution a time-parametrized family of homeomorphisms of $\R^n$, which makes it
	a flow. The proof involves on the one hand an accurate use of the machinery of differential forms and on the other hand it requires very technical estimates of Harmonic Analysis, some of which are new while others are adapted from well known concepts.
	
	The flow obtained as a
	solution of equation (\ref{pflow}) provides a velocity field $u$ and then a density, $\Xi$. The
	couple $(u, \Xi)$ satisfies the equation (\ref{eqtildeA}), the dual equation of (\ref{eq(A)}), so it is a
	weak solution of (\ref{eq(A)}). The uniqueness of solution of (\ref{eq(A)}) implies that the
	flows obtained from equation (\ref{fla}) and from equation (\ref{pflow}) coincide in the aforementioned interval.
	
	Using the persistence of the boundary regularity proved in \cite{BeGaLaVe}  and the topological properties and regularity results (see \cite{BeLaLe}) of solutions of the equation (\ref{eq(A)})
	for patches, the global analyticity in time of the solution of the aggregation equation is proven.

	\subsection{Plan of the paper} 
	
	The paper is developed as follows. In the first section we introduce the basic notation as well as a brief dictionary of differential forms, as used in the paper. Also a short compendium of the main operators used (mainly Riesz operators) and then a precise formulation of the main theorems.
	
	In section 2, the main known facts concerning the aggregation problem for patches are exposed. Mainly well posedness, boundedness and boundary regularity of the support of the solution for any fixed value of the time variable. Then cornerstone formulas that relate the flow corresponding to the solution with the initial patch are developed.
	
	Immediately we precisely state an auxiliary result (Theorem 2) and use it for proving the local version of Theorem 1. This implies the same result for the dual equation and the persistence of the regularity of the boundary of the support as proven in \cite{BeGaLaVe} allows an argument of connectivity completing the proof.
	
	Section 3 is devoted to the proof of Theorem 2. An a priori method of determining the coefficients of a possible analytic solutions of the flow-density equations is developed. Then we use an iterative method  for solving the resulting functional equations and a corresponding result (Theorem 3) for estimating the solutions leading to a sequence of inequalities.  Using rather combinatoric arguments we obtain local analytic solutions of the flow-density equations. 
	Also a jump formula for the Riesz kernel, as developed in section 4 is used in the case of the space variables being at the boundary of the patch. 
	
	Finally section 4 is the most technical in the paper. There we prove Theorem 3 that provides the estimates on the Riesz transform used in the proof of Theorem 2. Jump formula for this transform is proven in an Appendix.

	\section{Preliminaries}
	In order to have functions defined at all points in $\R^n$, we introduce a notion of density of a distribution. Let $\phi\in\ka{\infty}(\R^n)$ be a (test) radial function supported in the unit ball, whose integral is equal to $1$ and such that $\phi(0)=1$, Let us consider, for a distribution $T$, the limit 
	\begin{equation}\label{test}
		\lim_{\epsilon\to0} \langle T,\, \phi_{x_0,\, \epsilon}\rangle,
	\end{equation} 
	where $\phi_{x_0,\, \epsilon}(x)=\frac{1}{\epsilon^n}\, \phi(\frac{x-x_o}{\epsilon})$. 
	
	If this limit exists at some point $x_0$ and is independent of the choice of $\phi$ we call it {\it density function of $T$ at $x_0$} and denote it by $\Theta(T,x_0)$.
	
	
	\begin{lem}\label{dens}
		Let $\Omega\subset\R^n$ be a bounded domain such that $\partial\Omega\in\ka{1}$. If $T_{\chi_\Omega}$ is the distribution given by $\chi_\Omega$, we have 
		\begin{equation}\label{dns1}
			\Theta(T_{\chi_\Omega},x_0)=\begin{cases} 
				1  & \text{if }  x_0\in\Omega, \\ 
				\frac{1}{2} & \text{if }  x_0\in\partial\Omega, \\	
				0 &\text{if }  x_0\in\bar\Omega^c. 
			\end{cases}
		\end{equation}
	\end{lem}

	\begin{proof} The proof is after natural changes identical to the one in the case $n=2$ as we developed in \cite{BurMat} and we refer this paper to the interested reader.
	\end{proof}
	
	As a consequence of this fact, for $f\in\ka{\infty}(\R^n)$ we have the identity $$\Theta(T_{f\, \chi_\Omega},\, x)=f(x)\, \Theta(T_{\chi_\Omega},x).$$

	The space $\ka{k,\gamma}(U)$, where $U\subset\R^n$ is an open set, $k$ is a non-negative integer and $0<\gamma\le1$ is the space of functions with continuous derivatives up to the order $k$ such that each derivative of order $k$ extends to a $\gamma$-H\"older function in the closure of $U$.
	
	
	In this paper we will mainly use the spaces $\ka{k,\gamma}(U)$ for $k=0,1$, equipped with the norms 
	$$
	\|f\|_\gamma=\|f\|_\infty+\sup_{x\neq y;\, x,\, y\in U}\frac{|f(x)-f(y)|}{\|x-y\|^\gamma}
	$$ 
	and
	$$
	\|f\|_{1,\, \gamma}=\|f\|_\infty+\|\nabla f\|_\gamma.
	$$

	We use the notation $\ka{\omega}$ for the space of real analytic functions and~$\mathcal{D}(A)$ for compactly supported $\ka{\infty}$ functions whose support is contained in a closed set $A\subset\R^n$.

	\subsection{Differential forms} Differential forms represent a very convenient frame for our approach. Its use makes the cancellations more transparent and simple and the notation less heavy when dealing with problems concerning field theory. The standard vector operators can be written in terms of the exterior differential operator acting on several spaces of differential forms. In this context it is extremely useful and also a cornerstone idea the decomposition of a form as a sum of other three, one in the image of the exterior differential operator, the other in the image of its adjoint and a third one called harmonic annihilated by the exterior differential operator and by its adjoint.
	
	Let us provide a brief dictionary of the subject. For a general introduction and treatment of the subject see section 1.2 in \cite{Sch}. 
	
	By $x_j$ we denote the coordinate functions corresponding to the standard basis of $\R^n$, $e_1,\, \dots,\, e_n$. 
	
	The standard concept of the differential of a function  $$df=\sum_{j=1}^n\frac{\partial f}{\partial x_j}\, dx_j$$ combined with the standard antisymmetric exterior product, $\wedge$, produce the general differential forms:
	
	If $I=(j_j,\, \dots,\, j_k)$ is a $k$-index with $0<j_1<\dots<j_k<n$ then $$dx_I=dx_{j_1}\wedge\dots\wedge dx_{j_k}$$ is a way of producing the generators of $\Lambda^k\ka{l}(U)$, the space of forms $$\omega=\sum_I f_I\, dx_I$$ where $f_I\in\ka{l}(U)$ and $U\subset\R^n$ is an open set. It is worth remarking that $\Lambda^0\ka{l}(U)=\ka{l}(U)$ and $\Lambda^n\ka{l}(U)$ is generated by the volume form $$V=dx_1\wedge\dots\wedge dx_n$$ with coefficients in $\ka{l}(U)$.
	
	There is also a notion of duality among spaces of forms, called the Hodge star operator $$\star:\Lambda^k\ka{l}(U)\rightarrow\Lambda^{n-k}\ka{l}(U)$$ defined by the property $$(\star\omega)\wedge\eta=(\sum_I\omega_I\, \eta_I)\, V,$$ where $\omega=\sum_I\omega_I\, dx_I$ and $\eta=\sum_I\eta_I\, dx_I$. It satisfies the identity $$\star^2\omega=(-1)^{k\, (n-k)}\, \omega.$$

	The exterior differential operator acting on a $k$-form
	$$d:\Lambda^k\ka{l}(U)\rightarrow\Lambda^{k+1}\ka{l}(U)$$ is an extension of 
	the usual differential of a function  as defined above, by the formula $$d(\sum_I f_I\, dx_I)=\sum_I df_I\wedge dx_I.$$ The operator $d$ is an intrinsic object and satisfies  $d^2\stackrel{\text{def}}{=}d\circ d=0$.
	
	It has a formal adjoint relative to the euclidean metric of $\R^n$ acting on a $k$-form,
	$$\delta:\Lambda^k\ka{l}(U)\rightarrow\Lambda^{k-1}\ka{l}(U)$$ by the formula $$\delta\omega=(-1)^{nk+n+1}\, \star d \star\omega.$$
	
	It satisfies also $\delta\circ\delta=0$.
	
	There are two more facts that determine the importance of considering the point of view of differential exterior calculus as they capture faithfully many aspects of the euclidean geometry of $\R^n$. 
	
	On one hand there is the cornerstone of the Hodge theory: For a differential form $\omega\in \Lambda^k\ka{l}(U)$, the knowledge of $d\omega$, $\delta\omega$ and a harmonic form $u\in \Lambda^k\ka{l}(U)$ such that $du=0$ and $\delta u=0$ completely determine $u$. In fact $u$ codifies the "boundary behaviour" of $\omega$. This fact is of some remarkable use in systems of partial differential equations.  
	
	Aso  a partial description for the behaviour of  fields via useful identifications is possible. Namely if $$X=(X_1,\, \dots,\, X_n)$$ is a regular field, it induces two regular forms in a natural way. First $$\omega_X=\sum_{j=1}^n (-1)^{j-1}\, X_j\, \
	dx_1\wedge\dots\wedge dx_{j-1}\wedge dx_{j+1}\wedge\dots\wedge dx_n\stackrel{\text{def}}{=}\sum_{j=1}^n (-1)^{j-1}\, X_j\,  \widehat{dx_j}$$ and then $$d
	\omega_X=(\Div X)\, V.$$ Secondly 
	$$\omega^X=\star \omega_X=\sum_{j=1}^n X_j\, dx_j.$$ In dimension $3$ $$\star d\omega^X=\delta\omega_X=\omega^{\rot X}.$$
	
	Moreover if $f\in\ka{1}(U)$, then $$df=\omega^{\nabla f}.$$
	
	We finish this sketch with two more important operations.
	
	First the contraction of a form with one vector or field. If $e_p\in\R^n$ is a vector in the standard basis and $\omega=dx_{j_1}\wedge\dots\wedge dx_{j_k}$ then $$e_p\lrcorner\omega=\begin{cases}
		(-1)^{l-1}\, dx_{j_1}\wedge\dots\wedge dx_{j_{l-1}}\wedge dx_{j_{l+1}}\wedge\dots\wedge dx_{j_k}\ \  &\text{if}\ \  j_l=p\\
		
		\\
		
		0 \ \  &\text{otherwise}.
		
	\end{cases} $$ Then we extend by linearity the definition to general forms.
	
	The following formula is a particular case of current use in this paper
	$$e_j\lrcorner(e_l\lrcorner(df\wedge dg))=e_j\lrcorner(e_l\lrcorner((\sum_{u=1}^n\frac{\partial f}{\partial x_u}\, dx_u)\wedge (\sum_{v=1}^n\frac{\partial g}{\partial x_v}\, dx_v)))$$ 
	$$=e_j\lrcorner(e_l\lrcorner(\sum_{v=1}^n\sum_{u<v}[\frac{\partial f}{\partial x_u}\, \frac{\partial g}{\partial x_v}-\frac{\partial f}{\partial x_v}\, \frac{\partial g}{\partial x_u}]\, dx_u\wedge dx_v))$$
	$$=\frac{\partial f}{\partial x_l}\, \frac{\partial g}{\partial x_j}-\frac{\partial f}{\partial x_j}\, \frac{\partial g}{\partial x_l}.$$

	\medskip
	
	Finally, a key concept for this paper is the transport of differential forms by differentiable maps, namely if $W\subset\R^m$, open and $A:W\rightarrow U$ then for $\omega\in\Lambda^k\ka{l}(U)$ we define $A^*(\omega)\in\Lambda^{k}\ka{l}(W)$ in the following way: Assume that  $A=(A_1,\, \dots,\, A_n)$ , and $$\omega=\sum_{(j_1,\, \dots,\, j_k)}f_{(j_1,\, \dots,\, j_k)}\, dx_{j_1}\wedge\dots\wedge dx_{j_k}$$ then 
	$$A^*(\omega)=\sum_{(j_1,\, \dots,\, j_k)}f_{(j_1,\, \dots,\, j_k)}\circ A\  \ dA_{j_1}\wedge\dots\wedge dA_{j_k}.$$
	
	The map $A^*$ is linear and multiplicative and corresponds also to the transport of fields.
	
	\medskip 
	
	In general, it is a standard concept and notation the fact that for a form $\omega=\sum_I f_I\, dx_I$ we have an associated measure via the formula $$\int_A|\omega|=\sum_I\int_{A\cap<e_I>} |f|_I\, dm_{|I|},$$ where $dm_{|I|}$ is the volume element of the hyperplane generated by $e_{I}$. 
	
	Finally two observations useful in this paper. The first is that if $I$ is a $k$ multi-index with complementary multi-index $I'$,  then $$\star dx_I=\epsilon_{I,\, I'}^n\, dx_{I'}$$ where $\epsilon_{I,\, I'}^n$ is the sign associated to the number of transpositions necessary to locate $I$ into $I'$ to get the index $(1,\, 2,\, \dots,\, n)$. The second one, is that $f_1,\, \dots,\, f_n$ are functions, then $$df_1\wedge \dots\wedge df_n=\sum_{\sigma\in S_n}(-1)^{\epsilon(\sigma)}\, \frac{\partial f_1}{\partial x_{\sigma(1)}}\, \dots\frac{\partial f_n}{\partial x_{\sigma(n)}}\, V(x)$$ where $S_n$ is the group of permutations of order $n$ and $\epsilon(\sigma)$ is the number of transpositions leading $\sigma$ to the principal order $(1,\, 2,\, \dots,\, n)$.

	
	\subsection{The fundamental operators.}
	
	Other important actors in this play are the classical operators of potential theory.
	
	The {\bf Newtonian potentials} $K_j$, $j=1,\, \dots,\, n$ are the operators defined on the Schwartz space $\mathcal{S}(\R^n)$ by convolution with the kernels
	\begin{equation}\label{newton} K_j(x)=c_n\, \frac{x_j}{\|x\|^n}=\frac{\partial N}{\partial x_j}.
	\end{equation}\label{riesz} Here the constant $c_n$ depends only on the dimension and $N$ is the fundamental solution of the Laplace operator. 
	
	The {\bf Riesz operators} are the Calderon-Zygmund operators acting  on $\mathcal{S}(\R^n)$ by convolution with the kernel functions 
	\begin{equation}\label{riesz} R_{j,\, i}(x)\stackrel{\text{def}}{=}\frac{\partial K_j}{\partial x_i}(x)=c_n\, \frac{\|x\|^2\, \delta_{j,\, i}-n\, x_j\, x_i}{\|x\|^{n+2}}.
	\end{equation}
	For more information on Calderon-Zygmund operators see \cite{Duo}.
	
	\section{The patch problem for the aggregation equation.}
	
	This section is devoted to the proof of Theorem \ref{Main}. 
	
	As it has been shown in \cite{BeLaLe} we have that for a general $\rho_0\in L^1\cap L^\infty$ there exists a unique function $\rho:(\R^n\setminus\partial\Omega)\times [0,\, T)\rightarrow\R$ being a weak solution of (\ref{eq(A)}), $$\rho\in\ka{}([0,\, T^*),\, L^1(\R^n))\cap L^\infty(\R^n\times(0,\, T^*)),$$ for  some value $T^*\in(0,\, +\infty)$ related to $\rho_0$.

	In order to have a function defined at all points of $\R^n$ we will consider as an initial datum the function associated to the constant function $c$ by Lemma \ref{dens}\newline that is defined at every point in $\R^n$. Namely $$\rho_0(x)=c\, \biggl\{\chi_\Omega(x)+\frac{1}{2}\, \chi_{\partial\Omega}(x)\biggr\}$$
	
	The velocity field $$v(x,\, t)=[-\nabla\, N*\rho(\ ,\, t)](x)$$ is clearly  log-lipschitz in $\R^n$, for every $t\in[0,\, T^*)$ and then, as remarked in \cite{Che}, there exists a unique mapping $\psi\in\ka{}(\R^n\times[0,\, T^*);\R^n)$ satisfying equation (\ref{fla}) and there is a constant $C>0$ such that for any $t\in[0,\, T^*)$, 
	$$
	\psi(\ ,t)-I\in\ka{e^{-Ct\|\rho_0\|_{(L^p\cap L^\infty)(\R^n)}}}(\R^n).
	$$

	From the particular shape of $\rho_0$, we have that if $U\subset\Omega$ or $U$ is a bounded subset of $\R^n\setminus\bar\Omega$, then $v_0\in\ka{\infty}(U)$  and then, using for instance Proposition~8.3 in \cite{MaBe},  we can conclude that for any $t\in[0,\, T^*)$, 
	$
	\rho(\ ,t)\in\ka{\infty}(\psi(U,\, t))
	$
	and 
	$
	v(\ ,t)\in\ka{\infty}(\psi(U,\, t)).
	$ 
	
	Moreover, both the velocity and the {flow} given by the function $\psi$ in the theorem are globally defined with respect to the time variable and in general are regular beyond continuity in the $x$ variable .
	
	Another important fact related to $v$ is that since $v(\ ,\, t)$ is a convolution of a bounded measurable function with compact support with the kernel $-\nabla\, N$, then $v(\ ,\, t)$ is in $L^p(\R^n)$ for some $p$ and also 
	\begin{equation}\label{noharm}
		\lim_{\|x\|\to\infty} \|v(x,\, t)\|=0
	\end{equation} for every $t\in[0,\, T^*)$.

	The flow $\psi$ inherit the local space regularity (after derivation under the integral sign) too. Moreover, using 
	Theorem~1.3.1 of Chapter~1 in \cite{Hor}, we have that $\psi\in\ka{1}(U\times(0,\, T^*))$ and $\frac{\partial\psi}{\partial t}(\ ,t)\in\ka{1}(U)$.

	\subsection{The flow in terms of the density.} In order to avoid complexity and due to the transparency of concepts it provides, we reformulate the main objects and equations in terms of differential forms.
	
	Let us start defining 
	\begin{equation}\label{hyper}
		v(x,\, t)=\sum_{j=1}^n (-1)^{j-1}\, v_j(x,\, t)\, \widehat{dx_j}.
	\end{equation} 
	
	Then we have $$\star v(x,\, t)=\sum_{j=1}^n v_j(x,\, t)\, dx_j$$  
	satisfying  \begin{equation}\label{rot} d_x\star v=0
	\end{equation} and \begin{equation}\label{div} d_xv=\Div_x v\, V=-\rho\, V
	\end{equation}

	A first simple reason for this formulation is  
	
	\begin{prop}
		Given $\rho$ satisfying \eqref{noharm}, \eqref{rot} and \eqref{div}, then $v$ is completely determined in a unique way. 
	\end{prop}
	\begin{proof} The field $\nabla_x N*\rho$ implemented in (\ref{hyper}) provides a $(n-1)$-form satisfying the three conditions above. Moreover, any solution of \eqref{rot} has the form $\star v=d_x((N*\rho)+f)$, where $f\in\Lambda^1$ and $df=0$ therefore $f$ is a constant function. The condition \eqref{noharm} implies that $f\equiv0$.
	\end{proof}
	
	As a consequence also $\psi$ is completely determined by $\rho$ from \eqref{noharm}, \eqref{rot} and \eqref{div}.
	
	\medskip

	The main equation (\ref{eq(A)}) can be rephrased in terms of differential forms as $$\frac{\partial\rho}{\partial t}\, V+d(\rho\, v)=0,$$ and we follow the argument in \cite{BeLaLe}. The formula above is equivalent to $$\psi^*(\frac{\partial\rho}{\partial t}\, V+d\rho\wedge v)+\psi^*(\rho\, dv)=0$$ in the regions where $\psi(\ ,\, t)$ is an homeomorphism, and since $$\frac{\partial\rho}{\partial t}\, V+d\rho\wedge v=(\frac{\partial\rho}{\partial t}+\sum_{j=1}^n v_j\, \frac{\partial\rho}{\partial x_j})\, V,$$ we have $$\psi^*(\frac{\partial\rho}{\partial t}\, V+d\rho\wedge v)=(\frac{\partial\rho}{\partial t}\circ\psi+\sum_{j=1}^n v_j\circ\psi\, \frac{\partial\rho}{\partial x_j}\circ\psi)\, \psi^*(V)$$
	$$=(\frac{\partial\rho}{\partial t}\circ\psi+\sum_{j=1}^n \frac{\partial\psi_j}{\partial t}\, \frac{\partial\rho}{\partial x_j}\circ\psi)\, \psi^*(V)
	=\frac{\partial(\rho\circ\psi)}{\partial t}\, \psi^*(V).$$
	
	Since $dv=-\rho\, V$ we finally have $$\frac{\partial(\rho\circ\psi)}{\partial t}\, \psi^*(V)=\psi^*(\rho^2\, V)=\rho^2\circ\psi\ \psi^*(V),$$ and then $$\rho(\psi(x,\, t),\, t)=\frac{\rho_0(x)}{1-t\, \rho_0(x)}$$ as $\psi^*(V)$ never vanishes.

	Now, the equation relating $\psi$ with $\rho_0$. 
	
	\begin{prop}\label{main} If the choice of $\rho_0$ implies that $x\rightarrow \psi(x,\, t)$ is a diffeomorphism for any fixed $t$, then we have that 
		\begin{equation}\label{div1} 
			\sum_{j=1}^n (-1)^{j-1}\, d(\frac{\partial\psi_j}{\partial t}(\ ,\, t))(x)\wedge\widehat{d\psi_j(x,\, t)}=-\rho_0(x)\, V(x)
		\end{equation} and
		\begin{equation}\label{rot1}
			\sum_{j=1}^n d(\frac{\partial\psi_j}{\partial t}(\ ,\, t))(x)\wedge d\psi_j(\ ,\, t)(x)=0
		\end{equation}
		
	\end{prop}
	
	\begin{rmk}
		The role of the flow $\psi$ can be considered as the action of a family of diffeomorphisms on the forms describing the velocity and the density. This allows us to formulate equations directly relating the flow and the density.
	\end{rmk}	
	\begin{proof} 
		
		We have $$\psi(\ ,\, t)^*(v(\ ,\, t))=\sum_{j=1}^n (-1)^{j-1}\, v_j(\psi(\ ,\, t),\, t)\, \widehat{d\psi_j(\ ,\, t)}.$$ 
		
		Then $$d_x\psi(\ ,\, t)^*(v(\ ,\, t))=\psi(\ ,\, t)^*(dv(\ ,\, t))=\psi(\ ,\, t)^*(-\rho(\ ,\, t) \, V)$$
		$$=-\rho(\psi(\ ,\, t),\, t) \, \psi(\ ,\, t)^*(V)(x)=-\rho(\psi(\ ,\, t),\, t) \, \det J_x\psi(\ ,\, t)\, V(x),$$ and also $$d_x\left(\sum_{j=1}^n (-1)^{j-1}\, v_j(\psi(\ ,\, t),\, t)\, \widehat{d\psi_j(\ ,\, t)}\right)$$
		$$=\sum_{j=1}^n (-1)^{j-1}\,  \left(d_x\frac{\partial\psi_j}{\partial t}(x,\, t)\wedge\widehat{d\psi_j(x,\, t)}+\frac{\partial\psi_j}{\partial t}(x,\, t)\, d_x(\widehat{d\psi_j(x,\, t)})\right)$$
		$$=\sum_{j=1}^n (-1)^{j-1}\,  d_x\frac{\partial\psi_j}{\partial t}(x,\, t)\wedge\widehat{d\psi_j(x,\, t)}=(*),$$ because $$d_x(\widehat{d\psi_j(x,\, t)})=d_x\, \psi(\ ,\, t)^*(\widehat{dx_j})=\psi(\ ,\, t)^*(d_x \widehat{dx_j})=0.$$ 
		
		Then $$(*)=\sum_{j=1}^n d\psi_1(x,\, t)\wedge\dots\wedge d\psi_{j-1}(x,\, t)\wedge  d_x\frac{\partial\psi_j}{\partial t}(x,\, t)\wedge d\psi_{j+1}(x,\, t)\wedge\dots\wedge d\psi_n(x,\, t)$$
		$$=\frac{\partial\det J_x\psi}{\partial t}(x,\, t)\, V(x),$$ implying that  $$\frac{\partial\det J_x\psi}{\partial t}(x,\, t)=-\rho(\psi(\ ,\, t),\, t) \, \det J_x\psi(\ ,\, t)$$ and $$\det J_x\psi(x ,\, t)=\det J_x\psi(x,\, 0)\, e^{-\int_0^t \rho(\psi(x ,\, \tau),\, \tau)\, d\tau },$$ therefore	
		\begin{equation}\label{gral1} 
			d_x\psi(\ ,\, t)^*(v(\ ,\, t))=-\rho(\psi(\ ,\, t),\, t) \, \det J_x\psi(x,\, 0)\, e^{-\int_0^t \rho(\psi(x ,\, \tau),\, \tau)\, d\tau }\, V(x)
		\end{equation}
		
		\medskip
		and as we remarked above    $$\rho(\psi(\ ,\, t),\, t)=\frac{\rho_0(x)}{1-t\, \rho_0(x)}.$$ Consequently $$\int_0^t \rho(\psi(x ,\, \tau),\, \tau)\, d\tau=\int_0^t \frac{\rho_0(x)}{1-\tau\, \rho_0(x)}\, d\tau$$
		$$=-\ln(1-\tau\, \rho_0(x))]_0^t=-\ln(1-t\, \rho_0(x)),$$ and as $\psi(x,\, 0)=x$, we have (\ref{div1}).

		We also have $$d\psi(\ ,\, t)^*(\star v(\ ,\, t))=d\left(\sum_{j=1}^n v_j(\psi(\ ,\, t),\, t)\, d\psi_j(\ ,\, t)\right)$$
		$$=d\left(\sum_{j=1}^n \frac{\partial\psi_j}{\partial t}(\ ,\, t),\, t)\, d\psi_j(\ ,\, t)\right).$$ This implies (\ref{rot1}).
	\end{proof}

	\subsection{Analyticity of the flow: Proof of Theorem \ref{Main}. } The goal of this section is the proof of the analyticity in time of the flow $\psi$ corresponding to the weak   solution of (\ref{eq(A)}), for $t\in[0,\, T^*]$. 
	
	As we have seen in the previous sub-section the aforementioned flow $\psi$ satisfies in the case of $\rho_0$ being a patch the system 
	\begin{equation}\label{eqP}
		\begin{cases}
			\sum_{j=1}^n (-1)^{j-1}\, d(\frac{\partial\psi_j}{\partial t}(\ ,\, t))(x)\wedge\widehat{d\psi_j(x,\, t)}=-\rho_0(x)\, V(x)\\
			
			\\
			
			\sum_{j=1}^n d(\frac{\partial\psi_j}{\partial t}(\ ,\, t))(x)\wedge d\psi_j(\ ,\, t)(x)=0\\
			
			\\
			
			\psi(x,0)=x
		\end{cases}
	\end{equation} 
	
	Also this system has a solution $\phi$ that is time-analytic in an interval $[0,\, T_1]$ as it is stated in the next result that will be proven in section 4.

	\begin{thm}\label{T1} For any given function $\rho_0$ 
		the system (\ref{eqP})
		has a solution $\phi(x,\, t)$ analytic in $t$ in a neighbourhood of $0$ whose radius depend on $\rho_0$. Moreover for every $t$ in $[0,\, T_0]$, for some $T_0\le\min\{T_1,\, T^*\}$
		\begin{enumerate}
			\item[1)] $\phi(\ ,\, t):\R^n\rightarrow\R^n$ is an homeomorphism.
			
			
			\item[2)] $\phi(\ ,\, t)\in\ka{\infty}(\R^n\setminus\partial\Omega).$
			
		\end{enumerate}
		
	\end{thm}

	It is worth  noticing that the system (\ref{eqP}) has not a unique solution but, by construction, it has a unique analytic solution. Our goal is to prove that $\phi$ coincides in this interval with the unique flow given by the solution of equation (\ref{eq(A)}) .
	
	As we have seen above, the velocity field associated to $\psi$, $$v(y,\, t)=(\psi(\ ,\, t)^{-1})^*(\sum_{j=1}^n(-1)^{j-1}\frac{\partial\psi}{\partial t}(\ ,\, t)\, \widehat{d\psi_j(\ ,\, t)})(y)$$
	satisfies that $$d_y v(y,\, t)=\rho(y,\, t)\, V=\psi(\ \, t)^*(\rho_0\, V)(y).$$
	
	Analogously, let us define $$u(y,\, t)=(\phi(\ ,\, t)^{-1})^*(\sum_{j=1}^n(-1)^{j-1}\frac{\partial\phi}{\partial t}(\ ,\, t)\, \widehat{d\phi_j(\ ,\, t)})(y).$$ Then $$d_y u(y,\, t)=(\phi(\ ,\, t)^{-1})^*(d_x(\sum_{j=1}^n(-1)^{j-1}\frac{\partial\phi}{\partial t}(\ ,\, t)\, \widehat{d\phi_j(\ ,\, t)}))(y)$$
	$$=(\phi(\ ,\, t)^{-1})^*(\rho_0\, V)(y)\stackrel{\text{def}}{=}\Xi(y,\, t),$$ and 
	
	\begin{prop} The $n$-form $\Xi$ provides a weak solution of the aggregation equation with initial datum $\rho_0\, V$.  
	\end{prop}
	\begin{proof} If $\lambda\in\ka{1}([0,\, T];\, \ka{1}_0(\R^n))$, we have that 
		$$\int_{\R^n} \lambda(y,\, T)\, \Xi(y,\, T)-\int_{\R^n} \lambda(x,\, 0)\, \rho_0(x)\, V$$
		$$=\int_{\R^n} \lambda(y,\, T)\, (\phi(\ ,\, T)^{-1})^*(\rho_0\, V)(y)-\int_{\R^n} \lambda(x,\, 0)\, (\phi(\ ,\, 0)^{-1})^*(\rho_0\, V)(x)$$
		$$=\int_{\R^n} \lambda(\phi(x,\, T),\, T)\, \rho_0(x)\, V-\int_{\R^n} \lambda(\phi(x,\, 0),\, 0)\, \rho_0(x)\, V$$
		$$=\int_{\R^n} \{\lambda(\phi(x,\, T),\, T)-\lambda(\phi(x,\, 0),\, 0)\}\, \rho_0(x)\, V$$
		$$=\int_{\R^n}\, (\int_0^T \frac{d\ }{dt}\lambda(\phi(x,\, t),\, t)\, \, dt)\,  \rho_0(x)\, V$$
		$$=\int_{\R^n}\, (\int_0^T \{\frac{\partial\lambda }{\partial t}(\phi(x,\, t),\, t)\, \, dt)+\sum_{j=1}^n\, \frac{\partial\lambda}{\partial y_j}(\phi(x,\, t),\, t)\, \frac{\partial\phi_j}{\partial t}(x,\, t)\}\, \, dt)\,  \rho_0(x)\, V$$
		$$=\int_{\R^n}\, (\int_0^T \{\frac{\partial\lambda }{\partial t}(\phi(x,\, t),\, t)\, \, dt)+\sum_{j=1}^n\, \frac{\partial\lambda}{\partial y_j}(\phi(x,\, t),\, t)\, u_j(\phi(x,\, t),\, t)\}\, \, dt)\,  \rho_0(x)\, V$$
		$$=\int_0^T\, dt\, \int_{\R^n}\{\frac{\partial\lambda }{\partial t}+\sum_{j=1}^n\, \frac{\partial\lambda}{\partial y_j}\, u_j\}(\phi(x,\, t),\, t)\,  \rho_0(x)\, V$$
		$$=\int_0^T\, dt\, \int_{\R^n}\{\frac{\partial\lambda }{\partial t}+\sum_{j=1}^n\, \frac{\partial\lambda}{\partial y_j}\, u_j\}(y,\, t)\, (\phi(\ ,\, t)^{-1})^*( \rho_0(x)\, V)(y)$$

	\end{proof}
	
	\begin{cor} If $\rho_0$ corresponds to a patch, then
		$$\rho\, V=\Xi$$ almost everywhere.
	\end{cor}
	\begin{proof}
		As been proven in \cite{BeLaLe}, the aggregation equation having a patch as initial datum has a unique solution in the weak sense, so $\rho\, V=\Xi$ in the weak sense. Then $\rho\, V=\Xi$ almost everywhere.  
	\end{proof}
	
	To finish the proof of Theorem 1  we consider $w=u-v$ and then $$d_yw(y,\, t)=d_y(u-v)(y,\, t)=(\rho\, V-\Xi)(y,\, t)=0$$ and also, from the second equation in (\ref{eqP}) $$(d\star w)(y,\, t)=0.$$ This implies that $$\begin{cases} dw=0\\
		
		\delta w=0
		
	\end{cases}$$
	
	Moreover, $$\lim_{\|x\|\to\infty} v(x,\, t)=0$$ by definition and $$\lim_{\|x\|\to\infty} u(x,\, t)=0$$ by construction (the previous theorem). So $u=v$ and then $\phi=\psi$. 
	
	\subsection{Global analyticity in time.} We have that there exists a number $T_1>0$ and a function $\psi\colon\R^n\rightarrow\R^n$, such that $t\rightarrow\psi(z,\, t)$ is real analytic in $(0,\, T_1)$.
	
	
	If $T_1=T^*$ there is nothing to say. Otherwise, we take in account that the rescaling 
	$$
	s(t)=\ln\biggl(\frac{1}{1-ct}\biggr)
	$$ 
	and 
	$$
	\varpi(z,\, s)=\frac{1-c\, t(s)}{c}\, \rho(z,\, t(s))
	$$ 
	transforms the problem (\ref{eq(A)}) with initial condition $\rho_0(x)=c\, \chi_\Omega(x)$ in the problem (\ref{eqtildeA}) with initial condition $\chi_\Omega(x)$.

	For the problem $(\tilde A)$ it is proven in \cite{BeGaLaVe} that for every 
	$s\!\in\![0,\infty)$, if $\tilde\psi$ is the related flow, then $\tilde\psi(\partial\Omega,\, s)$ is a $\ka{1,\gamma}$ embedded submanifold of $\R^n$ of real dimension $n-1$. Since the rescalings above do not affect the $x$ variable, the same regularity is true in the case of $(\tilde{A})$, for $t\in[0,T^*)$.
	
	Then the datum $$
	\rho(x,\, T_1)=\frac{c}{(1-c\, T_1)^3}\, \rho_0(x),
	$$
	is used as initial density to iterate the procedure because the boundary $\partial\Omega_{T_1}$ is $\ka{1,\, \gamma}$, providing a new $T_2>T_1$ and we have analyticity in~$(0,\, T_2)$.
	
	An argument of connectivity and the uniqueness of solution of $(A)$ conclude that the flow is analytic in $[0,\, T^*)$.    
	
	
	It is worth noticing that, by duality, we also have analyticity for the case of the corresponding flow to the equation $(\tilde{A})$.  
	

	\section{Construction of an analytic solution for the flow-density equation (\ref{eqP}).} 
	
	We use the relationships obtained in proposition \ref{main} for constructing an analytic solution of the system (\ref{eqP}) by an "a priori" procedure. We assume that we have a solution analytic in the variable $t$. Then we use proposition \ref{main} to estimate the size of the coefficients of the corresponding power series in terms of $\rho_0$. These estimates are the key point for the analyticity of the components of the flow. 
	
	More precisely if we assume that for $j=1,\dots,n$, we have  $$\psi_j(x,\, t)=\sum_{s=0}^\infty\xi_j^{(s)}(x)\, t^s=x_j+\sum_{s=1}^\infty\xi_j^{(s)}(x)\, t^s.$$

	\subsection{A priori relationships} Using formula $(10)$ we obtain $$0=\sum_{j=1}^n d\frac{\partial\psi_j}{\partial t}\wedge d\psi_j=\sum_{s=0}^\infty t^s\, \sum_{p=0}^s (p+1)\, \sum_{j=1}^n d\xi_j^{(p+1)}\wedge d\xi_j^{(s-p)}.$$
	
	This implies that $$0=\sum_{j=1}^n d\xi_j^{(1)}\wedge d\xi_j^{(0)}=\sum_{j=1}^n d\xi_j^{(1)}\wedge dx_j$$ and $$(s+1)\, \sum_{j=1}^n d\xi_j^{(s+1)}\wedge dx_j=-\sum_{p=0}^{s-1} (p+1)\, \sum_{j=1}^n d\xi_j^{(p+1)}\wedge d\xi_j^{(s-p)}.$$
	
	Using formula $(9)$ we also have that $$-\rho_0\ V=\sum_{j=1}^n d\psi_1\wedge\dots\wedge d\psi_{j-1}\wedge  d\frac{\partial\psi_j}{\partial t}\wedge d\psi_{j+1}\wedge\dots\wedge d\psi_n$$
	$$=\sum_{j=1}^n(\sum_{s_1=0}^\infty t^{s_1}\, d\xi_1^{(s_1)})\wedge\dots\wedge(\sum_{s_{j-1}=0}^\infty t^{s_{j-1}}\, d\xi_{j-1}^{(s_{j-1})})\wedge(\sum_{s_j=0}^\infty (s_j+1)\, t^{s_j}\, d\xi_j^{(s_j+1)})$$
	$$\wedge(\sum_{s_{j+1}=0}^\infty t^{s_{j+1}}\, d\xi_{j+1}^{(s_{j+1})})\wedge\dots\wedge(\sum_{s_n=0}^\infty t^{s_n}\, d\xi_n^{(s_n)})$$
	$$=\sum_{s=0}^\infty t^s\, \sum_{s_1+\dots+s_n=s} \sum_{j=1}^n(s_j+1)\, d\xi_1^{(s_1)}\wedge\dots\wedge d\xi_{j-1}^{(s_{j-1})}\wedge d\xi_j^{(s_j+1)}\wedge d\xi_{j+1}^{(s_{j+1})}\wedge\dots\wedge d\xi_n^{(s_n)}.$$
	
	This implies that $$-\rho_0\, V=\sum_{j=1}^n d\xi_1^{(0)}\wedge\dots\wedge d\xi_{j-1}^{(0)}\wedge d\xi_j^{(1)}\wedge d\xi_{j+1}^{(0)}\wedge\dots\wedge d\xi_n^{(0)}$$
	$$=\sum_{j=1}^n dx_1\wedge\dots\wedge dx_{j-1}\wedge d\xi_j^{(1)}\wedge dx_{j+1}\wedge\dots\wedge dx_n,$$ and also that for any $s>0$, 
	$$(s+1)\, \sum_{j=1}^n d\xi_1^{(0)}\wedge\dots\wedge d\xi_{j-1}^{(0)}\wedge d\xi_j^{(s+1)}\wedge d\xi_{j+1}^{(0)}\wedge\dots\wedge d\xi_n^{(0)}$$
	$$=(s+1)\, \sum_{j=1}^n dx_1\wedge\dots\wedge dx_{j-1}\wedge d\xi_j^{(s+1)}\wedge dx_{j+1}\wedge\dots\wedge dx_n$$
	$$=-\sum_{j=1}^n\ \  \sum_{s_1+\dots+s_n=s-1} (s_j+1)\, d\xi_1^{(s_1)}\wedge\dots\wedge d\xi_{j-1}^{(s_{j-1})}\wedge d\xi_j^{(s_j+1)}\wedge d\xi_{j+1}^{(s_{j+1})}\wedge\dots\wedge d\xi_n^{(s_n)}.$$
	
	And then the coefficient functions of $\psi$ and the density are related by the formulas
	
	\begin{prop} \begin{enumerate}
			\item $\sum_{j=1}^n (-1)^{j} d\xi_j^{(1)}\wedge\widehat{dx_{j}}=\rho_0\, V$
			
			\item $\sum_{j=1}^n d\xi_j^{(1)}\wedge dx_j=0$
			
			\item If $s>0$, then $$\sum_{j=1}^n (-1)^{j}
			d\xi_j^{(s+1)}\wedge\widehat{dx_{j}}$$
			$$=\frac{1}{s+1} 
			\sum_{j=1}^n\ \  \sum_{s_1+\dots+s_n=s-1} (s_j+1)\, d\xi_1^{(s_1)}\wedge\dots\wedge d\xi_{j-1}^{(s_{j-1})}\wedge d\xi_j^{(s_j+1)}\wedge d\xi_{j+1}^{(s_{j+1})}\wedge\dots\wedge d\xi_n^{(s_n)}$$
			
			\item If $s>0$, then $$\sum_{j=1}^n d\xi_j^{(s+1)}\wedge dx_j=\frac{-1}{s+1}\sum_{p=0}^{s-1} (p+1)\, \sum_{j=1}^n d\xi_j^{(p+1)}\wedge d\xi_j^{(s-p)}.$$
		\end{enumerate}
		
	\end{prop}
	
	\begin{cor} Considering $\mu^{(s)}=\sum_{j=1}^n\xi_j^{(s)}\ dx_j$, we have 
		
		\begin{enumerate}
			\item $\mu^{(0)}=d(\frac{\|x\|^2}{2})$.
			
			\item $\begin{cases} d\mu^{(1)}=0\\
				
				d\star\mu^{(1)}=-\varrho_0\, V
				
			\end{cases}$
			
			\item If $s>0$, then 
			
			$\begin{cases} d\mu^{(s+1)}=\frac{-1}{s+1}\sum_{p=0}^{s-1} (p+1)\, \sum_{j=1}^n d\xi_j^{(p+1)}\wedge d\xi_j^{(s-p)}\\
				
				\\
				
				d\star\mu^{(s+1)}=\frac{-1}{s+1} 
				\sum_{j=1}^n\ \  \sum_{s_1+\dots+s_n=s-1} (s_j+1)\, d\xi_1^{(s_1)}\wedge\\
				
				\\
				
				\ \ \ \dots\wedge d\xi_{j-1}^{(s_{j-1})}\wedge d\xi_j^{(s_j+1)}\wedge d\xi_{j+1}^{(s_{j+1})}\wedge\dots\wedge d\xi_n^{(s_n)}.
				
			\end{cases}$
		\end{enumerate}
		
	\end{cor}

	This corollary implies that $$\delta\mu^{(1)}=-\star d\star\mu^{(1)}=\varrho_0\, \star V$$ and $$\delta\mu^{(s+1)}=(-1)^{n^2+1}\star d\star\mu^{(s+1)}$$
	$$=\star \left(\frac{(-1)^{n^2}}{s+1} 
	\sum_{j=1}^n\ \  \sum_{s_1+\dots+s_n=s-1} (s_j+1)\, d\xi_1^{(s_1)}\wedge\\
	\ \ \ \dots\wedge d\xi_{j-1}^{(s_{j-1})}\wedge d\xi_j^{(s_j+1)}\wedge d\xi_{j+1}^{(s_{j+1})}\wedge\dots\wedge d\xi_n^{(s_n)}\right).$$
	
	And since $\xi_j^{(s)}=e_j\lrcorner\mu^{(s)}$, defining  $$\Upsilon^{(s+1)}[\mu^{(1)},\, \dots,\, \mu^{(s)}]=\frac{-1}{s+1}\sum_{p=0}^{s-1} (p+1)\, \sum_{j=1}^n d\xi_j^{(p+1)}\wedge d\xi_j^{(s-p)}$$ and $$\Xi^{(s+1)}[\mu^{(1)},\, \dots,\, \mu^{(s)}]$$ 
	$$=\star \left(\frac{(-1)^{n^2}}{s+1} 
	\sum_{j=1}^n\ \  \sum_{s_1+\dots+s_n=s-1} (s_j+1)\, d\xi_1^{(s_1)}\wedge\\
	\ \ \ \dots\wedge d\xi_{j-1}^{(s_{j-1})}\wedge d\xi_j^{(s_j+1)}\wedge d\xi_{j+1}^{(s_{j+1})}\wedge\dots\wedge d\xi_n^{(s_n)}\right),$$ we finally have
	
	\begin{cor}\label{c3} With the previous notations
		
		\begin{enumerate}
			\item $\mu^{(0)}=d(\frac{\|x\|^2}{2})$.
			
			\item $\begin{cases} d\mu^{(1)}=0\\
				
				\delta\mu^{(1)}=\varrho_0\ \star V
				
			\end{cases}$
			
			\item If $s>0$, then 
			
			$\begin{cases} d\mu^{(s+1)}=\Upsilon^{(s+1)}[\mu^{(1)},\, \dots,\, \mu^{(s)}]\\
				
				\\
				
				\delta\mu^{(s+1)}=\Xi^{(s+1)}[\mu^{(1)},\, \dots,\, \mu^{(s)}].
				
			\end{cases}$
		\end{enumerate}
		
	\end{cor}

	Our next task consists in determining the $1$-forms $\mu^{(s)}$. 
	Let us assume that we know $\mu^{(l)}$ for $l\le s$ and then we take advantage of the fact that the knowledge of $d\mu^{(l)}$ and $\delta\mu^{(l)}$ determine the diagonal Laplacian operator $\triangle\mu^{(s+1)}$ in terms $\mu^{(l)}$ for all $l\le s$ and we can use the Newtonian potentials to recover $\mu^{(s+1)}$. Then we use the  following fact:
	
	\begin{prop} If $\mu$ is a $1$-form in $\R^n$, such that $$ \begin{cases} d\mu=\Upsilon\\

			\delta\mu=\Xi
		\end{cases}		
		$$ 
		then \begin{equation}\label{newt} \mu=\sum_{k=1}^n (\sum_{j<k} K_j[\Upsilon_{j,\, k}]-\sum_{j>k} K_j[\Upsilon_{k,\, j}]+K_k[\Xi])\, dx_k.\end{equation} 
	\end{prop}
	
	\begin{proof} We have that $\delta d\mu=\delta\Upsilon$ and $d\delta\mu=d\Xi$, so $$\triangle\mu=(\delta d+d\delta)\mu=\delta\Upsilon+d\Xi.$$ The operator $\triangle$ acts diagonally on forms and since in the scalar case $\triangle N\varphi=N\triangle\varphi=\varphi,$ for $\varphi\in\mathcal{D}(\R^n)$, we extend the definition of the operator $N$ to $k$-forms.  
		Then, under good conditions on $\Upsilon$ and $\Xi$, we can get a solution for $\mu$: $$\mu=N[\delta\Upsilon]+N[d\Xi]=\sum_{k=1}^n (N[(\delta\Upsilon)_k]+N[(d\Xi)_k])\, dx_k.$$ 
		
		If $\omega=\sum_{j<k}\omega_{j,\, k}\, dx_j\wedge dx_k,$ we have  $$\delta\omega=\sum_{k=1}^n(\sum_{j<k} \frac{\partial \omega_{j,\, k}}{\partial x_j}-\sum_{j>k} \frac{\partial \omega_{k,\, j}}{\partial x_j})\, dx_k.$$ Thereof $$\mu=\sum_{k=1}^n (\sum_{j<k} N[\frac{\partial \Upsilon_{j,\, k}}{\partial x_j}]-\sum_{j>k} N[\frac{\partial \Upsilon_{k,\, j}}{\partial x_j}]+N[\frac{\partial \Xi}{\partial x_k}])\, dx_k$$
		$$=\sum_{k=1}^n (\sum_{j<k} K_j[\Upsilon_{j,\, k}]-\sum_{j>k} K_j[\Upsilon_{k,\, j}]+K_k[\Xi])\, dx_k$$

	\end{proof}
	
	Now applying the identity (\ref{newt}) to $\mu^{(s)}$ in corollary \ref{c3} we have that $$\mu^{(1)}=\sum_j K_j[\rho_0]\, dx_j,$$ so $\xi^{(1)}_j=K_j[\rho_0]$ and from the definition of Riesz operators in (\ref{riesz})
	$$d\xi^{(1)}_j=dK_j[\rho_0]=\sum_l(\frac{\partial\ }{\partial x_l}\circ K_j)[\rho]\, dx_l=\sum_l R_{j,\, l}[\rho_0]\, dx_l$$  
	and also for $s>0$, $$\mu^{(s+1)}=\sum_{j=1}^n (\sum_{l<j} K_l[\Upsilon^{(s+1)}_{l,j}]-\sum_{l>j} K_l[\Upsilon^{(s+1)}_{j,l}]+K_j[\Xi^{(s+1)}])\, dx_j$$ and $$\xi_j^{(s+1)}=\sum_{l<j} K_l[\Upsilon^{(s+1)}_{l,j}]-\sum_{l>j} K_l[\Upsilon^{(s+1)}_{j,l}]+K_j[\Xi^{(s+1)}].$$ Therefore proceeding as above we have $$d\xi_j^{(s+1)}=\sum_i(\sum_{l<j} R_{l,\, i}[\Upsilon^{(s+1)}_{l,\, j}]-\sum_{l>j} R_{l,\, i}[\Upsilon^{(s+1)}_{j,\, l}]+R_{j,\, i}[\Xi^{(s+1)}])\, dx_i$$	
	$$=\sum_i(\{\sum_{l<j} \frac{-1}{s+1}\sum_{p=0}^{s-1} (p+1)\, \sum_{r=1}^n R_{l,\, i}[e_j\lrcorner(e_l\lrcorner(d\xi_r^{(p+1)}\wedge d\xi_r^{(s-p)}))]$$
	$$-\sum_{l>j} \frac{-1}{s+1}\sum_{p=0}^{s-1} (p+1)\, \sum_{r=1}^n R_{l,\, i}[e_l\lrcorner(e_j\lrcorner(d\xi_r^{(p+1)}\wedge d\xi_r^{(s-p)}))]$$	
	$$+\frac{(-1)^{n^2}}{s+1} 
	\sum_{k=1}^n\ \  \sum_{s_1+\dots+s_n=s-1} (s_k+1)$$
	$$ R_{k,\, i}[\star\, (d\xi_1^{(s_1)}\wedge\\
	\ \ \ \dots\wedge d\xi_{k-1}^{(s_{k-1})}\wedge d\xi_k^{(s_k+1)}\wedge d\xi_{k+1}^{(s_{k+1})}\wedge\dots\wedge d\xi_n^{(s_n)}) ]\}\, dx_i.$$

	Contracting with $e_i$ in both sides, we have that $$\frac{\partial\xi_j^{(s+1)}}{\partial x_i}=\sum_{l<j} \frac{-1}{s+1}\sum_{p=0}^{s-1} (p+1)\, \sum_{r=1}^n R_{l,\, i}[e_j\lrcorner(e_l\lrcorner(d\xi_r^{(p+1)}\wedge d\xi_r^{(s-p)}))]$$
	$$-\sum_{l>j} \frac{-1}{s+1}\sum_{p=0}^{s-1} (p+1)\, \sum_{r=1}^n R_{l,\, i}[e_l\lrcorner(e_j\lrcorner(d\xi_r^{(p+1)}\wedge d\xi_r^{(s-p)}))]$$
	$$+\frac{(-1)^{n^2}}{s+1} 
	\sum_{k=1}^n\ \  \sum_{s_1+\dots+s_n=s-1} (s_k+1)$$
	$$ R_{j,\, i}[\star\, (d\xi_1^{(s_1)}\wedge\\
	\ \ \ \dots\wedge d\xi_{k-1}^{(s_{k-1})}\wedge d\xi_k^{(s_k+1)}\wedge d\xi_{k+1}^{(s_{k+1})}\wedge\dots\wedge d\xi_n^{(s_n)})],$$ and then	
	$$\frac{\partial\xi_j^{(s+1)}}{\partial x_i}=\sum_{l<j} \frac{-1}{s+1}\sum_{p=0}^{s-1} (p+1)\, \sum_{r=1}^n R_{l,\, i}[\frac{\partial\xi_r^{(p+1)}}{\partial x_l}\, \frac{\partial\xi_r^{(s-p)}}{\partial x_j}-\frac{\partial\xi_r^{(p+1)}}{\partial x_j}\, \frac{\partial\xi_r^{(s-p)}}{\partial x_l}]$$
	$$-\sum_{l>j} \frac{-1}{s+1}\sum_{p=0}^{s-1} (p+1)\, \sum_{r=1}^n R_{l,\, i}[\frac{\partial\xi_r^{(p+1)}}{\partial x_j}\, \frac{\partial\xi_r^{(s-p)}}{\partial x_l}-\frac{\partial\xi_r^{(p+1)}}{\partial x_l}\, \frac{\partial\xi_r^{(s-p)}}{\partial x_j}]$$	
	$$+\frac{(-1)^{n^2}}{s+1} 
	\sum_{k=1}^n\ \  \sum_{s_1+\dots+s_n=s-1} (s_k+1)\, \sum_{\sigma\in S_n}(-1)^{\epsilon(\sigma)}$$
	$$ R_{j,\, i}[\frac{\partial\xi_1^{(s_1)}}{\partial x_{\sigma(1)}}\, \dots\, 
	\frac{\partial\xi_{k-1}^{(s_{k-1})}}{\partial x_{\sigma(k-1)}}\, \frac{\partial\xi_k^{(s_k+1)}}{\partial x_{\sigma(k)}}\, \frac{\partial\xi_{k+1}^{(s_{k+1})}}{\partial x_{\sigma(k+1)}}\,\dots\, \frac{\partial\xi_n^{(s_n)}}{\partial x_{\sigma(n)}}].$$

	\subsubsection{Extension of the recurrence to the boundary.} Now we consider the densities of the derivatives of the functions $\xi_j^{(s)}$. Since the operator $\Theta$ defined in Lemma \ref{dens} is linear, we have $$\Theta(\frac{\partial\xi^{(s)}_j}{\partial x_l})=\theta^{(s)}_{j,\, l}=\chi_\Omega\, \theta_{j,\, l}^{(s)}+\chi_{\R^n\setminus\bar\Omega}\, \theta_{j,\, l}^{(s)}+\chi_{\partial\Omega}\, \theta_{j,\, l}^{(s)}=\phi_{j,\, l}^{(s)}+\psi_{j,\, l}^{(s)}+\i_{j,\, l}^{(s)}.$$  
	
	Then using jump formula for the Riesz transform (see Appendix, Theorem \ref{jump}) and proceeding as in \cite{BurMat} we have $$\theta^{(1)}_{j,\, l}=\Theta(R_{j,\, l}[\rho])=\phi_{j,\, l}^{(1)}+\psi_{j,\, l}^{(1)}+\i_{j,\, l}^{(1)}.$$ And for $s>0$ we have $$\theta^{(s+1)}_{j,\, i}=\phi_{j,\, i}^{(s+1)}+\psi_{j,\, i}^{(s+1)}+\i_{j,\, i}^{(s+1)},$$ and defining $\Phi_{l,\, i}[\ \ ]=\chi_\Omega\, R_{l,\, i}[\ \ ]$, $\Psi_{l,\, i}[\ \ ]=\chi_{\R^n\setminus\bar\Omega}\, R_{l,\, i}[\ \ ]$ and $\Gamma_{l,\, i}[\ \ ]=\chi_{\partial\Omega}\, \Theta(R_{l,\, i}[\ \ ],\, )$ and taking in account the cancellations due to the supports we have
	$$\phi_{j,\, i}^{(s+1)}=\sum_{l\neq j} \frac{-1}{s+1}\sum_{p=0}^{s-1} (p+1)\, \sum_{r=1}^n \{\Phi_{l,\, i}[\phi^{(p+1)}_{r,\, l}\, \phi^{(s-p)}_{r,\, j}-\phi^{(p+1)}_{r,\, j}\, \phi^{(s-p)}_{r,\, l}]$$
	$$+\Phi_{l,\, i}[\psi^{(p+1)}_{r,\, l}\, \psi^{(s-p)}_{r,\, j}-\psi^{(p+1)}_{r,\, j}\, \psi^{(s-p)}_{r,\, l}]\}$$
	$$+\frac{(-1)^{n^2}}{s+1} 
	\sum_{k=1}^n\ \  \sum_{s_1+\dots+s_n=s-1} (s_k+1)\, \sum_{\sigma\in S_n}(-1)^{\epsilon(\sigma)}$$
	$$\{\Phi_{j,\, i}[\phi^{(s_1)}_{1,\, \sigma(1)}\, \dots\, 
	\phi^{(s_{k-1})}_{k-1,\, \sigma(k-1)}\, \phi^{(s_k+1)}_{k,\, \sigma(k)}\, \phi^{(s_{k+1})}_{k+1,\, \sigma(k+1)}\,\dots\, \phi^{(s_n)}_{n,\, \sigma(n)}]$$
	$$+\Phi_{j,\, i}[\psi^{(s_1)}_{1,\, \sigma(1)}\, \dots\, 
	\psi^{(s_{k-1})}_{k-1,\, \sigma(k-1)}\, \psi^{(s_k+1)}_{k,\, \sigma(k)}\, \psi^{(s_{k+1})}_{k+1,\, \sigma(k+1)}\,\dots\, \psi^{(s_n)}_{n,\, \sigma(n)}]\}$$
		$$\psi_{j,\, i}^{(s+1)}=\sum_{l\neq j} \frac{-1}{s+1}\sum_{p=0}^{s-1} (p+1)\, \sum_{r=1}^n \{\Psi_{l,\, i}[\phi^{(p+1)}_{r,\, l}\, \phi^{(s-p)}_{r,\, j}-\phi^{(p+1)}_{r,\, j}\, \phi^{(s-p)}_{r,\, l}]$$
	$$+\Psi_{l,\, i}[\psi^{(p+1)}_{r,\, l}\, \psi^{(s-p)}_{r,\, j}-\psi^{(p+1)}_{r,\, j}\, \psi^{(s-p)}_{r,\, l}]\}$$
	$$+\frac{(-1)^{n^2}}{s+1} 
	\sum_{k=1}^n\ \  \sum_{s_1+\dots+s_n=s-1} (s_k+1)\, \sum_{\sigma\in S_n}(-1)^{\epsilon(\sigma)}$$
	$$\{\Psi_{j,\, i}[\phi^{(s_1)}_{1,\, \sigma(1)}\, \dots\, 
	\phi^{(s_{k-1})}_{k-1,\, \sigma(k-1)}\, \phi^{(s_k+1)}_{j,\, \sigma(k)}\, \phi^{(s_{k+1})}_{k+1,\, \sigma(k+1)}\,\dots\, \phi^{(s_n)}_{n,\, \sigma(n)}]$$
	$$+\Psi_{j,\, i}[\psi^{(s_1)}_{1,\, \sigma(1)}\, \dots\, 
	\psi^{(s_{k-1})}_{k-1,\, \sigma(k-1)}\, \psi^{(s_k+1)}_{j,\, \sigma(k)}\, \psi^{(s_{k+1})}_{k+1,\, \sigma(k+1)}\,\dots\, \psi^{(s_n)}_{n,\, \sigma(n)}]\}$$
	$$\i_{j,\, i}^{(s+1)}=\sum_{l\neq j} \frac{-1}{s+1}\sum_{p=0}^{s-1} (p+1)\, \sum_{r=1}^n \Gamma_{l,\, i}[\frac{\partial\xi_r^{(p+1)}}{\partial x_l}\, \frac{\partial\xi_r^{(s-p)}}{\partial x_j}-\frac{\partial\xi_r^{(p+1)}}{\partial x_j}\, \frac{\partial\xi_r^{(s-p)}}{\partial x_l}]$$
	$$+\frac{(-1)^{n^2}}{s+1} 
	\sum_{k=1}^n\ \  \sum_{s_1+\dots+s_n=s-1} (s_k+1)\, \sum_{\sigma\in S_n}(-1)^{\epsilon(\sigma)}$$
	$$\Gamma_{j,\, i}[\frac{\partial\xi_1^{(s_1)}}{\partial x_{\sigma(1)}}\, \dots\, 
	\frac{\partial\xi_{k-1}^{(s_{k-1})}}{\partial x_{\sigma(k-1)}}\, \frac{\partial\xi_k^{(s_k+1)}}{\partial x_{\sigma(k)}}\, \frac{\partial\xi_{k+1}^{(s_{k+1})}}{\partial x_{\sigma(k+1)}}\,\dots\, \frac{\partial\xi_n^{(s_n)}}{\partial x_{\sigma(n)}}].$$

	\subsection{Estimates} The next step consists in giving estimates of the terms above using Theorem \ref{thmR} (of section 5) as a fundamental tool. More precisely the norm $$\|\ \|=\|\ \|_{L^2}+\|\ \|_{L^\infty}+ \|\ \|_{\ka{1,\, \gamma}(\Omega)}+\|\ \|_{\ka{1,\, \gamma}(\Omega^c)}$$ is multiplicative and makes continuous the action of the involved linear operators $\Phi_{i,\, j}$, $\Psi_{i,\, j}$ and $\Gamma_{i,\, j}$.
	
	We have $\|\phi_{j,\, l}^{(1)}\|=\|\chi_\Omega\, R_{j,\, l}[\rho]\|,$ $\|\psi_{j,\, l}^{(1)}\|=\|\chi_{\R^n\setminus\Omega}\, R_{j,\, l}[\rho]\|$ and $\|\i_{j,\, l}^{(1)}\|=\|\chi_\Omega\, R_{j,\, l}[\rho]\|$ and for $s>0$ we have  $$\|\phi_{j,\, i}^{(s+1)}\|\le\sum_{l\neq j} \sum_{p=0}^{s-1} \frac{p+1}{s+1}\, \sum_{r=1}^n \|\Phi_{l,\, i}\|\, \{ \|\phi^{(p+1)}_{r,\, l}\|\, \|\phi^{(s-p)}_{r,\, j}\|+\|\phi^{(p+1)}_{r,\, j}\|\, \|\phi^{(s-p)}_{r,\, l}\|$$
	$$+\|\psi^{(p+1)}_{r,\, l}\|\, \|\psi^{(s-p)}_{r,\, j}\|+\|\psi^{(p+1)}_{r,\, j}\|\, \|\psi^{(s-p)}_{r,\, l}\|\}$$
	$$+ 
	\sum_{k=1}^n\ \  \sum_{s_1+\dots+s_n=s-1} \frac{s_k+1}{s+1}\, \sum_{\sigma\in S_n}$$
	$$\|\Phi_{j,\, i}\|\, \{\|\phi^{(s_1)}_{1,\, \sigma(1)}\|\, \dots\, 
	\|\phi^{(s_{k-1})}_{k-1,\, \sigma(k-1)}\|\, \|\phi^{(s_k+1)}_{k,\, \sigma(k)}\|\, \|\phi^{(s_{k+1})}_{k+1,\, \sigma(k+1)}\|\,\dots\, \|\phi^{(s_n)}_{n,\, \sigma(n)}\|$$
	$$+\|\psi^{(s_1)}_{1,\, \sigma(1)}\|\, \dots\, 
	\|\psi^{(s_{k-1})}_{k-1,\, \sigma(k-1)}\|\, \|\psi^{(s_k+1)}_{k,\, \sigma(k)}\|\, \|\psi^{(s_{k+1})}_{k+1,\, \sigma(k+1)}\|\,\dots\, \|\psi^{(s_n)}_{n,\, \sigma(n)}]\|\}$$	
	$$\|\psi_{j,\, i}^{(s+1)}\|\le\sum_{l\neq j} \sum_{p=0}^{s-1} \frac{p+1}{s+1}\, \sum_{r=1}^n \|\Psi_{l,\, i}\|\, \{\|\phi^{(p+1)}_{r,\, l}\|\, \|\phi^{(s-p)}_{r,\, j}\|+\|\phi^{(p+1)}_{r,\, j}\|\, \|\phi^{(s-p)}_{r,\, l}]\|$$
	$$+\|\psi^{(p+1)}_{r,\, l}\|\, \|\psi^{(s-p)}_{r,\, j}\|+\|\psi^{(p+1)}_{r,\, j}\|\, \|\psi^{(s-p)}_{r,\, l}\|\}$$
	$$+
	\sum_{k=1}^n\ \  \sum_{s_1+\dots+s_n=s-1} \frac{s_k+1}{s+1} \, \sum_{\sigma\in S_n}$$
	$$\|\Psi_{j,\, i}\|\, \{\|\phi^{(s_1)}_{1,\, \sigma(1)}\|\, \dots\, 
	\|\phi^{(s_{k-1})}_{k-1,\, \sigma(k-1)}\|\, \|\phi^{(s_k+1)}_{k,\, \sigma(k)}\|\, \|\phi^{(s_{k+1})}_{k+1,\, \sigma(k+1)}\|\,\dots\, \|\phi^{(s_n)}_{n,\, \sigma(n)}\|$$
	$$+\|\psi^{(s_1)}_{1,\, \sigma(1)}\|\, \dots\, 
	\|\psi^{(s_{k-1})}_{k-1,\, \sigma(k-1)}\|\, \|\psi^{(s_k+1)}_{k,\, \sigma(k)}\|\, \|\psi^{(s_{k+1})}_{k+1,\, \sigma(k+1)}\|\,\dots\, \|\psi^{(s_n)}_{n,\, \sigma(n)}\|\}$$
	$$\|\i_{j,\, i}^{(s+1)}\|\le\sum_{l\neq j} \sum_{p=0}^{s-1} \frac{p+1}{s+1}\, \sum_{r=1}^n \|\Gamma_{l,\, i}[\frac{\partial\xi_r^{(p+1)}}{\partial x_l}\, \frac{\xi_r^{(s-p)}}{\partial x_j}]\|+\|\Gamma_{l,\, i}[\frac{\partial\xi_r^{(p+1)}}{\partial x_j}\, \frac{\xi_r^{(s-p)}}{\partial x_l}]\|$$
	$$+ 
	\sum_{k=1}^n\ \  \sum_{s_1+\dots+s_n=s-1} \frac{s_k+1}{s+1}\, \sum_{\sigma\in S_n}\|\Gamma_{j,\, i}[\frac{\partial\xi_1^{(s_1)}}{\partial x_{\sigma(1)}}\, \dots\, 
	\frac{\partial\xi_{k-1}^{(s_{k-1})}}{\partial x_{\sigma(k-1)}}\, \frac{\partial\xi_k^{(s_k+1)}}{\partial x_{\sigma(k)}}\, \frac{\partial\xi_{k+1}^{(s_{k+1})}}{\partial x_{\sigma(k+1)}}\,\dots\, \frac{\partial\xi_n^{(s_n)}}{\partial x_{\sigma(n)}}]\|.$$

	The last term is useful in the determination of the boundary values.
	
	Now, for the first two inequalities, let us denote 
	$\alpha_{a,\, b}^{(s)}=\|\phi_{a,\, b}^{(s)}\|$,  
	$\beta_{a,\, b}^{(s)}=\|\psi_{a,\, b}^{(s)}\|$, $K_{l,\, i}=\|\Phi_{l,\, i}\|$ and $L_{l,\, i}=\|\Psi_{l,\, i}\|$. We have 
	$$(s+1)\, \alpha_{j,\, i}^{(s+1)}\le\sum_{l\neq j}K_{l,\, i}\, \sum_{r=1}^n\, \sum_{p=0}^{s-1} (p+1)$$
	$$\{ \alpha^{(p+1)}_{r,\, l}\, \alpha^{(s-p)}_{r,\, j}+\alpha^{(p+1)}_{r,\, j}\, \alpha^{(s-p)}_{r,\, l}
	+\beta^{(p+1)}_{r,\, l}\, \beta^{(s-p)}_{r,\, j}+\beta^{(p+1)}_{r,\, j}\, \beta^{(s-p)}_{r,\, l}\}$$
	$$+ 
	\sum_{k=1}^n\ \  \sum_{s_1+\dots+s_n=s-1}( s_k+1)\, \sum_{\sigma\in S_n}
	K_{j,\, i}$$ $$
	\{\alpha^{(s_1)}_{1,\, \sigma(1)}\, \dots\, 
	\alpha^{(s_{k-1})}_{k-1,\, \sigma(k-1)}\, \alpha^{(s_k+1)}_{k,\, \sigma(k)}\, \alpha^{(s_{k+1})}_{k+1,\, \sigma(k+1)}\,\dots\, \alpha^{(s_n)}_{n,\, \sigma(n)}$$
	$$+\beta^{(s_1)}_{1,\, \sigma(1)}\, \dots\, 
	\beta^{(s_{k-1})}_{k-1,\, \sigma(k-1)}\, \beta^{(s_k+1)}_{k,\, \sigma(k)}\, \beta^{(s_{k+1})}_{k+1,\, \sigma(k+1)}\, \dots\, \beta^{(s_n)}_{n,\, \sigma(n)}\}$$
		$$(s+1)\, \beta_{j,\, i}^{(s+1)}\le\sum_{l\neq j} \sum_{p=0}^{s-1} (p+1)\, \sum_{r=1}^n L_{l,\, i}\, \{\alpha^{(p+1)}_{r,\, l}\, \alpha^{(s-p)}_{r,\, j}+\alpha^{(p+1)}_{r,\, j}\, \alpha^{(s-p)}_{r,\, l}$$
	$$+\beta^{(p+1)}_{r,\, l}\, \beta^{(s-p)}_{r,\, j}+\beta^{(p+1)}_{r,\, j}\, \beta^{(s-p)}_{r,\, l}\}$$
	$$+
	\sum_{k=1}^n\ \  \sum_{s_1+\dots+s_n=s-1} (s_k+1) \, \sum_{\sigma\in S_n}$$
	$$L_{j,\, i}\, \{\alpha^{(s_1)}_{1,\, \sigma(1)}\, \dots\, 
	\alpha^{(s_{k-1})}_{k-1,\, \sigma(k-1)}\, \alpha^{(s_k+1)}_{k,\, \sigma(k)}\, \alpha^{(s_{k+1})}_{k+1,\, \sigma(k+1)}\,\dots\, \alpha^{(s_n)}_{n,\, \sigma(n)}$$
	$$+\beta^{(s_1)}_{1,\, \sigma(1)}\, \dots\, 
	\beta^{(s_{k-1})}_{k-1,\, \sigma(k-1)}\, \beta^{(s_k+1)}_{k,\, \sigma(k)}\, \beta^{(s_{k+1})}_{k+1,\, \sigma(k+1)}\,\dots\, \beta^{(s_n)}_{n,\, \sigma(n)}\}.$$
	
	Let us define also $$A_{r,\, l,\, j}^{(s)}=\sum_{p=0}^{s-1} (p+1)\alpha^{(p+1)}_{r,\, l}\, \alpha^{(s-p)}_{r,\, j},$$
	$$B_{r,\, l,\, j}^{(s)}=\sum_{p=0}^{s-1} (p+1)\beta^{(p+1)}_{r,\, l}\, \beta^{(s-p)}_{r,\, j},$$
	$$C_{k,\, \sigma}^{(s)}=\sum_{s_1+\dots+s_n=s-1}( s_k+1)$$ 
	$$
	\alpha^{(s_1)}_{1,\, \sigma(1)}\, \dots\, 
	\alpha^{(s_{k-1})}_{k-1,\, \sigma(k-1)}\, \alpha^{(s_k+1)}_{k,\, \sigma(k)}\, \alpha^{(s_{k+1})}_{k+1,\, \sigma(k+1)}\,\dots\, \alpha^{(s_n)}_{n,\, \sigma(n)}$$
	$$D_{k,\, \sigma}^{(s)}=\sum_{s_1+\dots+s_n=s-1}( s_k+1)$$ 
	$$+\beta^{(s_1)}_{1,\, \sigma(1)}\, \dots\, 
	\beta^{(s_{k-1})}_{k-1,\, \sigma(k-1)}\, \beta^{(s_k+1)}_{k,\, \sigma(k)}\, \beta^{(s_{k+1})}_{k+1,\, \sigma(k+1)}\, \dots\, \beta^{(s_n)}_{n,\, \sigma(n)}.$$ Then
	$$(s+1)\, \alpha_{j,\, i}^{(s+1)}\le\sum_{l\neq j}K_{l,\, i}\, \sum_{r=1}^n\, \{A_{r,\, l,\, j}^{(s)}+A_{r,\, j,\, l}^{(s)}
	+B_{r,\, l,\, j}^{(s)}+B_{r,\, j,\, l}^{(s)}\}$$
	$$+ 
	\sum_{k=1}^n\, \sum_{\sigma\in S_n}
	K_{j,\, i}
	\{C_{k,\, \sigma}^{(s)}+D_{k,\, \sigma}^{(s)}\}
	$$
	$$(s+1)\, \beta_{j,\, i}^{(s+1)}\le\sum_{l\neq j}L_{l,\, i}\, \sum_{r=1}^n\, \{A_{r,\, l,\, j}^{(s)}+A_{r,\, j,\, l}^{(s)}
	+B_{r,\, l,\, j}^{(s)}+B_{r,\, j,\, l}^{(s)}\}$$
	$$+ 
	\sum_{k=1}^n\, \sum_{\sigma\in S_n}
	L_{j,\, i}\, 
	\{C_{k,\, \sigma}^{(s)}+D_{k,\, \sigma}^{(s)}\}
	.$$
	
	\medskip
	
	Next, after defining $\alpha_{j,\, i}^{(0)}=\beta_{j,\, i}^{(0)}=0$ for all $i,\, j$, we consider, \newline for $\tau\in[0,\, \infty)$, the functions $$f_{j,\, i}^{(N)}(\tau)=\sum_{s=0}^N \alpha_{j,\, i}^{(s)}\, \tau^s$$
	and $$g_{j,\, i}^{(N)}(\tau)=\sum_{s=0}^N \beta_{j,\, i}^{(s)}\, \tau^s.$$ Then 
	$$(f_{j,\, i}^{(N+1)})'(\tau)=\sum_{s=0}^N (s+1)\, \alpha_{j,\, i}^{(s+1)}\, \tau^s=\alpha_{j,\, i}^{(1)}+\sum_{s=1}^N (s+1)\, \alpha_{j,\, i}^{(s+1)}\, \tau^s$$
	$$\le\alpha_{j,\, i}^{(1)}+\sum_{l\neq j}K_{l,\, i}\, \sum_{r=1}^n\, \sum_{s=0}^N\{A_{r,\, l,\, j}^{(s)}+A_{r,\, j,\, l}^{(s)}
	+B_{r,\, l,\, j}^{(s)}+B_{r,\, j,\, l}^{(s)}\}\, \tau^s$$
	$$+ \sum_{k=1}^n\, \sum_{\sigma\in S_n}
	K_{j,\, i}\, \sum_{s=0}^N
	\{C_{k,\, \sigma}^{(s)}+D_{k,\, \sigma}^{(s)}
	\}\, \tau^s$$
	and $$(g_{j,\, i}^{(N+1)})'(\tau)=\sum_{s=0}^N (s+1)\, \beta_{j,\, i}^{(s+1)}\, \tau^s$$ 
	$$\le\beta_{j,\, i}^{(1)}+\sum_{l\neq j}L_{l,\, i}\, \sum_{r=1}^n\, \sum_{s=0}^N\{A_{r,\, l,\, j}^{(s)}+A_{r,\, j,\, l}^{(s)}
	+B_{r,\, l,\, j}^{(s)}+B_{r,\, j,\, l}^{(s)}\}\, \tau^s$$
	$$+ \sum_{k=1}^n\, \sum_{\sigma\in S_n}
	L_{j,\, i}\, \sum_{s=0}^N
	\{C_{k,\, \sigma}^{(s)}+D_{k,\, \sigma}^{(s)}
	\}\, \tau^s.$$
	
	\medskip
	
	Now, \begin{lem} Using the notation above the following inequalities are satisfied:
		
		\begin{itemize} 
			
			\item[a)] $\sum_{s=0}^NA_{r,\, l,\, j}^{(s)}\, \tau^s\le(f_{r,\, l}^{(N)})'(\tau)\, f_{r,\, j}^{(N)}(\tau)$
			
			\item[b)] $\sum_{s=0}^NB_{r,\, l,\, j}^{(s)}\, \tau^s\le(g_{r,\, l}^{(N)})'(\tau)\, g_{r,\, j}^{(N)}(\tau)$
			
			\item[c)] $\sum_{s=0}^NC_{k,\, \sigma}^{(s)}\, \tau^s\le (f_{k,\, \sigma(k)}^{(N)})'(\tau)\, \Pi_{l\neq k}f_{l,\, \sigma(l)}^{(N)}(\tau)$
			
			\item[d)] $\sum_{s=0}^ND_{k,\, \sigma}^{(s)}\, \tau^s\le(g_{k,\, \sigma(k)}^{(N)})'(\tau)\, \Pi_{l\neq k}g_{l,\, \sigma(l)}^{(N)}(\tau)$
		\end{itemize}
		
	\end{lem}
	\begin{proof} It is based on the argument used in pg. 22 of \cite{BurMat} and we only develop the third case:
		$$f_{1,\, \sigma(1)}^{(N)}(\tau)\, \dots\, f_{k-1,\, \sigma(k-1)}^{(N)}(\tau)\, (f_{k,\, \sigma(k)}^{(N)})'(\tau)\, f_{k+1,\, \sigma(k+1)}^{(N)}(\tau)\, \dots f_{n,\, \sigma(n)}^{(N)}(\tau)$$
		$$=(\sum_{p_1=0}^N \alpha_{1,\, \sigma(1)}^{(p_1)}\, \tau^{p_1})\, \dots\, (\sum_{p_k=0}^{N-1} (p_k+1)\, \alpha_{k,\, \sigma(k)}^{(p_k+1)}\, \tau^{p_k})\, \dots\, (\sum_{p_n=0}^N \alpha_{n,\, \sigma(n)}^{(p_n)}\, \tau^{p_n})$$
		$$=\sum_{p_1=0}^N\, \dots\, \sum_{p_k=0}^{N-1}\, \dots\, \sum_{p_n=0}^N  (p_k+1)\, \alpha_{1,\, \sigma(1)}^{(p_1)}\, \dots\, \alpha_{k,\, \sigma(k)}^{(p_k+1)}\, \tau^{p_k}\, \dots\, \alpha_{n,\, \sigma(n)}^{(p_n)}\, \tau^{p_1+\dots+p_n}$$
		$$=\sum_{s=0}^{n\, N-1}\, \left(\sum_{p_k=0}^{\min\{s,\, N-1\}} (p_k+1)\, \alpha_{k,\, \sigma(k)}^{(p_k+1)}\,  \,\sum_{\sum_{l\neq k}p_l=s-p_k;\ p_l\le N}\alpha_{1,\, \sigma(1)}^{(p_1)}\,  \dots\, \widehat{\alpha_{k,\, \sigma(k)}^{(p_k)}}\, \dots\, \alpha_{n,\, \sigma(n)}^{(p_n)}\, \right)\tau^s$$
		$$\geq\sum_{s=0}^{N}\, \left(\sum_{p_1+\dots+p_n=s-1} (p_k+1)\,  \alpha_{1,\, \sigma(1)}^{(p_1)}\,  \dots\, \alpha_{k,\, \sigma(k)}^{(p_k+1)}\, \dots\, \alpha_{n,\, \sigma(n)}^{(p_n)}\, \right)\tau^s=\sum_{s=0}^{N}C_{k,\, \sigma}^{(s)}\, \tau^s$$

	\end{proof}

	So $$(f_{j,\, i}^{(N+1)})'(\tau)\le\alpha_{j,\, i}^{(1)}$$
	$$+\sum_{l\neq j}K_{l,\, i}\, \sum_{r=1}^n\, \{(f_{r,\, l}^{(N)})'(\tau)\, f_{r,\, j}^{(N)}(\tau)+(f_{r,\, j}^{(N)})'(\tau)\, f_{r,\, l}^{(N)}(\tau)$$
	$$+(g_{r,\, l}^{(N)})'(\tau)\, g_{r,\, j}^{(N)}(\tau)+(g_{r,\, j}^{(N)})'(\tau)\, g_{r,\, l}^{(N)}(\tau)\}$$
	$$+ \sum_{k=1}^n\, \sum_{\sigma\in S_n}
	K_{j,\, i}\, 
	\{(f_{k,\, \sigma(k)}^{(N)})'(\tau)\, \Pi_{l\neq k}f_{l,\, \sigma(l)}^{(N)}(\tau)+(g_{k,\, \sigma(k)}^{(N)})'(\tau)\, \Pi_{l\neq k}g_{l,\, \sigma(l)}^{(N)}(\tau)
	\}$$
	$$=\alpha_{j,\, i}^{(1)}+\sum_{l\neq j}K_{l,\, i}\, \sum_{r=1}^n\, (f_{r,\, l}^{(N)}\, f_{r,\, j}^{(N)}+g_{r,\, l}^{(N)}\, g_{r,\, j}^{(N)})'(\tau)$$
	$$+ 
	K_{j,\, i}\, \sum_{\sigma\in S_n} (\Pi_{k=1}^n f_{l,\, \sigma(l)}^{(N)}
	+\Pi_{k=1}^n g_{l,\, \sigma(l)}^{(N)})'(\tau)
	$$
	and 
	$$(g_{j,\, i}^{(N+1)})'(\tau)\le\beta_{j,\, i}^{(1)}+\sum_{l\neq j}L_{l,\, i}\, \sum_{r=1}^n\, \{(f_{r,\, l}^{(N)})'(\tau)\, f_{r,\, j}^{(N)}(\tau)+(f_{r,\, j}^{(N)})'(\tau)\, f_{r,\, l}^{(N)}(\tau)$$
	$$+(g_{r,\, l}^{(N)})'(\tau)\, g_{r,\, j}^{(N)}(\tau)+(g_{r,\, j}^{(N)})'(\tau)\, g_{r,\, l}^{(N)}(\tau)\}$$
	$$+ \sum_{k=1}^n\, \sum_{\sigma\in S_n}
	L_{j,\, i}\, \sum_{s=0}^N
	\{(f_{k,\, \sigma(k)}^{(N)})'(\tau)\, \Pi_{l\neq k}f_{l,\, \sigma(l)}^{(N)}(\tau)+(g_{k,\, \sigma(k)}^{(N)})'(\tau)\, \Pi_{l\neq k}g_{l,\, \sigma(l)}^{(N)}(\tau)
	\}$$
	$$=\beta_{j,\, i}^{(1)}+\sum_{l\neq j}L_{l,\, i}\, \sum_{r=1}^n\, (f_{r,\, l}^{(N)}\, f_{r,\, j}^{(N)}+g_{r,\, l}^{(N)}\, g_{r,\, j}^{(N)})'(\tau)$$
	$$+ 
	L_{j,\, i}\, \sum_{\sigma\in S_n} (\Pi_{k=1}^n f_{l,\, \sigma(l)}^{(N)}
	+\Pi_{k=1}^n g_{l,\, \sigma(l)}^{(N)})'(\tau)
	$$
	
	Finally, defining $h_{j,\, i}^{(N+1)}=f_{j,\, i}^{(N+1)}+g_{j,\, i}^{(N+1)}$ and adding the two expressions above, we have $$(h_{j,\, i}^{(N+1)})'(\tau)\le\alpha_{j,\, i}^{(1)}+\beta_{j,\, i}^{(1)}$$
	$$+\sum_{l\neq j}(K_{l,\, i}+L_{l,\, i})\, 2\, \sum_{r=1}^n\, (h_{r,\, l}^{(N)}\, h_{r,\, j}^{(N)})'(\tau)$$
	$$+2\, (K_{j,\, i}+ 
	L_{j,\, i})\, \sum_{\sigma\in S_n} (\Pi_{k=1}^n h_{l,\, \sigma(l)}^{(N)}
	)'(\tau),
	$$ so, since all the terms are positive, the primitive functions satisfy the inequalities too and since $\alpha_{j,\, i}^{(0)}=\beta_{j,\, i}^{(0)}=0$, defining $c_{j,\, i}=\alpha_{j,\, i}^{(1)}+\beta_{j,\, i}^{(1)}$ and $M_{l,\, i}=2\,(K_{l,\, i}+L_{l,\, i})$ we have
	\begin{equation}\label{desig} 
		h_{j,\, i}^{(N+1)}(\tau)\le c_{j,\, i}^{(1)}\, \tau+\sum_{l\neq j}M_{l,\, i}\, \sum_{r=1}^n\, h_{r,\, l}^{(N)}\, h_{r,\, j}^{(N)}(\tau)\\
		+M_{j,\, i}\, \sum_{\sigma\in S_n} \Pi_{k=1}^n h_{l,\, \sigma(l)}^{(N)}
		(\tau).
	\end{equation} 
	
	\begin{prop}
		For any $N\in\N\setminus\{0\}$, let $h_N$ a sequence of functions satisfying~\eqref{desig}, defining $$M=\max \{M_{l,\, i};\, l,\, i=1,\, \dots,\, n\}$$ and $$c=\max \{c_{l,\, i}^{(1)};\,l,\, i=1,\, \dots,\, n\},$$ there exists $\alpha\in (0,\, 1)$ such that for $\tau\in(0,\, \frac{(1-\alpha)^2}{8\, n\, (n-1)\, M\, c})$ we have 
		$$
		h_N(\tau)\le\frac{(1-\alpha)}{4\, \sqrt{2}\, n\, (n-1)\, M}+\frac{1}{2}\, \min\{(\frac{\alpha}{n!\, M})^{\frac{1}{n-1}},\, \frac{\frac{3}{2}\, (1-\alpha)}{4\, n\, (n-1)\, M}\}.$$
	\end{prop}          
	
	\begin{proof}
		We want to prove that there exists $E>0$ and $\tau_0>0$ such that  
		$$
		h_{j,\, i}^{(N)}(\tau)\le E
		$$ for any $N\in\N\setminus\{0\}$ and for every $i,\, j=1,\, \dots,\, n$, if $0\le\tau<\tau_0$. 
		
		By the previous inequality \eqref{desig}, we have that 
		$$h_{j,\, i}^{(N+1)}(\tau)\le c_{j,\, i}^{(1)}\, \tau_0+\sum_{l\neq j}M_{l,\, i}\, \sum_{r=1}^n\, h_{r,\, l}^{(N)}\, h_{r,\, j}^{(N)}(\tau)$$
$$		+M_{j,\, i}\, \sum_{\sigma\in S_n} \Pi_{k=1}^n h_{l,\, \sigma(l)}^{(N)}
		(\tau)$$
		$$\le c_{j,\, i}^{(1)}\, \tau_0+\sum_{l\neq j}M_{l,\, i}\, \sum_{r=1}^n\, E_{r,\, l}\, E_{r,\, j}	+M_{j,\, i}\, \sum_{\sigma\in S_n} \Pi_{k=1}^n E_{l,\, \sigma(l)},$$ and assuming that the statement is true for $N$, 
		we have $$h_{j,\, i}^{(N+1)}(\tau)\le c\, \tau_0+n\, (n-1)\, M\, E^2+n!\, M\, E^n,$$ and we want $h_{j,\, i}^{(N+1)}(\tau)\le E$, so $$c\, \tau_0+n\, (n-1)\, M\, E^2+n!\, M\, E^n-E\le0.$$
		
		Then we consider the polynomial $$p(x)=x^n+\frac{1}{(n-2)!}\, x^2-\frac{1}{n!\, M}\, x+\frac{c\, \tau_0}{n!\, M}$$ for $x\geq0$.
		
		The function $p(x)$ is convex, $p(0)>0$ and $\lim_{x\to+\infty}p(x)=+\infty$. 
		
		Moreover if $0<\alpha<1$ and $0\le x\le(\frac{\alpha}{n!\, M})^{\frac{1}{n-1}}$, then $$p(x)\le\frac{1}{(n-2)!}\, x^2-\frac{1-\alpha}{n!\, M}\, x+\frac{c\, \tau_0}{n!\, M}=q(x).$$
		
		The discriminant of $q$, $$D(q)=(\frac{1-\alpha}{n!\, M})^2-4\, \frac{c\, \tau_0}{(n-2)!\, n!\, M}$$ is positive whenever $$\tau_0<\frac{(1-\alpha)^2}{4\, n\, (n-1)\, c\, M}$$ and the roots of $q$ are $$\frac{(n-2)!}{2}\, \{\frac{1-\alpha}{n!\, M}\pm\sqrt{(\frac{1-\alpha}{n!\, M})^2-4\, \frac{c\, \tau_0}{(n-2)!\, n!\, M}}\}$$
		$$=\frac{1-\alpha\pm\sqrt{(1-\alpha)^2-4\, c\, \tau_0\, n\, (n-1)\, M}}{2\, n\, (n-1)\, M}.$$
		
		In order to have room enough for $x$ to make $q(x)<0$, we need $$\frac{1-\alpha-\sqrt{(1-\alpha)^2-4\, c\, \tau_0\, n\, (n-1)\, M}}{2\, n\, (n-1)\, M}<(\frac{\alpha}{n!\, M})^{\frac{1}{n-1}}.$$ 
		
		Let us put $$\tau_0=\gamma\, \frac{(1-\alpha)^2}{4\, n\, (n-1)\, c\, M}$$ for some $\gamma\in(0,\, 1)$ and consider the function $$\varphi(\alpha)=\frac{(1-\alpha)\, \sqrt{(1-\gamma)}}{2\, n\, (n-1)\, M}-(\frac{\alpha}{n!\, M})^{\frac{1}{n-1}}.$$ It is clear that $\varphi(0)>0$ and $\varphi(1)<0$ and also, for $\alpha\in(0,\, 1)$, $$\varphi'(\alpha)=-\frac{\sqrt{(1-\gamma)}}{2\, n\, (n-1)\, M}-(\frac{1}{n!\, M})^{\frac{1}{n-1}}\, \frac{1}{(n-1)\, \alpha^\frac{n-2}{n-1}}<0,$$ so there is an interval of allowable values of $\alpha$ contained in $(0,\, 1)$ and once $\alpha$ is chosen in this interval then for any choice of $\gamma$ we have that if $\tau_0=\gamma\, \frac{(1-\alpha)^2}{4\, n\, (n-1)\, c\, M}$ and $$\frac{(1-\alpha)\, \sqrt{(1-\gamma)}}{2\, n\, (n-1)\, M}<x<(\frac{\alpha}{n!\, M})^{\frac{1}{n-1}},$$ then $p(x)<0$. We have taken $\mu=1-\alpha$ and $\gamma=\dfrac{1}{2}$ in the statement of the proposition.

	\end{proof}

	\subsection{End of the proof of Theorem \ref{T1}.}
	
	The previous proposition implies that 
	$$
	\sum_{s=0}^N\biggl\|\frac{\partial\xi_j^{(s)}}{\partial x_i}\biggr\|\, \tau^s\le h_N(\tau)\le C,
	$$ where $C=C(n,\, M)$ is the constant in the previous proposition, and
	for every $N$ and $\tau\in[0,\, \frac{1}{16\, n\, (n-1)\, M\, c})$.

	Since the functions $\frac{\partial\xi^{(s)}}{\partial  x_i}$ are in $\ka{\gamma}$ and, as the recurrence shows, they decay at $\infty$ of order $\frac{1}{\|x\|^n}$, then
	$$
	\xi_j^{(s)}= \sum_{i=1}^n K_i[\frac{\partial\xi_j^{(s)}}{\partial x_i}]
	$$ 
	is in $\ka{1,\, \gamma}$. Moreover 
	$$
	\|\xi_j^{(s)}\|_{L^\infty}\le C \, \|\nabla\xi_j^{(s)}\|_{\gamma}.
	$$
	
	
	Consequently, for any $N$ and $j$, 
	$$
	\biggl|\sum_{s=0}^N \xi_j^{(s)}(x)\, t^s\biggr|\le\sum_{s=0}^N \|\xi_j^{(s)}\|_{L^\infty}\, |t|^s\le C\, \sum_{s=0}^N 
	\|\nabla\xi_j^{(s)}\|\, |t|^s
	$$ 
	and the last term is uniformly bounded in $N$, for $|t|\le\frac{1}{16\, n\, (n-1)\, M\, c}$.
	
	This implies the existence of an analytic solution of the equations~(\ref{div1}) and~(\ref{rot1}), in a neighbourhood of $t=0$.
	
	The statements 1) and 2) are consequence of the facts that  $\phi(\R^n,\, t)$ is a closed set and $\phi(\ ,\, t)$ is one to one for every $t$. Moreover $\phi(\R^n ,\, t)=\R^n$ and $\phi(\ ,\, t)$ is differentiable in $\R^n\setminus\partial\Omega$. as being proved in 
	the planar case (Proposition 6 in \cite{BurMat}). $\square$

	\section{Estimates of the Riesz operators.} 
	
	In this section we develop the concepts and techniques leading to the estimates used in the proof of Theorem \ref{T1}. This is mainly achieved by performing  analysis of the Riesz operators in particular spaces. 
	
	\subsection{Basic geometric lemmas}
	The metric and geometric properties of $\partial\Omega$ play a crucial role in the behaviour of the Riesz transform in $\Omega$. The next lemma is a synthesis of these properties.
	
	\begin{lem}[The geometric lemma]\label{geom}
		Let $\Omega\subset\R^n$ be a bounded domain such that $\partial\Omega\in\ka{1,\gamma}$, defined by a function $\rho$.
		
		There exists $0<R_0<1$ such that if
		$$
		U_{R_0}(\partial\Omega):=\cup_{x\in\partial\Omega}B_{R_0}(x),
		$$
		then there exists $R_1$ such that for any $x_0\in U_{R_0}$ the level set $$\{\rho=\rho(x_0)\}\cap B_{R_1}(x_0)$$ coincides with the graph of a function $\varphi_{x_0}$.
		
		Moreover $\varphi_{x_0}$ is $\ka{1,\gamma}$, $\varphi_{x_0}(0)=0$ and $\nabla\varphi_{x_0}(0)=0$.
	\end{lem}
	\medskip
	
	\begin{rmk} The function $\varphi_{x_0}$ is defined on a region of the tangent hyperplane to the level set of $\rho$ across $x_0$.
		
		Then
		$$\rho(x)=\rho(x_0+\sum_{j=1}^{n-1} s_j\, \tau_j(x_0)+\varphi_{x_0}(s)\,\eta(x_0)),
		$$ where $\tau_1(x_0),\, \dots,\, \tau_{n-1}(x_0)$ is an orthonormal complement to $\eta(x_0)$ in $T_{x_0}(\R^n)$.
		
	\end{rmk}
	
	
	\begin{proof}
		The proof is identical to the corresponding one in the case $n=2$ as we developed in \cite{BurMat}.
	\end{proof} 
	
	The next fact is proved in Lemma 3 in \cite{MOV} in a more general situation. 
	
	\begin{lem}[Riesz-Geometric lemma]\label{rszgl} 
		If $\Omega\subset\R^n$ is a bounded domain with boundary of class $\ka{1, \gamma}$, then there exists 
		$$
		\lim_{\epsilon\to0}\int_{C^{\frac{R_0}{2}}_\epsilon(x)\cap\bar \Omega}(R_{j,\, i})(\zeta-x)\, dm(\zeta)$$
		$$=\biggl\{\int_{B_{\frac{R_0}{2}}(x)\cap\bar \Omega\cap\{\kappa_x>0\}}-\int_{B_{\frac{R_0}{2}}(x)\cap \Omega^c\cap\{\kappa_x<0\}}\biggr\}(R_{j,\, i})(\zeta-x)\, dm(\zeta).
		$$  
	\end{lem}

	\subsection{Decomposition of the singularities}
	
	Set, for any $x\in\R^n$, $d(x)=d(x,\, \partial\Omega)$ and $\delta(x)=\max\{d(x),\, \frac{R_0}{2}\}$ ($R_0$ as in Lemma \ref{geom}). 
	Then
	
	\begin{prop}\label{desing} Let $\Omega\subset\R^n$ be a bounded domain with boundary of class $\ka{1, \gamma}$ and let $W$ denote $\Omega$ or $\R^n\setminus\bar\Omega$. 
		
		If $f\in\ka{\gamma}(\bar W)\cap L^p(W)$ with $p>1$ and we identify $f$ with its extension by $0$ outside $W$, then there exists 
		$$
		R_{j,\, i}[f](x)=\lim_{\epsilon\to0} (R_{j,\, i})_\epsilon[f](x)
		$$ 
		for every $x\in\R^n$.
		
		Moreover
		$$
		R_{j,\, i}[f](x)=Q_{j,\, i}[f](x)+L_{j,\, i}[f](x)+f(x)\, \Theta_{j,\, i;\, W}^{\frac{R_0}{2}}(x),
		$$ 
		where
		\begin{align*}
			Q_{j,\, i}[f](x)&=\int_{\R^n\setminus B_{\delta(x)}(x)}f(\zeta)\, (R_{j,\, i})(\zeta-x)\, dm(\zeta),\\*[5pt]
			L_{j,\, i}[f](x)&=\int_{B_{\delta(x)}(x)}\frac{f(\zeta)-f(x)}{\|\zeta-x\|^\gamma}\, 
			\|\zeta-x\|^\gamma\, (R_{j,\, i})(\zeta-x)\, dm(\zeta),
		\end{align*} 
		and	if $\kappa_x$ is the first degree polynomial defining the hyperplane tangent to $\partial W$ at the point $x$, then 	
		$$
		\Theta_{W;\, j,\ i}^{\frac{R_0}{2}}(x)=\begin{cases} 
			0 &\text{if } d(x)>0,\\
			\{\int_{B_{\frac{R_0}{2}}(x)\cap\bar W\cap\{\kappa_x>0\}}-\int_{B_{\frac{R_0}{2}}(x)\cap W^c\cap\{\kappa_x<0\}}\}(R_{j,\, i})(\zeta-x)\, dm(\zeta)\,  &\text{if } d(x)=0. 
		\end{cases}
		$$ 
		
	\end{prop}
	
	\medskip
	\begin{rmk} The term $\Theta_{W;\, j,\ i}^{\frac{R_0}{2}}(x)$ is an intrinsic geometric object. Sometimes we will also use it in the form 
$$
		\int_{B_{\frac{R_0}{2}}(x)\cap W}(R_{j,\, i})(\zeta-x)\, dm(\zeta).
		$$
		
	\end{rmk}
	
	\medskip
	
	\begin{rmk} In order to avoid notation, from now on we will use in the proofs of results the notation $(R_{j,\, i})_x(\zeta)$ instead of $(R_{j,\, i})(\zeta-x) dm(\zeta)$.
	\end{rmk}
	
	\medskip
	
	\begin{rmk} Since $R_{j,\, i}$ are classical Calder\'on--Zygmund operators, the existence of the principal value is well known for functions in many different spaces (See \cite[Corollary 5.8]{Duo}), nevertheless, we need here the existence of the principal value at each point in the plane, for functions in the aforementioned class. 
	\end{rmk}
	
	\begin{proof} Let $x\in\R^n$. 
		
		\begin{itemize}
			
			\item  If $d(x)>0$, then for $0<\epsilon<d(x)$ we have 
			$$
			(R_{j,\, i})_\epsilon[f](x)=\int_{\R^n\setminus B_{\epsilon}(x)}f\, (R_{j,\, i})_x$$
			$$=\int_{\R^n\setminus B_{d(x)}(x)}f\, (R_{j,\, i})_x
			+\int_{B_{d(x)}(x)\setminus B_\epsilon(x)}f\, (R_{j,\, i})_x.
			$$
			
			\begin{itemize}
				\item In the case of $x\notin W$, the second term is $0$ and we have 
				$$
				R_{j,\, i}[f](x)=\lim_{\epsilon\to0} (R_{j,\, i})_\epsilon[f](x)=\int_{\R^n\setminus B_{d(x)}(x)}f\, (R_{j,\, i})_x.
				$$
				
				\item In the case of $x\in W$, we have   
				\begin{equation*}
					\begin{split}
						(	R_{j,\, i})_\epsilon[f](x)&=\!\int_{\R^n\setminus B_{d(x)}(x)}f\, (R_{j,\, i})_x\!+\!\int_{B_{d(x)}(x)\setminus B_\epsilon(x)}(f-f(x))\, (R_{j,\, i})_x\\*[5pt]
						&=\!\int_{\R^n\setminus B_{d(z)}(z)}f\, (R_{j,\, i})_x\!+\!\int_{B_{d(x)}(x)\setminus B_\epsilon(x)}\!\frac{f(\zeta)-f(x)}{\|\zeta-x\|^\gamma}\, 
						\|\zeta-x\|^\gamma(R_{j,\, i})_x\, dm(\zeta),\!\!\!
					\end{split}
				\end{equation*} 
				and in the second term  
				has a weakly singular kernel acting against an integrable function, so the limit exists and  
				\begin{equation*}
					\begin{split}
						R_{j,\, i}[f](x)&=\lim_{\epsilon\to0} (R_{j,\, i})_\epsilon[f](x)\\*[5pt]
						&=\int_{\R^n\setminus B_{d(x)}(x)}f\, (R_{j,\, i})_x+\int_{B_{d(x)}(x)}\frac{f(\zeta)-f(x)}{\|\zeta-x\|^\gamma}\, 
						\|\zeta-x\|^\gamma\, (R_{j,\, i})_x\, dm(\zeta).
					\end{split}
				\end{equation*} 
				
				If $d(x)\leq\frac{R_0}{2}$, then 
				$$
				\int_{\R^n\setminus B_{d(x)}(x)}f\, (R_{j,\, i})_x=
				\int_{\R^n\setminus B_{\frac{R_0}{2}}(x)}f\, (R_{j,\, i})_x+\int_{B_{\frac{R_0}{2}}(x)\setminus B_{d(x)}(x)}f\, (R_{j,\, i})_x,
				$$ 
				and  applying again the cancellation lemma (below), we have that the second integral is  
				$$
				\int_{C_{d(x)}^{\frac{R_0}{2}}(x)}\frac{f(\zeta)-f(x)}{\|\zeta-x\|^\gamma}\, 
				\|\zeta-x\|^\gamma\, (R_{j,\, i})_x\, dm(\zeta),
				$$ where $C_{d(x)}^{\frac{R_0}{2}}=B_{\frac{R_0}{2}}(x)\setminus B_{d(x)}(x)$.

				Implementing these facts in the formula for $R_{j,\, i}[f](x)$ above, we have the result. 
			\end{itemize}	
			
			\item If $d(x)=0$, then we for any $\epsilon<\frac{R_0}{2}$ we have 
			$$
			(R_{j,\, i})_\epsilon[f](z)=\int_{\R^n\setminus B_{\frac{R_0}{2}}(x)}f\, (R_{j,\, i})_x
			+\int_{C^{\frac{R_0}{2}}_\epsilon(x)}f\, (R_{j,\, i})_x.
			$$
			
			The second term, is 
			\begin{equation*}
				\begin{split}
					\int_{C^{\frac{R_0}{2}}_\epsilon(x)\cap\bar W} f\, (R_{j,\, i})_x &=\int_{C^{\frac{R_0}{2}}_\epsilon(x)\cap\bar W}(f-f(x))\, (R_{j,\, i})_x+f(x)\, \int_{C^{\frac{R_0}{2}}_\epsilon(x)\cap\bar W}(R_{j,\, i})_x\\*[5pt]
					&=\int_{C^{\frac{R_0}{2}}_\epsilon(x)\cap\bar W}\frac{f(\zeta)-f(x)}{\|\zeta-x\|^\gamma}\, 
					\|\zeta-x\|^\gamma\, (R_{j,\, i})_x\, dm(\zeta)\\*[5pt]
					&\quad+f(x) \int_{C^{\frac{R_0}{2}}_\epsilon(x)\cap\bar W}(R_{j,\, i})_x.
				\end{split}
			\end{equation*}
			
			From these two terms, the first one is weakly singular integral and has a limit
			$$
			\int_{B_{\frac{R_0}{2}}(x)\cap\bar W}\frac{f(\zeta)-f(x)}{\|\zeta-x\|^\gamma}\, 
			\|\zeta-x\|^\gamma\, (R_{j,\, i})_x\, dm(\zeta).
			$$
			
			
			The second term, according to the Riesz-Geometric lemma has a limit as $\epsilon\to0$, that is
			\begin{equation*}
				f(x)\, \biggl\{\int_{B_{\frac{R_0}{2}}(x)\cap\bar W\cap\{\kappa_x>0\}}-\int_{B_{\frac{R_0}{2}}(x)\cap W^c\cap\{\kappa_x<0\}}\biggr\}(R_{j,\, i})_x.\qedhere
			\end{equation*}
		\end{itemize}	
	\end{proof}

	\subsection{The estimates.} 
	
	Let us assume WLOG that $0\in\Omega$ and also that $R_0$ in Lemma~\ref{geom} is not larger than $1$. 
	
	The parameter 
	$$
	\vartriangle(x)=\sqrt{\max\{\|x\|^2,\, d(x)^2\}},
	$$ 
	for $x\in(\overline{\Omega\cup U_{R_0}})^c$ plays an important role in the estimates.
	
	Concerning the Riesz operators $R_{j,i}$ we will use the same notation for the operators and for the kernels and even using $R$ and avoiding the sub-indexes unless it be strictly necessary.  We also avoid using the indication for Lebesgue measure ($dm$) in most cases.
	
	The following theorem condensates the main estimates. 
	
	\begin{thm}\label{thmR} Let $\Omega$ be a bounded domain with boundary of class $\ka{1, \gamma}$, $\phi\in\ka{\infty}(\Omega)\cap \Lip{(\gamma,\bar\Omega)}$
		and $\psi\in\ka{\infty}(\R^n\setminus\bar\Omega)\cap \Lip{(\gamma,\R^n\setminus\Omega)}\cap L^2(\R^n\setminus\Omega)$. We use the same notation $\phi$ or $\psi$ for their extension by $0$ to the whole space. 
		
		Assume also that for a fixed constant $C(\psi)$ we have $$|\psi(x)|\le\frac{C(\psi)}{\delta(x)^n}$$ for $x\in(\Omega\cup U_{R_0})$. We will use the notation $f=\phi$ or $f=\psi$ whenever no distinction be necessary.


		\begin{enumerate}
			\item  $R[f]\in(L^\infty\cap L^2)(\R^n)$.
			
			\item There exists a constant $C=C(\gamma,\, \Omega,\, R_0)$ such that 
			
			\begin{enumerate}
				\item[i)] $$\|\chi_\Omega\, R[f]\|_{L^\infty}\le C\, \biggl(\|f\|_{Lip(\gamma,\, \bar\Omega)}+\|f\|_{Lip(\gamma,\, \R^n\setminus\Omega)}\biggr),
				$$ 
				
				\item[ii)] If $x\in\Omega^c\cap U_{R_0}$ then $$|(\chi_{\R^n\setminus\bar\Omega}\, R[f])(x)|\le C\, \biggl(\|f\|_{Lip(\gamma,\, \bar\Omega)}+\|f\|_{Lip(\gamma,\, \R^n\setminus\Omega)}\biggr),
				$$ 
				
				\item[iii)] If $x\in\R^n\setminus(\overline{\Omega^c\cap U_{R_0}})^c$ then $$|(\chi_{\R^n\setminus\Omega}\, R[\phi])(x)|\le C\, \frac{\|\phi\|_{Lip(\gamma,\, \bar\Omega)}}{\delta(x)^n},
				$$ 
				
				\item[iv)] If $x\in\R^n\setminus(\overline{\Omega^c\cap U_{R_0}})^c$ then $$|(\chi_{\R^n\setminus\Omega}\, R[\psi])(x)|\le C\, (1+\|\psi\|_{Lip(\gamma,\, \bar\Omega)})\, C(\psi)\, \biggl\{\frac{1}{\delta(x)^n}+\frac{1}{\triangle(x)^n}\, (1+\ln(\triangle(x)))\biggr\}.
				$$ 
				
			\end{enumerate}
		\end{enumerate} 	
	\end{thm}
	
	\medskip
	
	\begin{rmk}
		In particular, for $f=\chi_\Omega$ we have  $$|R[\chi_\Omega](x)|\le C\, \biggl(\chi_{U_{R_0}\cup\Omega}(x)+\frac{\chi_{\R^n\setminus (U_{R_0}\cup\Omega)}(x)}{\delta(x)^n}\biggr).$$
	\end{rmk}
	
	The proof this theorem is the subject of the forthcoming subsections.

	\subsubsection{Uniform estimates.} Using the decomposition given in Proposition \ref{desing} we estimate $Q[f]$ and $L[f]$ in different situations depending on the support of $f$ and the position of the point $x$. The proof follows the argument for the same result in \cite{BurMat} and we will use this reference as a guideline and give a sketch of ideas and main calculations necessary in the actual case.

	{\bf 1) Estimates for $L$:}
	
	\begin{prop} 
		If $x\in\Omega\cup\overline{U_{R_0}}$, there exists a constant $C=C(\gamma,\, \Omega)>0$ such that
		$$
		|L[\phi](x)|\le C\, \|\phi\|_\gamma
		$$ 
		and
		$$
		|L[\psi](x)|\le C\, \|\psi\|_\gamma.
		$$

	\end{prop}

	\begin{proof} 
		We use the compactness of the support of $\phi$ plus Lemmas \ref{Aux1} and \ref{Aux2} below. Actually for $x\in\overline{U_{R_0}}$ let $\tau\in\partial\Omega$ such that $d(x)=\|x-\tau\|$ and  $\kappa_\tau$ defining the half space determined by the tangent line to $\partial\Omega$ across $\tau$ and 
		containing the inward normal vector to $\partial\Omega$ in $\tau$. Then the integral $L[f](x)$ can be decomposed in integrals over the domains resulting from the decomposition   $$(B_{\delta(x)}(x)\cap\{\kappa_\tau\le0\})\cup(B_{\delta(x)}(x)\cap\{\kappa_\tau\geq0\})$$
		$$=(B_{\delta(x)}(x)\cap\{\kappa_\tau\le0\}\cap\{\rho\le0\})
		\cup(B_{\delta(x)}(x)\cap\{\kappa_\tau\le0\}\cap\{\rho>0\})$$
		$$\cup(B_{\delta(x)}(x)\cap\{\kappa_\tau\geq0\}\cap\{\rho\le0\})\cup(B_{\delta(x)}(x)\cap\{\kappa_\tau\geq0\}\cap\{\rho>0\}.$$

		If $f=\phi$, in the case of $x\in\Omega$, we use the $\gamma$-Lipschitz bound to obtain weakly singular functions in the integrals performed on  
		$\{\rho\le0\}$ to obtain the bound $C\, \|\phi\|_\gamma
		$. In the other cases $\{\rho\l>0\}$ we obtain the bound $C\, \|\phi\|_\infty$ using the aforementioned lemmas.

	\end{proof}	
	
	\begin{lem}\label{Aux1} 
		Under the conditions and notation of Theorem \ref{thmR}, and assuming that $\tau\in\partial\Omega$ is a point minimizing the distance of the points in $\partial\Omega$ to the point $x$, the integrals 
		$$ 
		\int_{B_{\delta(x)}(x)\cap\{\kappa_\tau\le0\}\cap\{\rho>0\}} (R)_x
		$$ 
		and 
		$$
		\int_{B_{\delta(x)}(x)\cap\{\kappa_\tau\geq0\}\cap\{\rho\le0\}} (R)_x
		$$ 
		are uniformly bounded.
		
	\end{lem}
	
	\begin{proof} 
		Taking coordinates centered at $\tau$ given by the frame $\eta=\eta(\tau)=\frac{\nabla\rho(\tau)}{\|\nabla\rho(\tau)\|}$, and
		$v_1,\, \dots,\, v_{n-1}$, complementing the first vector to an orthonormal basis, we take advantage to the fact that after an affine isometric change of coordinates $B_{\delta(x)}(x)\cap\{\rho=0\}$ can be 
		described by the graph of a function on $B_{\delta(x)}(x)\cap\{\kappa_\tau=0\}$. This function is in the $Lip(\gamma)$ class and this fact together with the cancellation property cause the kernel to be weakly singular. Further rather direct computation give the estimates.
		
	\end{proof}	
	
	\begin{lem}\label{Aux2} 
		$$
		\biggl|\int_{B_{\delta(x)}(x)\cap\{\kappa_\tau\geq0\}}(R_{j,\, i})_x
		\biggr|\le C(n).
		$$
	\end{lem}
	
	\begin{proof} 
		The cancellation property is used to reduce the estimate to the case of $$\int_{B_{\frac{R_0}{2}}(0)\cap\{y_n\geq\alpha\}}(R_{j,\, i})_0$$ and then Stokes theorem and direct computations provide the estimates.  
		
	\end{proof}
	
	\begin{prop} 
		If $x\in\R^n\setminus\overline{U_{R_0}\cup\Omega}$ then we have
		$$
		L[\phi](x)=0,
		$$ 
		and there exists a constant $C=C(\gamma,\, \Omega,\, R_0)$ such that 
		\begin{multline*}
			|L[\psi](x)|\le C\, (1+\|\psi\|_\gamma) \, C(\psi)\biggl\{\frac{1}{(\max\{R_0^2,\,d(x)^2\})^\frac{n}{2}}\\*[5pt]
			+\frac{1}{\max\{d(x)),\,  \|x\|\}^n}\, \{1+\ln(\max\{d(x),\,  \|x\|\})\biggr\}.
		\end{multline*} 	
		
		
		Here we are assuming that $\|\psi\|_\gamma>0$ and also that $0\in\Omega\setminus U_{R_0}$.	
	\end{prop}
	
	\begin{proof} 
		The case of $\phi$ is immediate.
		
		
		In the case of $\psi$ the integral is performed on the entirely ball of radius $d(x)$ centered at $x$. The choice of $$\alpha(x)\le(\max\{\frac{|\psi(x)|}{\|\psi\|_\gamma},\, \frac{1}{\max\{R_0^2,\,d(x)^2\}}\})^\frac{1}{\gamma}$$ and the decomposition of the integral in two, one performed in $B_{\alpha(x)}(x)$ and the other in $B_{d(x)}(x)\setminus B_{\alpha(x)}(x)$.
		
		The term corresponding to $B_{\alpha(x)}(x)$ is estimated by $$ \frac{C(n)}{\gamma}\, \max\{|\psi(x)|,\, \frac{\|\psi\|_\gamma}{(\max\{R_0^2,\,d(x)^2\})^\frac{n}{2}}\}$$ in a direct way.
		
		The term corresponding to $B_{d(x)}(x)\setminus B_{\alpha(x)}(x)$ can be decomposed in three integrals. One over $B_{d(x)}(x)\cap U_{\frac{R_0}{2}}$, another over $B_{d(x)}(x)\cap (U_{R_0}\setminus U_{\frac{R_0}{2}})$ and the third one over $(B_{d(x)}(x)\setminus B_{\alpha(x)}(x))\setminus U_{R_0}$. Frome those the first one has an estimate of type $\frac{\|\psi\|_\infty}{(\max\{R_0^2,\,d(x)^2\})^\frac{n}{2}}$ also in a direct way. The second one has the same type of estimate whenever $d(x)\geq\frac{3\, R_0}{2}$.
		
		If $R_0<d(x)<\frac{3\, R_0}{2}$ the second term is estimated in terms of $\frac{\|\psi\|_\gamma}{(\max\{R_0^2,\,d(x)^2\})^\frac{n}{2}}.$
		
		For the third term in the previous decomposition one uses the estimate 
		$|\psi(x)|\le\frac{C(\psi)}{(\max\{R_0^2,\,d(x)^2\})^\frac{n}{2}}$ in the new decomposition: $$(B_{d(x)}(x)\setminus B_{\alpha(x)}(x))\setminus U_{R_0}=\biggl(B_{\frac{d(x)}{2}}(x)\setminus B)_{\alpha(x)}(x)\biggr)\cup\biggl(B_{d(x)}(x)\setminus B_{\frac{d(x)}{2}}(x)\, )\setminus U_{R_0}\biggr).$$ The integral in the first set in the decomposition is estimated by  
		$$\frac{C(\psi)}{d(x)^n}\, \{\ln(d(x)\, \|\psi\|_\gamma^{\frac{1}{\gamma}})+\, \ln(2\, |\psi(x)|^{\frac{1}{\gamma}})\},
		$$ in the case of $\frac{|\psi(x)|}{\|\psi\|_\gamma}\geq\frac{1}{(\max\{R_0^2,\,d(x)^2\})^\frac{n}{2}},$
		and by 
		$$\frac{C(\psi)}{d(x)^n}\, \{\ln(d(x))+\ln(2\, \max\{R_0^2,\,d(x)^n\}^{\frac{1}{\gamma}}\}$$ in the opposite case.

		For the term $(VIII)$, we consider 
		$$
		(B_{d(x)}(x)\setminus B_{\frac{d(x)}{2}}(x)\, )\setminus U_{R_0}=E_1(x)\cup E_2(x),
		$$ 
		where $E_1(x)$ and $E_2(x)$ are the intersection of the domain of integration with the set 
		$
		\{\zeta\in\R^n:\, d(\zeta)\geq
		\|\zeta-x\|\}
		$ 
		or its complement, respectively. Then the integral over $E_1(x)$ is bounded directly by 
		$\frac{C(\psi)}{d(x)^n}.
		$
		
		
		For the integral over $E_2(x)$ we have to take in account the position of $x$ with respect to $\Omega$. For doing so, let 
		$
		M=\max\{\|y\|:y\in\overline{(U_{R_0}\cup\Omega)}\}
		$ and then consider the decomposition $$E_2(x)=(E_2(z)\cap B_{2M}(0)\, )\cup(E_2(z)\setminus B_{2M}(0)\, ).$$
		
		And the integral over the first set above can be directly estimated by
		$\frac{C(\psi)}{d(x)^n}.$
		
		The integral over the second set can be controlled by  
		$$
		C(\psi)\,  \int_{(B{d(x)}(x)\setminus B_{\frac{d(z)}{2}}(x)\, )\setminus B_{2M}(0)}\frac{dm(\zeta)}{\|\zeta\|^n\, \|\zeta-x\|^n}.
		$$ and then, in the case $\|x\|\geq 5M$, we have $d(x)\geq 3M$ and for $r<3M$, the balls $B_{2M}(0)$ and $B_{\frac{d(z)}{2}}(x)$ are 
		mutually disjoint and we consider then the decomposition in disjoint sets $(B{d(x)}(x)\setminus B_{\frac{d(z)}{2}}(x)\, )\setminus B_{2M}(0)=A_1\cup A_2\cup A_3,
		$ 
		where
		\begin{align*}
			A_1&=C^{d(x)}_{\frac{d(x)}{2}}(x)\setminus B_{2\|x\|}(0),\\*[5pt]
			A_2&=\{\zeta\in B_{2\|x\|}(0)\setminus B_{2M}(0):\, \|\zeta\|\le \|\zeta-x\|\},
			\intertext{and}
			A_3&=\{\zeta\in B_{2\|x\|}(0)\setminus B_{\frac{d(x)}{2}}(x):\, \, \|\zeta\|\geq \|\zeta-x\|\}.
		\end{align*}
		
		Then the integral over
		$A_1$ is bounded by 
		$
		\frac{C}{\|x\|^n},
		$ the integral over $A_2$ is bounded by 
		$
		\frac{C}{\|x\|^n}\, \ln\biggl(\frac{\|x\|}{M}\biggr)
		$ and the integral in $A_3$ is bounded by 
		$\frac{C}{\|x\|^n}\, \ln\biggl(C\, \frac{\|x\|}{d(x)}\biggr).$

		If $\|x\|\leq 5M$, then the set 
		$(B{d(x)}(x)\setminus B_{\frac{d(z)}{2}}(x)\, )\setminus B_{2M}(0)=A_1\cup A_2\cup A_3,
		$ decomposes in a disjoint union of the resulting intersection with the set 
		$
		\{\|\zeta\|>\|\zeta-x\|\}
		$ 
		and its complement, namely $G_1$ and $G_2$. The corresponding integrals are bounded by $\frac{C}{d(x)^n}
		$
		and 
		$
		\frac{C}{\|x\|^n}
		$ respectively.

	\end{proof}

	2) {\bf Estimates for $Q$:}

	\begin{prop} 
		There exists a constant $C=C(\Omega,\, R_0)$ such that if $x\in\Omega\cup U_{R_0}$, we have 
		$$
		|Q[\phi](x)|\le C\, \|\phi\|_\infty
		$$ 
		and 
		$$
		|Q[\psi](x)|\le C\, (\|\psi\|_\infty+C(\psi)).
		$$	
	\end{prop}
	
	\begin{proof} We consider two cases:
		
		A) In the case of $x\in\Omega\setminus U_{R_0}$, we have $\delta(x)=d(x)\geq R_0$ and 
		then $Q[\phi](x)$ is controlled in terms of $\|\phi\|_\infty$.
		
		For $Q[\psi](x)$ the domain of integration $\R^n\setminus B_{d(x)}(x)$ can be changed for free in the domain $\R^n\setminus\Omega= \biggl((\R^n\setminus\Omega)\cap U_{R_0}\biggr)\cup\biggl(\R^n\setminus(\Omega\cup U_{R_0})\biggr)$
		and the integral over the first domain has the same control as $Q[\phi](x)$.
		
		
		The integral over the second one is controlled by 
		$$
		C(\psi)\, \int_{\R^n\setminus(\Omega\cup U_{R_0})}\frac{dm(\zeta)}{\max\{R_0^2,\,d(\zeta)^2\}^\frac{n}{2}\, \|\zeta-x\|^n},
		$$ 
		and we have the decomposition 
		$
		\R^n\setminus(\Omega\cup U_{R_0})=E_1\cup E_2,
		$ 
		where $E_1$ is the intersection of $\R^n\setminus(\Omega\cup U_{R_0})$ with 
		$
		\{d(\zeta)\geq \|\zeta-x\|\},
		$ 
		and $E_2$ with the complement.
		
		In both cases the integral is bounded by constants depending only on the geometry of $\Omega$ and the corresponding term is controlled by $C(\psi).$

		B) In the case of $x\in U_{R_0}$, we have the decomposition
		$$\R\setminus B_{R_0}(x)=\biggl(\R\setminus (B_{R_0}(z)\cup U_{R_0}\cup\Omega)\biggr)\cup\biggl((U_{R_0}\cup\Omega)\setminus B_{R_0}(z)\biggr).$$

		
		In the case of $f=\phi$, the integral over the first region is equal to  
		$0$. In the case of $f=\psi$ is bounded in terms of $C(\psi)$.
		
		The integral over the second region is bounded in terms of  
		$\|f\|_\infty
		$ 
		in all cases.
		

	\end{proof}
	
	\begin{prop} 
		There exists a constant $C=C(\Omega,\, R_0)$ such that if  $x\in\R^n\setminus\overline{(\Omega\cup U_{R_0})}$, then 
		$$
		|Q[\phi](x)|\le\, C\, \|\phi\|_\infty\, \frac{1}{\max\{R_0^2,\,d(x)^2\}^\frac{n}{2}}
		$$ 
		and 
		$$
		|Q[\psi](x)|\le\, C\biggl \{\|\psi\|_\infty\, \frac{1}{\delta(x)^n}+C(\psi)\, \frac{1}{\max\{\|x\|^2,\,d(x)^2\}^\frac{n}{2}}\, (1+\ln\max\{\|x\|,\,d(x)\})\biggr\}.
		$$	
	\end{prop}
	
	\begin{proof}
		If $x\in\R^n\setminus(\Omega\cup U_{R_0})$, then the integral for $\phi$ is in fact an integral on 
		$
		\Omega\setminus B_{\delta(x)}(x)
		$ 
		and is bounded in terms of 
		$
		\|\phi\|_\infty\, \frac{1}{\delta(x)^n},
		$ 
		leading to the statement.
		
		
		In the case of $\psi$ we use the decomposition of the domain \newline
		$\R^n\setminus(\Omega\cup B_{\delta(x)}(x))=\biggl(\R^n\setminus(\Omega\cup U_{R_0}\cup B_{\delta(x)}(x))\biggr)\cup\biggl(U_{R_0}\setminus(\Omega\cup B_{\delta(x)}(x))\biggr)$

		
		The integral over the second region is controlled in terms of 
		$
		\|\psi\|_\infty\, \frac{1}{\delta(x)^n}.
		$
		
		
		For integral over the first region we have the control
		$$C(\psi)\, \int_{\R^n\setminus(\Omega\cup U_{R_0}\cup B_{\delta(x)}(x))}\frac{dm(\zeta)}{\max\{R_0^2,\,d(\zeta)^2\}^\frac{n}{2}\, \|\zeta-x\|^n},
		$$ 
		and we consider 
		$
		\R^n\setminus(\Omega\cup B_{d(x)}(x))=E_1(x)\cup E_2(x),
		$ 
		where $E_1(x)$ and $E_2(x)$ are the intersection of the domain of integration with the set \newline
		$
		\{\zeta\in\R^n:\, d(\zeta)\geq \|\zeta-x\|\}
		$ 
		or its complement. Then the integral over $E_1(x)$ is controlled in terms of 
		$
		\frac{1}{d(x)^n}.
		$
		
		
		The integral over $E_2(x)$ is more involved. We use the decomposition \newline
		$
		E_2(x)=(E_2(z)\cap B_{2M}(0)\, )\cup (E_2(z)\setminus B_{2M}(0)\, )
		$. The integral over the second region above is controlled by
		$$
		\int_{\R^n\setminus(B_{2M}(0)\cup B_{\frac{d(x)}{2}}(x))}\frac{dm(\zeta)}{\|\zeta\|^n\, \|\zeta-x\|^n}.
		$$
		
		
		For the other integrals, if $\|x\|\geq 5M$, then arguing as in previous proofs in this section, we consider the decomposition in disjoint sets \newline 
		$
		\R^n\setminus(B_{2M}(0)\cup B_{\frac{d(x)}{2}}(x))=A_1\cup A_2\cup A_3,
		$ 
		where
		\begin{align*}
			A_1&=\R^n\setminus B_{2\|x\|}(0),\\*[5pt]
			A_2&=\{\zeta\in B_{2\, \|x\|}(0)\setminus B_{2M}(0):\, \|\zeta\|\le \|\zeta-x\|\},
			\intertext{and} 
			A_3&=\{\zeta\in B_{2\, \|x\|}(0)\setminus B_{\frac{d(x)}{2}}(x):\, \, \|\zeta\|\geq \|\zeta-x\|\},
		\end{align*}
		and the integral over $A_1$ is bounded in terms of $\frac{1}{\|x\|^n},
		$. For the integral over $A_2$, we have the bound in terms of  
		$
		\frac{1}{\|x\|^n}\, \ln\biggl(\frac{\|x\|}{M}\biggr),
		$ and for $\zeta\in A_3$ we have the estimate in terms of  $ 
		\frac{1}{\|x\|^n}\, \ln\biggl(\frac{8\, \|x\|}{d(x)}\biggr).$

		If $\|x\|\leq 5M$, then the set 
		$
		\R^n\setminus(B_{2M}(0)\cup B_{\frac{d(x)}{2}}(x))
		$ 
		decomposes in a disjoint union of the resulting intersection with the set 
		$
		\{\|\zeta\|>\|\zeta-x\|\}
		$ 
		and its complement, namely $G_1$ and $G_2$. The integrals over these regions are bounded respectively by 
		$
		\frac{1}{d(x)^n}
		$ 
		and 
		$
		\frac{1}{\|x\|^n}.
		$

		Finally, the integral over 
		$E_2(z)\cap B_{2M}(0)$ is bounded in terms of $\frac{1}{d(x)^n}$ in all cases.
		
	\end{proof}

	\subsection{Lipschitz estimates} 
	
	We assume the same hypotheses, definitions and notation of the previous sections and subsections.
	

	\begin{prop}\label{thm2} Consider $\phi$ and~$\psi$ satisfying the conditions of Theorem \ref{thmR}.
		
		There exists a constant $C=C(\gamma,\, \Omega,\, R_0)$ such that for $f=\phi$ or $\psi$  and $(x,y)\in\Omega\times\Omega$, or $(x,y)\in(\R^n\setminus\bar\Omega)\times(\R^n\setminus\bar\Omega)$ we have that,   
		$$
		\frac{| R[f](x)- R[f](y)|}{\|x-y\|^\gamma}\le C\, \|f\|_\gamma.
		$$ 
	\end{prop}

	\begin{rmk} As we will see,
		the theorem implies that $R_{j,\, i}[f]$ have a Lipschitz extension to $\bar\Omega$ and also to $\R^n\setminus\Omega$ but these extensions differ in a jump along $\partial\Omega$.
	\end{rmk}
	
	\subsubsection{Proof of Proposition \ref{thm2}:} Again we follow the ideas of the corresponding statement in  \cite{BurMat}. 
	
	We will estimate the quotients in the left hand side of the inequality in the statement above, only in the case of $\|x-y\|\le\frac{R_0}{4}$. 
	
	Most of the proof goes along for $\phi$ or $\psi$ with no distinction, so unless it be necessary, we will use $f$ for $\phi$ or $\psi$.

	The goal is estimating $\frac{| R[f](x)- R[f](y)|}{\|x-y\|^\gamma}$ for $x,\, y\in\Omega$ and for $x,\, y\in\R^n\setminus\bar\Omega$, using the decomposition of the singularities in Proposition \ref{desing}. In fact it is enough to estimate 	
	$$\frac{| Q[f](x)- Q[f](y)|}{\|x-y\|^\gamma}
	+\frac{| L[f](x)- L[f](y)|}{\|x-y\|^\gamma}$$ because the boundedness of the $\gamma$-Lipschitz norm of the boundary term is easier and is left to the end of the section.


	
	The first term
	\begin{equation*}
		\begin{split}
			Q[f](x)-Q[f](y)
			&=\int_{\R^n\setminus (B_{\delta(x)}(x)\cup B_{\delta(y)}(y))}f\, \{(R)_x-(R)_y\}\\*[5pt]
			&\quad+
			\int_{B_{\delta(y)}(y)\setminus B_{\delta(x)}(x)}f\, (R)_x-\!\int_{B_{\delta(x)}(x)\setminus B_{\delta(y)}(y)}f\,(R)_y
			\!=\!(1)+(2)\!
		\end{split}
	\end{equation*}
	and he second term
	\begin{equation*}
		\begin{split}
			L[f](y)-L[f](x)
			&=\int_{B_{\delta(x)}(x)\cap B_{\delta(y)}(y)}\biggl[\frac{f(\zeta)-f(y)}{\|\zeta-y\|^\gamma}\, 
			\|\zeta-y\|^\gamma\, (R)_y(\zeta)\\*[5pt]
			&\hspace*{4cm}-\frac{f(\zeta)-f(x)}{\|\zeta-x\|^\gamma}\, 
			\|\zeta-x\|^\gamma\, (R)_x(\zeta)\biggr]\\*[5pt]
			&\quad+\int_{B_{\delta(y)}(y)\setminus B_{\delta(x)}(x)}\frac{f(\zeta)-f(y)}{\|\zeta-y\|^\gamma}\, 
			\|\zeta-y\|^\gamma\, (R)_y(\zeta)\\*[5pt]
			&\quad-
			\int_{B_{\delta(x)}(x)\setminus B_{\delta(y)}(y)}\frac{f(\zeta)-f(x)}{\|\zeta-x\|^\gamma}\, 
			\|\zeta-x\|^\gamma\, (R)_x(\zeta)\\*[5pt]
			&=(3)+(4).
		\end{split}
	\end{equation*}
	
	\noindent
	{\bf 1.} The integral $(1)$, after a straightforward computation and use of the mean value theorem we have  
	$$(1)=\|y-x\|\, \int_{\R^n\setminus (B_{\delta(y)}(y)\cup B_{\delta(x)}(x))}f(\zeta)\, (\nabla_x(R)_{\alpha\, x+(1-\alpha)\, y}(\zeta),\, \frac{x-y}{\|y-x\|}\, )\, dm(\zeta)$$
	and if we consider $\tau=\frac{x+y}{2}$ and then $\tau-x
	=\frac{y-x}{2}\overset{\text{def}}{=}a$, $\tau-y=-\frac{y-x}{2}=-a$ and 
	$$
	\lambda=\sqrt{\frac{\delta(x)^2-\|a\|^2+\delta(y)^2-\|a\|^2}{2}},
	$$ 
	then $B_{\lambda}(\tau)\subset B_{\delta(x)}(x)\cup B_{\delta(y)}(y)$, and, since $f$ is a bounded function, then the integral above is
	bounded in terms of  
	$$\|f\|_\infty\,  \int_\lambda^\infty \frac{1}{(r-\|a\|)^{2}}\, dr\simeq\, \|f\|_\infty\, \frac{1}{\lambda-\|a\|}\simeq\frac{C\, \|f\|_\infty}{R_0},$$ for  $\lambda-\|a\|>5\, \|a\|\geq C\, R_0$.

For the remaining terms, both in the decomposition of $Q$ and $L$,  the absolute and mutual positions of $x$ and $y$, specially in relation to $\partial\Omega$, will play an important role. For this purpose, we will consider the following situations:

\begin{itemize}
	\item[1)] The cases of $B_{\delta(y)}(y)\subset B_{\delta(x)}(x)$, or $B_{\delta(x)}(x)\subset B_{\delta(y)}(y)$, corresponding respectively to the facts that $\delta(y)+\|x-y\|\le\delta(x)$ or $\delta(x)+\|x-y\|\le\delta(y)$, so is 
	$$
	\frac{\|x-y\|}{\delta(x)+\delta(y)}\le 1-\frac{2\, \delta(y)}{\delta(x)+\delta(y)}
	$$ 
	or 
	$$\frac{\|x-y\|}{\delta(x)+\delta(y)}\le 1-\frac{2\, \delta(x)}{\delta(x)+\delta(y)},
	$$ 
	respectively.
	
	The situations are completely symmetric and in each case only one term in $(2)$ survives.

	\item[2)] The case where the conditions 
	$$
	\begin{cases} 
		B_{\delta(y)}(y)\cap B_{\delta(x)}(x)\neq\emptyset ,\\ 
		B_{\delta(y)}(y)\cap B_{\delta(x)}(x)^c\neq \emptyset,
	\end{cases}
	$$ 
	or reciprocally, are both satisfied.
	
	In this case, we have 
	$$
	\begin{cases} \delta(x)+\delta(x)\geq\|x-y\|, \\
		\delta(x)\le\|x-y\|+\delta(y),
	\end{cases}
	$$ 
	or 
	$$
	1-\frac{2\, \delta(y)}{\delta(x)+\delta(y)}\le\frac{\|x-y\|}{\delta(x)+\delta(y)}\le 1,
	$$ 
	and, simultaneously, 
	$$
	1-\frac{2\, \delta(x)}{\delta(x)+\delta(y)}\le\frac{\|x-y\|}{\delta(x)+\delta(y)}\le 1.
	$$ 
\end{itemize}

\noindent
{\bf 2.} For the term $(2)$, let us consider first the case 1). WLOG we assume that we are in the first situation of this case. Then 
$$
(2)=\int_{B_{\delta(y)}(y)\setminus B_{\delta(x)}(x)}f\, (R)_x.
$$

\begin{itemize}
	
	\item If $x\notin U_{\frac{R_0}{2}}$ and  
	\begin{equation*}
		\begin{split}
			(2)&=\int_{B_{d(y)}(y)\setminus B_{d(x)}(x)}\frac{f(\zeta)-f(x)}{\|\zeta-x\|^\gamma}\, 
			\|\zeta-x\|^\gamma\, (R)_x\, dm(\zeta),
		\end{split}
	\end{equation*}
	by the cancellation property of the Riesz kernel in a ball, and since $B_{\delta(y)}(y)\subset\Omega$ or $B_{\delta(y)}(y)\subset\bar\Omega^c$, we have the bound  $C\, \|f\|_\gamma\, \|y-x\|.$

	\item If $x\in U_{\frac{R_0}{2}}$, we have that $\delta(x)=\frac{R_0}{2}$ and then $(2)$ is bounded by $$C\, \|f\|_\gamma\, (d(y)-d(x)+\|x-y\|).$$

	In the situation of the case 2) 
	$$
	(2)=
	\int_{B_{\delta(y)}(y)\setminus B_{\delta(x)}(x)}f\, (R)_x-\int_{B_{\delta(x)}(x)\setminus B_{\delta(y)}(y)}f\,(R)_y.
	$$ 
	
	\begin{itemize}
		\item If $x,y\notin U_{\frac{R_0}{2}}$ we have that $x,y\in\Omega$ or $x,y\in\bar\Omega^c$ are the only possibilities, and also that $\delta=d$, so the first term of $(2)$ is estimated by 
		$$
		C\, \|f\|_{L^\infty}\, \|x-y\|\, \frac{d(x)+\|x-y\|+d(y)}{d(x)}\le C\, \|f\|_{L^\infty}\, \|x-y\|.
		$$

		
		The other term is similar.

		\item If $y\notin U_{\frac{R_0}{2}}$, but $x\in U_{\frac{R_0}{2}}$, we have that 
		$$
		(2)=\int_{B_{d(y)}(y)\setminus B_{\frac{R_0}{2}}(x)}f\, (R)_x-\int_{B_{\frac{R_0}{2}}(x)\setminus B_{d(y)}(y)}f\, (R)_y,
		$$ 
		and we consider several subcases, and in each one working separately the cases $f=\phi$ and $f=\psi$.  
		
		\begin{itemize} 
			\item If $x,y\in\Omega$, then $(2)$ is
			$$\int_{B_{d(y)}(y)\setminus B_{\frac{R_0}{2}}(x)}\phi\, (R)_x-\int_{B_{\frac{R_0}{2}}(x)\setminus B_{d(y)}(y)}\phi\, (R)_y,$$
			or
			$$
			-\int_{B_{\frac{R_0}{2}}(x)}\psi\, (R)_y.
			$$ For the first two integrals the bound is 
			$$\frac{C\, \|\phi\|_\infty\, (\dmt(\Omega)+R_0)}{R_0^2}\, \|y-x\|
			$$  and for the third one is
			$$\|\psi\|_{\infty}\, \frac{C\, (\dmt(\Omega)+R_0)}{R_0^2}\, (d(y)-d(x))$$ because $d(y)\geq\frac{R_0}{2}$.

			\item If $x,y\in\Omega^c$, then $(2)$ is
			$$
			-\int_{B_{\frac{R_0}{2}}(x)}\phi\, (R)_y,
			$$ 
			or 
			$$
			\int_{B_{d(y)}(y)\setminus B_{\frac{R_0}{2}}(x)}\psi\, (R)_x-\int_{B_{\frac{R_0}{2}}(x)\setminus B_{d(y)}(y)}\psi\, (R)_y.
			$$	
			
			
			This case is similar to the previous one exchainging the roles of $\phi$  and $\psi$, also in the estimates.


			\item If $x,y\in U_{\frac{R_0}{2}}$, we have that 
			$$
			(2)=
			\int_{B_{\frac{R_0}{2}}(y)\setminus B_{\frac{R_0}{2}}(x)}f\, (R)_x-\int_{B_{\frac{R_0}{2}}(x)\setminus B_{\frac{R_0}{2}}(y)}f\, (R)_y,
			$$ 
			and then 
			$$
			|(2)|\le2\, \frac{4\, \|f\|_{L^\infty}}{R_0^2}\, m(B_{\frac{R_0}{2}}(y)\setminus B_{\frac{R_0}{2}}(x))
			$$ 
			and the we can apply the previous procedure. 
		\end{itemize}
	\end{itemize}
\end{itemize}

\noindent
{\bf 3.} Now, the term 
$$
(3)\!=\!\int_{B_{\delta(y)}(y)\cap B_{\delta(x)}(x)\!}\biggl[\frac{f(\zeta)\!-\!f(x)}{\|\zeta-x\|^\gamma}\, 
\|\zeta-x\|^\gamma\, (R)_x-\frac{f(\zeta)\!-\!f(y)}{\|\zeta-y\|^\gamma}\, 
\|\zeta-y\|^\gamma\, (R)_y\biggr],\!\!\!
$$ where $x,\, y\in W$.



In such case, $d(x,\partial B_{\delta(y)}(y))=\delta(y)-\|x-y\|$ and $d(y,\partial B_{\delta(x)}(x))=\delta(x)-\|x-y\|$, so the maximum radius of a ball centered at $x$ and contained in $B_{\delta(x)}(x)\cap B_{\delta(y)}(y)$ is equal to $\min\{\delta(x),\,\delta(y)-\|x-y\|\}$. For the point $y$ we have an analogous expression. And since $\delta(x)\geq\frac{R_0}{2}\geq2\, \|x-y\|$, and analogously for $\delta(y)$, we define 
\begin{align*}
	k^\gamma_x(\zeta)&=\|\zeta-x\|^\gamma,\\
	\Delta^\gamma_x (\zeta)&=\frac{f(\zeta)-f(x)}{k_x^\gamma(\zeta)},
\end{align*} 
and then we consider the following decomposition:
$$(3)=\int_{B_{\delta(y)}(y)\cap B_{\delta(x)}(x)}[\Delta^\gamma_x\, k^\gamma_x\, (R)_x-\Delta^\gamma_y\, k^\gamma_y\, 
(R)_y]$$
$$=\biggl\{\int_{B_{\frac{\|x-y\|}{2}}(x)\cap W}\Delta^\gamma_x\, k^\gamma_x\, (R)_x-\int_{B_{\frac{\|x-y\|}{2}}(y)\cap W}\Delta^\gamma_y\, k^\gamma_y\, (R)_y\biggr\}$$
$$\quad+\biggl\{\int_{(B_{\delta(y)}(y)\cap B_{\delta(x)}(x)\cap W)\setminus B_{\frac{\|x-y\|}{2}}(x)}\Delta^\gamma_x\, k^\gamma_x\, (R)_x\\*[5pt]$$
$$\quad-\int_{(B_{\delta(y)}(y)\cap B_{\delta(x)}(x)\cap W)\setminus B_{\frac{\|x-y\|}{2}}(y)}\Delta^\gamma_y\, k^\gamma_y\, (R)_y\biggr\}\\*[5pt]$$
$$+\biggl\{\int_{B_{\frac{\|x-y\|}{2}}(x)\cap W^c}\Delta^\gamma_x\, k^\gamma_x\, (R)_x-\int_{B_{\frac{\|x-y\|}{2}}(y)\cap W^c}\Delta^\gamma_y\, k^\gamma_y\, (R)_y\biggr\}\\*[5pt]$$
$$\quad+\biggl\{\int_{(B_{\delta(y)}(y)\cap B_{\delta(x)}(x)\cap W^c)\setminus B_{\frac{\|x-y\|}{2}}(x)}\Delta^\gamma_x\, k^\gamma_x\, (R)_x\\*[5pt]$$
$$\quad-\int_{(B_{\delta(y)}(y)\cap B_{\delta(x)}(x)\cap W^c)\setminus B_{\frac{\|x-y\|}{2}}(y)}\Delta^\gamma_y\, k^\gamma_y\, (R)_y\biggr\}\\*[5pt]$$
$$=(31)+(32)+(33)+(34).$$

The term $(31)$ is directly estimated in terms of  $$ \|f\|_\gamma\, \|x-y\|^\gamma
.$$ 

Also since for $\tau=\frac{x+y}{2}$ we have $B_{\frac{\|x-y\|}{2}}(x)\cup B_{\frac{\|x-y\|}{2}}(y)\subset B_{\|x-y\|}(\tau)$, then
\begin{equation*}
	\begin{split}
		(32)&=\{\int_{(B_{\|x-y\|}(\tau)\cap W)\setminus B_{\frac{\|x-y\|}{2}}(x)}\Delta^\gamma_x\, k^\gamma_x\, (R)_x\\*[5pt]
		&-\int_{(B_{\|x-y\|}(\tau)\cap W)\setminus B_{\frac{\|x-y\|}{2}}(y)}\Delta^\gamma_y\, k^\gamma_y\, (R)_y\}\\*[5pt]
		&\quad+\int_{(B_{\delta(y)}(y)\cap B_{\delta(x)}(x)\cap W)\setminus B_{\|x-y\|}(\tau)}[\Delta^\gamma_x\, k^\gamma_x\, (R)_x-\Delta^\gamma_y\, k^\gamma_y\, (R)_y]\\*[5pt]
		&=(321)+(322),
	\end{split}
\end{equation*}
and the term 
$(321)$ has the same estimate than $(31)$ in a direct way.

For the terms $(322)$, $(33)$ and $(34)$ we will consider different situations related to the size of the balls and the distance to the boundary of $W$.

\begin{itemize}
	\item[A)] If $B_{\delta(x)}(x)\cap B_{\delta(y)}(y)\subset W$, then $$(322)=\int_{(B_{\delta(y)}(y)\cap B_{\delta(x)}(x))\setminus B_{\|x-y\|}(\tau)}[\Delta^\gamma_x\, k^\gamma_x\, (R)_x-\Delta^\gamma_y\, k^\gamma_y\, (R)_y]$$
	and we analyze the original expression 
	$$
	\Delta^\gamma_x\, k^\gamma_x\, (R)_x-\Delta^\gamma_y\, k^\gamma_y\,(R)_y=(f(\zeta)-f(x))\, (R)_x(\zeta)
	-(f(\zeta)-f(y))\, (R)_y(\zeta)=(*)
	$$ in different situations.

	
	\begin{itemize}
		
		\item[$\bullet$] If $x\notin U_{\frac{R_0}{4}}$, then, as $\|x-y\|<\frac{R_0}{4}$, we have that $y,\tau=\frac{x+y}{2}\in B_{\frac{R_0}{4}}(x)$, 
		and then, using the decomposition
		$$
		(*)\!=\!(f(\zeta)-f(\tau)) ((R)_x-(R)_y)(\zeta)+
		(f(y)-f(\tau))\, (R)_y(\zeta)$$
		$$-(f(x)-f(\tau))\, (R)_x(\zeta),\!\!
		$$
		we have 
		\begin{equation*}
			\begin{split}
				(322)&=\int_{(B_{\delta(y)}(y)\cap B_{\delta(x)}(x))\setminus B_{\|x-y\|}(\tau)}\Delta^\gamma_\tau\, k^\gamma_\tau\, ((R)_x-(R)_y)\\*[5pt]
				&\quad+(f(y)-f(\tau))\, \int_{(B_{\delta(y)}(y)\cap B_{\delta(x)}(x))\setminus B_{\|x-y\|}(\tau)}(R)_y\\*[5pt]
				&\quad-(f(x)-f(\tau))\, \int_{(B_{\delta(y)}(y)\cap B_{\delta(x)}(x))\setminus B_{\|x-y\|}(\tau)}(R)_x\\*[5pt]
				&=(3221)
				+(3222).
			\end{split}
		\end{equation*}
		
		
		For the integral $(3321)$ we use the change of variables 
		$$
		\zeta\rightarrow s=\frac{\zeta-\tau}{\|a\|}=\phi(\zeta),
		$$ 
		where $a=\frac{x-y}{2}$, then $\phi(\tau)=0$ and $J\phi=\|a\|^{-n}$. 
		
		We have 
		\begin{equation*}
			\begin{split}
				k^\gamma_\tau(\zeta)\, |((R)_x-(R)_y)(\zeta)|&\le C\, \|\zeta-\tau\|^\gamma\, \frac{\|x-y\|}{\|\tau'-\zeta\|^{n+1}}\\*[5pt]
				&\le C'\, \frac{\|x-y\|}{\|\tau-\zeta\|^{n+1-\gamma}},
			\end{split}
		\end{equation*} 
		where $\tau'=\alpha\, x+(1-\alpha)\, y$, and as if $\zeta\notin B_{\|x-y\|}(\tau)$ then $$\|\zeta-\tau'\|\geq\min\{\|\zeta-x\|,\, \|\zeta-y\|\}\geq\frac{\|\zeta-\tau\|}{2}.$$
		
		Therefore 
		$$
		|(3221)|\le C''\, \frac{\|x-y\|}{\|a\|^{1-\gamma}}\, \|f\|_\gamma\, \int_{\R^n\setminus B_{2}(0)}\frac{1}{\|s\|^{n+1-\gamma}}\, dm(s).
		$$
		
		\medskip 
		
		For the term $(3322)$, we have 
		
		\begin{lem}\label{unifball} 
			If  $x\notin U_{\frac{R_0}{4}}$ and $\|x-y\|<\frac{R_0}{4}$, then 
			$$
			\biggl|\int_{(B_{\delta(y)}(y)\cap B_{\delta(x)}(x))\setminus B_{\|x-y\|}(\tau)}(R_{j,\, i})_x\biggr|\le C(n) .
			$$
		\end{lem}
		
		\begin{proof} 
			Using Stokes formula, 
			\begin{equation*}
				\begin{split}
					&\int_{(B_{\delta(y)}(y)\cap B_{\delta(x)}(x))\setminus B_{\|x-y\|}(\tau)}(R_{j,\, i})_x\\*[5pt]
					&=\int_{(B_{\delta(y)}(y)\cap B_{\delta(x)}(x))\setminus B_{\|x-y\|}(\tau)}(\frac{\partial K_j}{\partial x_i})_x\, V\\*[5pt]
					&=\int_{(B_{\delta(y)}(y)\cap B_{\delta(x)}(x))\setminus B_{\|x-y\|}(\tau)} d (K_j)_x\wedge\, (-1)^{i-1} \widehat {dx_i}\\*[5pt]
					&=\, (-1)^{i-1}\biggl\{\int_{(\partial B_{\delta(y)}(y))\cap B_{\delta(x)}(x)}\!+\!\int_{B_{\delta(y)}(y)\cap\partial(B_{\delta(x)}(x))}\!-\!\int_{\partial B_{|x-y|}(\tau)}\biggr\}\, (K_j)_x\, \widehat {dx_i}\\*[5pt]
					&=(I)+(II)-(III).
				\end{split}
			\end{equation*}
			
			And we use the change $\zeta=\tau+2\, \|a\|\, \theta$ where $\theta\in S_1^{n-1}(0)$ and we have 
			\begin{equation*}
				\begin{split}
					(III)& \simeq\, \int_{\partial B_{2\, \|a\|}(\tau)}\,  \frac{(\zeta-x)_j}{\|\zeta-x\|^n}\, \widehat {d\zeta_i}\\*[5pt]
					&=(2\, \|a\|)^{n-1}\, \int_{S_1^{n-1}(0)}\, \frac{(\tau-x+2\, \|a\|\, \theta)_j}{\|\tau-x+2\, \|a\|\, \theta\|^n}\, \widehat {d\theta_i}\\*[5pt]
					&=\int_{S_1^{n-1}(0)}\, \frac{(\frac{\tau-x}{2\, \|a\|}+\theta)_j}{\|\frac{\tau-x}{2\, \|a\|}+\theta\|^n}\, \widehat {d\theta_i}.
 				\end{split}
			\end{equation*}

			Then $$|(III)|\lesssim \int_{S_1^{n-1}(0)}\, \frac{1}{(1-\frac{\|a\|}{2\, \|a\|})^{n-1}}\, |\widehat {d\theta_i}|.$$

			
			The integral 
			$$(II)=\int_{B_{\delta(y)}(y)\cap\partial(B_{\delta(x)}(x))}(K_j)_x\, \widehat {d\zeta_i}
			$$ and using the coordinates $\zeta=x+\delta(x)\, \theta$ we have that
			$$(II)=\int_{\{\theta\in S_1^{n-1}(0): \|x+\delta(x)\, \theta-y\|\le\delta(y)\}}\theta_j\, \widehat {d\theta_i},
			$$ and it is uniformly bounded.
			
			
			In the same way, in our situation, 
			$$(I)=\int_{(\partial B_{\delta(y)}(y))\cap B_{\delta(x)}(x)} \frac{(\zeta-x)_j}{\|\zeta-x\|^n}\, \widehat {d\zeta_i}
			$$ and now the change of coordinates $\zeta=y+\delta(y)\, \theta$ we have
			$$(I)=\delta(y)^{n-1}\, \int_{\{\theta\in S_1^{n-1}(0): \|y+\delta(y)\, \theta-x\|\le\delta(x)\}}\frac{(y+\delta(y)\, \theta-x)_j}{\|y+\delta(y)\, \theta-x\|^n}\, \widehat {d\theta_i}$$
			
			Therefore
			$$|(I)|\le (\frac{\delta(y)}{\delta(y)-\|x-y\|})^n\, \int_{S_1^{n-1}(0)}|\widehat {d\theta_i}|
			$$
			and if $d(y)>\frac{R_0}{2}$, then $\|x-y\|\le\frac{R_0}{4}<\frac{d(y)}{2}$ implying that 
			\begin{equation*}
				|(I)|\le 2^n\, \int_{S_1^{n-1}(0)}|\widehat {d\theta_i}|.
			\end{equation*}
			
			If $d(y)\le\frac{R_0}{2}$, then $\delta(y)=\frac{R_0}{2}$ and 
			$$(I)=(\frac{R_0}{2})^{n-1}\, \int_{\{\theta\in S_1^{n-1}(0): \|y+\frac{R_0}{2}\, \theta-x\|\le\delta(x)\}}\frac{(y+\frac{R_0}{2}\, \theta-x)_j}{\|y+\frac{R_0}{2}\, \theta-x\|^n}\, \widehat {d\theta_i},$$ implying that 
			$$
			|(I)|\le(\frac{R_0}{2})^{n-1}\, \int_{\{\theta\in S_1^{n-1}(0): \|y+\frac{R_0}{2}\, \theta-x\|\le\delta(x)\}}\frac{1}{(\frac{R_0}{2}- \|y-x\|)^{n-1}}\, |\widehat {d\theta_i}|$$
			$$\le\, (\frac{\frac{R_0}{2}}{\frac{R_0}{2}-\frac{R_0}{4}})^{n-1}\, \int_{S_1^{n-1}(0)}|\widehat {d\theta_i}|$$

		\end{proof}
		
		The lemma implies that 
		$$
		|(3222)|\le C\, \|x-y\|^\gamma\, \|f\|_\gamma.
		$$ 
		
		\item[ii)] If $B_{\delta(y)}(y)\cap B_{\delta(x)}(x)\cap W^c\neq\emptyset$, then we consider the following cases: 
		
	\end{itemize}

	\item[$\bullet$] If $y\notin U_{\frac{R_0}{4}}$, the situation is completely equivalent to the previous one.

	The discussion above also covers completely the cases of $x\notin U_{\frac{R_0}{2}}$ or $y\notin U_{\frac{R_0}{2}}$ for then $B_{\delta(y)}(y)\cap B_{\delta(x)}(x)\cap W^c=\emptyset$ therefore $$(33)+(34)=0$$ in this case.

	\item[$\bullet$] If $x,\, y\in U_{\frac{R_0}{4}}$ we have $\delta(x)=\delta(y)=\frac{R_0}{2}$.Moreover there exists $p\in W^c$ such that $\|p-x\|\le\frac{R_0}{4}$. This implies that $p\in B_{\frac{R_0}{2}}(x)\cap W^c$ and also $\|p-y\le\|p-x\|+\|x-y\|<\frac{R_0}{2}$ therefore $B_{\delta(y)}(y)\cap B_{\delta(x)}(x)\cap W^c\neq\emptyset$, and this situation does not occur in te case A). 
	
\end{itemize}

If $B_{\delta(y)}(y)\cap B_{\delta(x)}(x)\cap W^c\neq\emptyset$, then $x,\, y\in U_{\frac{R_0}{2}}$ so $\delta(x)=\delta(y)=\frac{R_0}{2}$.
Then both $x$ and $y$ are points of $W$ but eventually the point $\tau=\frac{x+y}{2}$ does not belong to $W$ so we take $\tau'\in \bar W$ at minimum distance of $\tau$. Then \begin{equation*}
	\begin{split}
		(322)&=\int_{(B_{\frac{R_0}{2}}(y)\cap B_{\frac{R_0}{2}}(x))\setminus B_{\|x-y\|}(\tau)\cap W} [(f(\zeta)-f(x)) (R)_x(\zeta)\\*[5pt]
		&\hspace*{5.5cm}
		-(f(\zeta)-f(y)) (R)_y(\zeta)]\\*[5pt]
		&=\int_{(B_{\frac{R_0}{2}}(y)\cap B_{\frac{R_0}{2}}(x)\cap W)\setminus B_{\|x-y\|}(\tau)}(f(\zeta)-f(\tau'))\, [(R)_x(\zeta)-(R)_y(\zeta)]\\*[5pt]
		&\quad+\biggl\{(f(\tau')-f(x))\, \int_{(B_{\frac{R_0}{2}}(y)\cap B_{\frac{R_0}{2}}(x)\cap W)\setminus B_{\|x-y\|}(\tau)}(R)_x(\zeta)\\*[5pt]
		&\quad-(f(\tau')-f(y))\, \int_{(B_{\frac{R_0}{2}}(y)\cap B_{\frac{R_0}{2}}(x)\cap W)\setminus B_{\|x-y\|}(\tau)}(R)_x(\zeta)\biggr\}\\*[5pt]
		&=(3223)+(3224).
	\end{split} 
\end{equation*}

Using the mean value theorem for the function $z\rightarrow R(z-b)$, we have 
$$|(3223)|\le C\, 
\int_{(B_{\frac{R_0}{2}}(y)\cap B_{\frac{R_0}{2}}(x)\cap W)\setminus B_{\|x-y\|}(\tau)}|\triangle^\gamma_{\tau'}(\zeta)|\, k_{\tau'}^\gamma(\zeta)\, \frac{\|x-y\|}{\|\zeta-\lambda\|^{n+1}},
$$ where $\lambda=\alpha\, x+(1-\alpha)\, y$, and as if $\zeta\notin B_{\|x-y\|}(\tau)$ then $$\|\zeta-\lambda\|\geq\min\{\|\zeta-x\|,\, \|\zeta-y\|\}\geq\frac{\|\zeta-\tau\|}{2}.$$ Then 
$$|(3223)|\le C\, 
\int_{(B_{\frac{R_0}{2}}(y)\cap B_{\frac{R_0}{2}}(x)\cap W)\setminus B_{\|x-y\|}(\tau)}|\triangle^\gamma_{\tau'}(\zeta)|\, k_{\tau'}^\gamma(\zeta)\, \frac{\|x-y\|}{\|\zeta-\tau\|^{n+1}}\, dm(\zeta),
$$

Since $\tau'\in \partial W$ then, 
$$
\|\tau'-\tau\|\le\|\tau-x\|=\frac{\|x-y\|}{2}=\|a\|,
$$ 
then 
$$
\|x-\tau'\|\le\|x-\tau\|+\|\tau-\tau'\|\le 2\, \|a\|
$$ 
and 
$$
\|\zeta-\tau'\|\le\|\zeta-\tau\|+\|\tau-\tau'\|\le\|\zeta-\tau\|+\|a\|,
$$ 
so 
$$
\|\zeta-\tau'\|^\gamma\le(\|\zeta-\tau\|+\|a\|)^\gamma
$$ 
and then $$|(3223)|\le C\, \|f\|_\gamma\, 
\int_{(B_{\frac{R_0}{2}}(y)\cap B_{\frac{R_0}{2}}(x)\cap W)\setminus B_{\|x-y\|}(\tau)}  \frac{(\|\zeta-\tau\|+\|a\|)^\gamma\, \|x-y\|}{\|\zeta-\tau\|^{n+1}}\, dm(\zeta),
$$
$$\le C'\, \|f\|_\gamma\, \|x-y\|\, 
\int_{(B_{\frac{R_0}{2}}(y)\cap B_{\frac{R_0}{2}}(x)\cap W)\setminus B_{\|x-y\|}(\tau)}  \frac{1}{\|\zeta-\tau\|^{n+1-\gamma}}\, dm(\zeta),
$$
$$\le C''\, \|f\|_\gamma\, \|x-y\|\, 
\int_{\|x-y\|}^\infty
\frac{1}{r^{2-\gamma}}\, dr,
$$

\medskip

Also, a repetition of the arguments shows that the term $(3224)$ has the same estimate as the term $(3222)$.

The terms $(33)$ and $(34)$ do not vanish necessarily in this situation and it is convenient to write them as $$\int_{B_{\frac{R_0}{2}}(y)\cap B_{\frac{R_0}{2}}(x)\cap W^c}[f(y)\, (R)_y(\zeta)-f(x)\, (R)_x]\, dm(\zeta)=(35)$$
$$=(f(y)-f(x))\, \int_{B_{\frac{R_0}{2}}(y)\cap B_{\frac{R_0}{2}}(x)\cap W^c}(R)_y(\zeta)\, dm(\zeta)$$
$$+f(x)\, \int_{B_{\frac{R_0}{2}}(y)\cap B_{\frac{R_0}{2}}(x)\cap W^c}[(R)_y-(R)_x](\zeta)\, dm(\zeta)$$
$$(351)+(352).$$

Now, $$|(351)|\le \|f\|_\gamma\, \|x-y\|\, |\int_{B_{\frac{R_0}{2}}(y)\cap B_{\frac{R_0}{2}}(x)\cap W^c}(R)_y(\zeta)\, dm(\zeta)|$$ and

\begin{lem}\label{tralari} 
	If  $x,\, y\in U_{\frac{R_0}{2}}\cap W$ there exists a constant $C(n,\, R_0)>0$ such that 
	$$
	\biggl|\int_{B_{\frac{R_0}{2}}(y)\cap B_{\frac{R_0}{2}}(x)\cap W^c}(R)_y(\zeta)\, dm(\zeta)\biggr|\le C(n,\, R_0).
	$$
\end{lem}

\begin{proof} Since $$B_{\frac{R_0}{2}-\|x-y\|}(y)\subset B_{\frac{R_0}{2}}(y)\cap B_{\frac{R_0}{2}}(x)\subset B_{\frac{R_0}{2}+\|x-y\|}(y)$$ we have $$\int_{B_{\frac{R_0}{2}}(y)\cap B_{\frac{R_0}{2}}(x)\cap W^c}(R)_y(\zeta)\, dm(\zeta)=\int_{B_{\frac{R_0}{2}-\|x-y\|}(y)\cap W^c}(R)_y(\zeta)\, dm(\zeta)$$
	$$+\int_{[B_{\frac{R_0}{2}}(y)\cap B_{\frac{R_0}{2}}(x)\cap W^c]\setminus B_{\frac{R_0}{2}-\|x-y\|}(y)}(R)_y(\zeta)\, dm(\zeta)=(I)+(II)$$ and $$|(II)|\le C(n,\, R_0)$$ for $\|\zeta-y\|$ is bounded from above and from below in terms of $R_0$.
	
	Also if $y_0\in\partial W$ is a point minimizing the distance of $y$ to $W^c$, then $$(I)=\int_{B_{\frac{R_0}{2}-\|x-y\|}(y)\cap\{\rho>0\}}(R)_y(\zeta)\, dm(\zeta)$$
	$$=\int_{B_{\frac{R_0}{2}-\|x-y\|}(y)\cap\{\kappa_{y_0}>0\}}(R)_y(\zeta)\, dm(\zeta)$$
	$$+\biggl\{\int_{B_{\frac{R_0}{2}-\|x-y\|}(y)\cap \{\rho>0,\, \kappa_{y_0}<0\}}(R)_y(\zeta)\, dm(\zeta)$$
	$$-\int_{B_{\frac{R_0}{2}-\|x-y\|}(y)\cap \{\rho<0,\, \kappa_{y_0}>0\}}(R)_y(\zeta)\, dm(\zeta)\biggr\}=(III)+(IV).$$
	
	By the Riesz's geometric lemma (lemma \ref{rszgl}) the term $(IV)$ is bounded.
	
	For the term $(III)$ the translation  
	$t=\zeta-y$ makes $(III)\simeq\, (V)$ where  $$(V)=\int_{B_{\frac{R_0}{2}-\|x-y\|}(0)\cap\{(t,\, \eta(y_0))>d(y)\}}\frac{\|t\|^2\, \delta_{j,\, i}-n\, t_j\,  t_i}{\|t\|^{n+2}}\, dm(t)$$
	$$=\int_{\partial(B_{\frac{R_0}{2}-\|x-y\|}(0)\cap\{(t,\, \eta(y_0))>d(y)\})}\, d\biggl((-1)^{i-1}\, \frac{s_j}{\|s\|^{n}}\, \widehat{ds_i}\biggr).$$ 
	
	Using the notation  $\lambda=\frac{R_0}{2}-\|x-y\|$, $\beta=d(y)$ and $\mu=\sqrt{\lambda^2-\beta^2}$  we have $$(V)=\int_{\partial (B_\lambda(0))\cap\{(t,\, \eta(y_0))>\beta\}}(-1)^{i-1}\, \frac{t_j}{\|t\|^{n}}\, \widehat{dt_i}$$
	$$-\int_{B_\lambda(0)\cap\{(t,\, \eta(y_0))=\beta\})}\, (-1)^{i-1}\, \frac{t_j}{\|t\|^{n}}\, \widehat{dt_i}$$ 
	and consider for $u_1,\, \dots,\, u_{n-1}$ an orthonormal complement of the normal vector $u_n=\eta(y_0)$ the parametrizations $\phi_-,\, \phi_+:B_\mu^{n-1}(0)\rightarrow\R^n$ given by $$\phi_+(s_1,\, \dots,\, s_{n-1})= \sum_{k=1}^{n-1}s_k\, u_k+(\sqrt{\lambda^2-\|s\|^2})\, \eta(y_0)$$ and $$\phi_-(s_1,\, \dots,\, s_{n-1})=\sum_{k=1}^{n-1}s_k\, u_k+\beta\, \eta(y_0)$$ and then
	$$(V)=(-1)^{i-1}\, \biggl(\int_{B_\mu^{n-1}(0)} \frac{(\phi_+)_j}{\|\phi_+\|^{n}}\, \widehat{d(\phi_+)_i}
	-\int_{B_\mu^{n-1}(0)}\,  \frac{(\phi_-)_j}{\|\phi_-\|^{n}}\, \widehat{d(\phi_-)_i}\biggr).$$
	
	Now some computations:
	$$(\phi_-)_l=(\phi_-,\, e_l)=\sum_{k=1}^{n-1}s_k\, (u_k,\, e_l)+\beta\, (\eta(y_0),\, e_l)$$ 
	$$=A_l(s)+\beta\, (\eta(y_0),\, e_l)$$so
	$$\widehat{d(\phi_-)_i}=\bigwedge_{l\neq i}\biggl(\sum_{k=1}^{n-1} (u_k,\, e_l)\, ds_k\biggr)=\Gamma_i\, V_{n-1}(s)$$
	and similarly 
	$$(\phi_+)_l=(\phi_+,\, e_l)=\sum_{k=1}^{n-1}s_k\, (u_k,\, e_l)+(\sqrt{\lambda^2-\|s\|^2})\, (\eta(y_0),\, e_l)$$
	$$=A_l(s)+(\sqrt{\lambda^2-\|s\|^2})\, (\eta(y_0),\, e_l)$$	
	so
	$$\widehat{d(\phi_+)_i}=\bigwedge_{l\neq i}\biggl(\sum_{k=1}^{n-1} (u_k,\, e_l)\, ds_k+(\eta,\, e_l)\, d(\sqrt{\lambda^2-\|s\|^2})\biggr)$$ 
	$$=\Gamma_i\, V_{n-1}(s)+d(\sqrt{\lambda^2-\|s\|^2}\wedge\biggl(\sum_{p\neq i} \epsilon(i,\, p)\, \bigwedge_{l\neq i,\, p}((\eta,\, e_l)\,  A_l)\biggr)$$
	$$=\Gamma_i\, V_{n-1}(s) +d(\sqrt{\lambda^2-\|s\|^2}\wedge\Upsilon_i.$$ Finally
	$\|\phi_-\|=\sqrt{\|s\|^2+\beta^2}$ and $\|\phi_+\|=\lambda$ and we have $$\int_{B_\mu^{n-1}(0)} \frac{(\phi_+)_j}{\|\phi_+\|^{n}}\, \widehat{d(\phi_+)_i}$$
	$$=\int_{B_\mu^{n-1}(0)} \frac{A_j(s)+(\sqrt{\lambda^2-\|s\|^2})\, (\eta(y_0),\, e_j)}{\lambda^{n}}\, (\Gamma_i\, V_{n-1}(s) +d(\sqrt{\lambda^2-\|s\|^2}\wedge\Upsilon_i)$$
	$$=\frac{1}{\lambda^n}\, \{\Gamma_i\, \int_{B_\mu^{n-1}(0)} (A_j(s)\, +(\sqrt{\lambda^2-\|s\|^2})\, (\eta(y_0),\, e_j)\, )\, V_{n-1}(s)$$
	$$+\int_{B_\mu^{n-1}(0)} A_j(s)\, \frac{-\sum_{q=1}^{n-1}s_q\, ds_q\wedge\Upsilon_i}{\sqrt{\lambda^2-\|s\|^2}}$$
	$$+(\eta(y_0),\, e_j)\, \int_{B_\mu^{n-1}(0)}\sqrt{\lambda^2-\|s\|^2}\, \frac{-\sum_{q=1}^{n-1}s_q\, ds_q\wedge\Upsilon_i}{\sqrt{\lambda^2-\|s\|^2}}\}.$$ The first and the third terms are clearly bounded in terms of $R_0$.
	
	The second term is a linear combination of integrals of type $$\int_{B_\mu^{n-1}(0)} \frac{s_p\, s_q\, }{\sqrt{\lambda^2-\|s\|^2}}\, V_{n-1}(s)$$ and using spherical coordinates and integration by parts we see that it is bounded in terms of $R_0$ too.
	
	\medskip
	
	For the integral bearing $\phi_-$
	$$\int_{B_\mu^{n-1}(0)}\,  \frac{(\phi_-)_j}{\|\phi_-\|^{n}}\, \widehat{d(\phi_-)_i}=\int_{B_\mu^{n-1}(0)}\,  \frac{(A_j(s)+\beta\, (\eta(y_0),\, e_j)}{
		(\|s\|^2+\beta^2)^{\frac{n}{2}}}\, \Gamma_i\, V_{n-1}(s)$$
	consisting in a linear combination with constant coefficients of integrals of the type $$\int_{B_\mu^{n-1}(0)}\,  \frac{s_l}{
		(\|s\|^2+\beta^2)^{\frac{n}{2}}}\, V_{n-1}(s)=(1)$$
	and $$\int_{B_\mu^{n-1}(0)}\,  \frac{\beta}{
		(\|s\|^2+\beta^2)^{\frac{n}{2}}}\, V_{n-1}(s)=(2).$$
	
	Using spherical coordinates $$(1)=\int_0^\mu\frac{r^{n-1}\, dr}{(\beta^2+r^2)^{\frac{n}{2}}}\, \int_{S^{n-2}_1(0)} \theta_l\, \sigma_{n-2}(\theta)$$ and $$\int_{S^{n-2}_1(0)} \theta_l\, \sigma_{n-2}(\theta)=0.$$
	
An easy argument shows the boundedness of term $(2)$ in terms of $R_0$. 
\end{proof}

Finally $$|(352)|\le\|f\|_\infty\, |\int_{B_{\frac{R_0}{2}}(y)\cap B_{\frac{R_0}{2}}(x)\cap W^c}[(R)_y-(R)_x](\zeta)\, dm(\zeta)|$$ and 

\begin{lem} If $x,\, y\in U_{\frac{R_0}{2}}\cap W$ there exists $C(n,\, R_0)>0$ such that $$\int_{B_{\frac{R_0}{2}}(y)\cap B_{\frac{R_0}{2}}(x)\cap W^c}[(R)_y-(R)_x](\zeta)\, dm(\zeta)\le C(n,\, R_0)\, \|x-y\|^\gamma$$
	
\end{lem}
\begin{proof} Assume $d=d(x,\, \partial W)\le d(y,\, \partial\Omega)$ and consider $\tau=\frac{x+y}{2}$. In such case $$B_{\frac{R_0-\|x-y\|}{2}}(\tau)\subset B_{\frac{R_0}{2}}(y)\cap B_{\frac{R_0}{2}}(x).$$ Therefore $$\int_{B_{\frac{R_0}{2}}(y)\cap B_{\frac{R_0}{2}}(x)\cap W^c}[(R)_y-(R)_x](\zeta)\, dm(\zeta)$$
	$$=\int_{B_{\frac{R_0-\|x-y\|}{2}}(\tau)\cap W^c}[(R)_y-(R)_x](\zeta)\, dm(\zeta)$$
	$$+\int_{[(B_{\frac{R_0}{2}}(y)\cap B_{\frac{R_0}{2}}(x))\setminus B_{\frac{R_0-\|x-y\|}{2}}(\tau)]\cap W^c}[(R)_y-(R)_x](\zeta)\, dm(\zeta)=(I)+(II).$$
	
	If $\zeta\in B_{\frac{R_0-\|x-y\|}{2}}(\tau)$ then $\frac{R_0}{4}<\frac{R_0}{2}-\|x-y\|\le\|\zeta-x\|$ and the same is true for $\|\zeta-y\|$ leading to the estimate
	$$|(II)|\le C\, \|x-y\|\, m\biggl([(B_{\frac{R_0}{2}}(y)\cap B_{\frac{R_0}{2}}(x))\setminus B_{\frac{R_0-\|x-y\|}{2}}(\tau)]\cap W^c\biggr).$$
	
	Now for the term $(I)$. From the geometric lemma \ref{geom} we know the existence of $R_1>0$ independent of $x,\, y$ such that for every $u\in\partial\Omega$ there exists a function $\varphi\in\ka{1,\, \gamma}$ such that $$B_{R_1}(u)\cap W^c=\{\zeta:\ (\zeta-u,\, \eta(u))>\varphi(\Pi(\zeta-u))\}$$ where $\eta(u)$ is the normal vector to $\partial\Omega$ and $\Pi$ is the projection on the tangent hyperplane to $\partial\Omega$ at $u$.
	
	In the case of $\|x-y\|\geq\frac{R_1}{2}$ $$|(I)|\le\frac{2\, \|x-y\|}{R_1}\, \int_{B_{\frac{R_0-\|x-y\|}{2}}(\tau)\cap W^c}[|(R)_y|+|(R)_x|](\zeta)\, dm(\zeta)$$ that can be estimated by the lemma \ref{tralari}.
	
	If $\|x-y\|<\frac{R_1}{2}$ choose $\tau_1\in\partial\Omega$ such that $\|x-\tau_1\|=d$. Then $$B_{R_1}(\tau)\cap W^c=[B_{R_1}(\tau_1)\cap W^c\cap B_{\frac{R_0-\|x-y\|}{2}}(\tau)]$$
	$$\cup[(B_{R_1}(\tau_1)\cap W^c)\setminus B_{\frac{R_0-\|x-y\|}{2}}(\tau_1)]$$ and the integral $$(I)=\biggl\{\int_{B_{R_1}(\tau_1)\cap W^c\cap B_{\frac{R_0-\|x-y\|}{2}}(\tau)}$$
	$$+\int_{(B_{R_1}(\tau_1)\cap W^c)\setminus B_{\frac{R_0-\|x-y\|}{2}}(\tau)}\biggr\}[(R)_y-(R)_x](\zeta)\, dm(\zeta)=(III)+(IV).$$
	
	For the term $(IV)$:
	
	In the case of $d\geq\frac{R_1}{4}$ both $\|\zeta-x\|$ and $\|\zeta-y\|$ are bounded from below by $\frac{R_1}{4}$. If $d<\frac{R_1}{4}$ then for $\zeta\in B_{R_1}(u)\cap W^c$ we have $\|\zeta-x\|\geq\|\zeta-\tau_1\|-\|\tau_1-x\|>\frac{3\, R_1}{4}$. In both cases the integral $$|(IV)|\le\|x-y\|\, \int_{(B_{R_1}(\tau)\cap W^c)\setminus B_{\frac{R_0-\|x-y\|}{2}}(\tau_1)}\biggr\}[\frac{|(R)|_y}{\|\zeta-x\|}+\frac{|(R)_x|}{\|\zeta-y\|}]\, dm(\zeta)$$ satisfies the desired estimate.
	
	\medskip
	
	For the term $(III)$:
	
	Since $$B_{R_1}(\tau_1)\cap W^c=[B_{R_1}(\tau_1)\cap W^c\cap B_{\frac{R_0-\|x-y\|}{2}}(\tau)]\cup [(B_{R_1}(\tau_1)\cap W^c)\setminus B_{\frac{R_0-\|x-y\|}{2}}(\tau)]$$ we have 
	$$(III)=\biggl\{\int_{B_{R_1}(\tau_1)\cap W^c}-\int_{(B_{R_1}(\tau_1)\cap W^c)\setminus B_{\frac{R_0-\|x-y\|}{2}}(\tau)}\biggr\}[(R)_y-(R)_x]=(V)+(VI)$$
	
	If $\zeta\notin B_{\frac{R_0-\|x-y\|}{2}}(\tau)$ we have $\|\zeta-x\|\geq\|\zeta-\tau\|-\|\tau-x\|\geq\frac{R_0}{2}-\|x-y\|>\frac{R_0}{4}$. For $\|\zeta-y\|$ we have an identical estimate. Therefore the term $(VI)$  can be estimated in the same way as the term $(IV)$.

	For estimating the term $(V)$ we may assume that, after a rotation  and translation we have $\tau_1=0$. Then the normal vector to $\partial\Omega$ at this point is $\eta(\tau_1)=e_n$ and $\zeta=(s_1,\, \dots,\, s_{n-1},\, t)=(s,\, t)$. We also use the notation $C_0=\|\varphi\|_{\ka{1,\, \gamma}}\simeq\|\rho\|_{\ka{1,\, \gamma}}$. where $\varphi$ and $\rho$ are defining $\partial\Omega$, as in Lemma \ref{geom}.
	
	And we have:
	
	\underline{Claim}: If $$U=\{(s,\, t):\ \|s\|\le\frac{R_1}{1+C_0^2\, R_1^{2\, \gamma}},\ \varphi(s)<t<\sqrt{R_1^2-\|s\|^2}\}$$ then  $$B_{R_1}(\tau_1)\cap W^c=U\cup(\{(s,\, t):\ \frac{R_1}{1+C_0^2\, R_1^{2\, \gamma}}<\|s\|<R_1\}\cap B_{R_1}(\tau_1)\cap W^c),$$ a disjoint union.

	\underline{Proof of the claim} In fact, there exists $\mu\in(0,\, 1]$ such that if $\|s\|<\mu$ then $|\varphi(s)|<\sqrt{R_1^2-\|s\|^2}$. This is so because $|\varphi(s)|\le C_0\, \|s\|^{1+\gamma}$ and the possibility of choosing $\mu_0$ in such a way that $C_0\, \|s\|^{1+\gamma}\le\sqrt{R_1^2-\|s\|^2}$ will proves the claim.
	
	In order to fall in this case we need that $$C_0\, (\mu\, R_1)^{1+\gamma}\le R_1\, \sqrt{1-\mu^2}$$ so is $$(C_0\,  R_1^{\gamma})^2\, \mu^{2\, (1+\gamma)}\le 1-\mu^2.$$ Since $\mu<1$ then $\mu^{2\, (1+\gamma)}\le\mu^2$ and a choice of $\mu$ satisfying 
	$$(C_0\,  R_1^{\gamma})^2\, \mu^2\le 1-\mu^2$$ will be enough. But this is equivalent to $\mu\le\frac{1}{1+(C_0\,  R_1^{\gamma})^2}$. Therefore $t\in[\varphi(s),\, \sqrt{R_1^2-\|s\|^2}]$ if $\|s\|\le\frac{R_1}{1+(C_0\,  R_1^{\gamma})^2}$. Then the claim is  proved.  
	
	\medskip

	According to the claim, we have $$(V)=\biggl\{\int_{U}+\int_{B_{R_1}(\tau_1)\cap W^c)\cap\{\frac{R_1}{1+(C_0\,  R_1^{\gamma})^2}<\|s\|<R_1\}}\biggr\}[(R)_y-(R)_x]$$
	$$=(VII)+(VIII).$$
	
	We assume that $d<\frac{R_1}{4\, (1+(C_0\,  R_1^{\gamma})^2)}$. Otherwise we can proceed directly as in the case of the term $(IV)$. In this case, for $\zeta\in B_{R_1}(\tau_1)\cap W^c\cap\{\frac{R_1}{1+(C_0\,  R_1^{\gamma})^2}<\|s\|<R_1\}$ we have that $\|\zeta-x\|\geq\|\zeta-\tau_1\|-\|\tau_1-x\|>\frac{3\, R_1}{4\, (1+(C_0\,  R_1^{\gamma})^2)}$ and also that $\|\zeta-y\|\geq\|\zeta-x\|-\|y-x\|>\frac{3\, R_1}{4\, (1+(C_0\,  R_1^{\gamma})^2)}$ then proceeding as for the term $(IV)$ gives the control of $(8)$.
	
	\medskip
	
	Finally for the term $(VII)$ it is possible to modify the domain of integration in a suitable way in order to apply  Main Lemma in \cite{MOV}. For doing so we first consider $\mu_0=\frac{1}{1+(C_0\,  R_1^{\gamma})^2}$ and then the partition of the segment $[\mu_0\, R_1,\, R_1]$ in segments determined by the points $\alpha_k=\mu_0\, R_1+k\, A$ for $k=1,\, \dots,\, 4$ where $A=\frac{R_1\, (1-\mu_0)}{4}$.
	
	Next we consider a decreasing function $E\in\ka{\infty}(\R)$ such that $E(\varsigma)=1$ if $\varsigma\le\mu_0\, R_1$ and $E(\varsigma)=0$ if $\varsigma\geq\alpha_2$.
	
	The function defined in $\R^{n-1}$ by the formula $$\chi(s)=\varphi(s)\, E(\|s\|^2)$$ is in $\ka{1,\, \gamma}(\R^{n-1})$ and coincides with $\varphi$ in the ball $B_{\mu_0\, R_1}^{n-1}(0)$.
	
	Next we modify this function defining $$\psi(s)=\begin{cases}\chi(s)\ \ \ &\text{if}\ s\in B_{\alpha_3}(0)\\
		
		A-\sqrt{A^2-(\|s\|-\alpha_3)^2}\ \ \ &\text{if}\ \alpha_3\le\|s\|\le R_1.
	\end{cases}$$ We have that $\psi\in\ka{1,\, 1}$ at the points of $\partial B_{\alpha_3}(0)$ and this implies that it is in $\ka{1,\, \gamma}(B_{R_1}(0))$. The graph of this function defines an embedded open $\ka{1,\, \gamma}$ hypersurface of $\R^n$ that we will call $\Sigma_1$. 
	
	We also have the cylindric $\ka{\infty}$ closed hypersurface $$\Sigma_2=\{(s,\, t)\in\R^n:\ \|s\|=R_1,\ A\le t\le\sqrt{R_1^2-\alpha_1^2}-A\}.$$
	
	Finally we define a new function $$\varrho(s)=\sqrt{R_1^2-\alpha_2^2}+E(s)\, (\sqrt{R_1^2-\|s\|^2}-\sqrt{R_1^2-\alpha_2^2})$$
	$$=\sqrt{R_1^2-\|s\|^2}+(1-\chi(s))\, \sqrt{R_1^2-\alpha_2^2}$$ and then the function
	$$\kappa(s)=\begin{cases}\varrho(s)\ \ \ &\text{if}\ s\in B_{\alpha_3}(0)\\
		-A+\sqrt{R_1^2-\alpha_3^2}+\sqrt{A^2-(\|s\|^2-\alpha_3)^2}\ \ \ &\text{if}\ \alpha_3\le\|s\|\le R_1.
		
	\end{cases}
	$$ 
	The graph of $\kappa$ determines a $\ka{\infty}$ embedded open hypersurface $\Sigma_3$.
	
	We claim that the  closed set $$\Sigma=\Sigma_1\cup\Sigma_2\cup\Sigma_3$$ is connected and in fact it is an hypersurface defined by a function $\phi$ as the boundary of an open bounded set $\Xi$ and whose meaningful constants depend only on the meaningful constants of the function $\rho$. The procedure is in fact a well known globalization one. A variant of this procedure for a more general situation (but that can be easily adapted here) can be found in \cite{Bur}.
	
	From this decomposition  we obtain $$(VII)=\biggl\{\int_U+\int_{\Xi\cap  U^c}-\int_{U\cap\Xi^c}\}\, [(R)_y-(R)_x].$$
	
	Since we have that $$U\cap\{\|s\|\le\alpha_1\}=\Xi\cap\{\|s\|\le\alpha_1\},$$ then if $\zeta\in(\Xi\cap U^c)\cup(\Xi^c\cap U)$ we have $$\|\zeta-x\|\|\zeta\|-\|x\|\geq\|s\|-\|x\|\geq\alpha_1-d$$
	$$\geq\mu_0\, R_1+\frac{R_1\, (1-\mu_0)}{4}-\frac{R_1}{4\, (1+C_0^2\, R_1^{2\, \gamma})}=\frac{R_1}{4}+\frac{1}{4\, (1+C_0^2\, R_1^{2\, \gamma})}$$ and $$\|\zeta-y\|\geq\|\zeta-x\|-\|x-y\|$$
	$$\geq R_1\, (\frac{1}{4}+\frac{1}{2\, (1+C_0^2\, R_1^{2\, \gamma})}-\frac{R_1}{4\, (1+C_0^2\, R_1^{2\, \gamma})})=\frac{R_1}{4\, (1+C_0^2\, R_1^{2\, \gamma})}.$$ 
	
	This means that $(10)-(11)$ can be estimate just as $(8)$. Finally the direct application of the Main Lemma in \cite{MOV} provides the necessary estimate of $(9)$.

	

\end{proof}

\noindent
{\bf 4.} Finally the term 
\begin{equation*}
	\begin{split}
		(4)&=\int_{B_{\delta(x)}(x)\setminus B_{\delta(y)}(y)}\frac{f(\zeta)-f(x)}{\|\zeta-x\|^\gamma}\, 
		\|\zeta-x\|^\gamma\, (R)_x(\zeta)\\*[5pt]
		&\quad-
		\int_{B_{\delta(y)}(y)\setminus B_{\delta(x)}(x)}\frac{f(\zeta)-f(y)}{\|\zeta-y\|^\gamma}\, 
		\|\zeta-y\|^\gamma\, (R) y(\zeta)\, dm(\zeta)=(41)+(42).
	\end{split}
\end{equation*}


The term 
$$
|(41)|\le \|f\|_\gamma\, \int_{B_{\delta(x)}(x)\setminus B_{\delta(y)}(y)}\frac{dm(\zeta)}{\|\zeta-x\|^{n-\gamma}},
$$ 
and since $\|x-y\|\le\frac{R_0}{4}<\delta(y),\, \delta(x)$, then we have that
$x\in B_{\delta(y)}(y)$. Moreover, if $\|\zeta-x\|<\delta(y)-\|x-y\|$, then $\delta(y)>\|\zeta-x\|+\|x-y\|\geq\|\zeta-y\|$. This implies that 
$$
B_{\delta(y)-\|y-x\|}(x)\subset B_{\delta(y)}(y)
$$ 
and then, since the $\delta(y)-\|x-y\|<\delta(x)$ (otherwise the domain of integration is empty) the previous integral is bounded by 
\begin{equation*}
	\begin{split}
		\int_{C^{\delta(x)}_{\delta(y)-\|x-y\|}(x)}&\frac{dm(\zeta)}{\|\zeta-x\|^{n-\gamma}}\\*[5pt]
		&=\sigma(S_1^{n-1}(0))\, \int^{\delta(x)}_{\delta(y)-\|x-y\|}\frac{dr}{r^{1-\gamma}}\\*[5pt]
		&=\frac{\sigma(S_1^{n-1}(0))}{\gamma}\, \{\delta(x)^\gamma-(\delta(y)-\|x-y\|)^\gamma\}\\*[5pt]
		&=\sigma(S_1^{n-1}(0))\, \frac{\delta(x)-(\delta(y)-\|x-y\|)}{(\delta(y)-\|x-y\|+\lambda\, (\delta(x)-(\delta(y)-\|x-y\|))^{1-\gamma}}\\*[5pt]
		&\le\sigma(S_1^{n-1}(0))\, \frac{\delta(x)-\delta(y)+\|x-y\|)}{(\delta(y)-\|x-y\|)^{1-\gamma}}
		\le\sigma(S_1^{n-1}(0))\, \frac{\delta(x)-\delta(y)+\|x-y\|)}{(\frac{R_0}{4})^{1-\gamma}},
	\end{split}
\end{equation*}
where $\lambda\in(0,1)$.

The term $(42)$ is symmetrically analogous to $(41)$.

With the previous arguments we have the Riesz transform of $\phi$ and $\psi$ is H\"older at $\Omega$ or $(\bar\Omega)^c$. The following lemma completes the case of~$\partial\Omega$.

\begin{lem}
	If $W=\Omega$ or $W=(\bar\Omega)^c$ and $f\in \Lip{(\gamma, W)}$ and $\|f\|_\gamma<+\infty$, then $f$ extends to a Lipschitz function on $\bar W$, with the same Lipschitz norm.  
\end{lem}

\begin{proof} 
	Since $\partial W$ is compact, then $f$ is uniformly continuous in $\bar U_{\frac{R_0}{4}}$.
	
	If $x,y\in\partial\Omega$ and $\|x-y\|<\frac{R_0}{4}$, for $U_{\frac{R_0}{4}}\cap W$ we have 
	$$
	f(x)-f(y)=f(x)-f(x')+f(x')-f(y')+f(y')-f(y)=(1)+(2)+(3).
	$$  
	
	If $\|x'-x\|,\, \|y'-y\|<\min\{\delta,\, \frac{\|x-y\|}{3}\}$, then 
	$$
	(1),\, (3)\le M\, \|x-y\|^\gamma.
	$$ 
	
	Also 
	$$
	\|x'-y'\|\le \frac{5}{3}\, \|x-y\|,
	$$ 
	so 
	\begin{equation*}
		(2)\le M\, \|x-y\|^\gamma.\qedhere
	\end{equation*}
\end{proof}

Finally, we can apply this result to the boundary term in Proposition \ref{desing} and we have

\begin{cor} 
	There exists a constant $C_1$, depending only on $n$, $\gamma$ and $\Omega$ such that 
	$$
	|f(x)\, \Theta_\Omega^{\frac{R_0}{2}}(x)-f(y)\, \Theta_\Omega^{\frac{R_0}{2}}(y)|\le \|x-y\|^\gamma\, \|f\|_\gamma\, C_1.
	$$
\end{cor}

\section{Appendix}
\subsubsection{The jump formula for $R_{j,\, i}$}  We begin with a very precise computation we will need in the proof of the jump formula. It reffers to the singular integral resulting from the newtonian potential (defined in 2.2) integrated on the intersection of an hyperplane and a ball.

\begin{lem}\label{half} 
	
	If $x\in\R^n$ and $\eta\in S^{n-1}_1(0)$,  $\lambda\in\R\setminus\{0\}$ and $(K_j)_x(\zeta)=\frac{(\zeta-x )_j}{\|\zeta-x)\|^n}$ we have that $$ \int_{B_{2\, \lambda}(x)\cap\{(\zeta-x,\, \eta)=0\}}(K_j)_{x+\lambda\, \eta}(\zeta)\, \widehat{d\zeta_i}$$ is a real number depending only on $\eta$. We call it $\mathcal{K}_{j,\, i}(\eta)$ .

\end{lem}
\begin{proof} Considering the mapping $\phi:\R^n\rightarrow\R^n$ given by
	$s\rightarrow x+\sum_{j=1}^{n-1}s_j\, u_j+s_n\, \eta$ where \newline $u_1,\, \dots,\, u_{n-1},\, \eta$ is an orthonormal basis for $\R^n$ we have that $$B_{2\, \lambda}(x)\cap\{(\zeta-x,\, \eta)=0\}=\phi(B_{2\, \lambda}(0)\cap\{s_n=0\})$$ and then the integral in the statement is equal to   $$\int_{B_{2\, \lambda}(0)\cap\{s_n=0\}}\frac{(\phi(s)-(x+\lambda\, \eta))_j}{\|\phi(s)-(x+\lambda\, \eta)\|^n}\, \widehat{d\phi_i(s)}=(*).$$
	
	Also $$\widehat{d\phi_i(s)}=\bigwedge_{l\neq i}d\phi_l(s)=\sum_{k=1}^n A_{i,k}(s)\, \widehat{ds_k}$$ and in our domain of integration $s_n=0$, consequently   $$(\phi(s)-(x+\lambda\, \eta))_j=\sum_{p=1}^{n-1}s_p\, (u_p,\, e_j)-\lambda\, (\eta,\, e_j)$$ and 
	$$\|\phi(s)-(x+\lambda\, \eta)\|^n=\|\sum_{l=1}^{n-1}s_l\, u_l-\lambda\, \eta\|^n=[\sum_{l=1}^{n-1}s_l^2+\lambda^2]^\frac{n}{2}$$
	
	Therefore
	$$(*)=\int_{B_{2\, \lambda}(0)\cap\{s_n=0\}}\frac{(\phi(s)-(x+\lambda\, \eta))_j}{\|\phi(s)-(x+\lambda\, \eta)\|^n}\, A_{i,\, n}(s)\, \widehat{ds_n}.$$

	If $s=2\, \lambda\, \tau$, where $\tau\in B_1^{n-1}(0)$ the integral above is equal to consider the $$\int_{B^{n-1}_{1}(0)}\frac{2\, \sum_{p=1}^{n-1} \tau_p\, (u_p,\, e_j)-(\eta,\, e_j)}{[4\, \sum_{l=1}^{n-1} \tau_l^2+1]^\frac{n}{2}}\,  A_{i,\, n}(2\, \lambda\, \tau)\,  dm(\tau),$$ but $$A_{i,\, n}(2\, \lambda\, \tau)=V^{n-1}_{\widehat{e_n}}(\widehat {u_i}),$$ the volume of the trace in the hyperplane $<e_1,\, \dots,\, e_{n-1}>$ of the parallelepiped generated by $u_1,\, \dots,\, \widehat{u_i},\, \dots, u_{n-1},\, \eta$, for $l\neq i$.
	
	Finally the integrals $$\int_{B^{n-1}_{1}(0)}\frac{\tau_p\, }{[4\, \sum_{l=1}^{n-1} \tau_l^2+1]^\frac{n}{2}}\,  dm(\tau)=0.$$ In all cases and the constants do not depend on the chosen basis, $u$, for $<\eta>^\perp$  
	and the result is a universal constant.

\end{proof}

Now we prove the jump formula.

\begin{thm}\label{jump} Let $\Omega\subset\R^n$ a region with compact boundary $\partial\Omega\in\ka{1,\, \gamma}$, where $\gamma\in(0,1)$. Let $ g\in\ka{\infty}(\R^n\setminus\partial\Omega)$ such that $\chi_\Omega\, g$ extends to a function in $Lip(\gamma,\, \bar\Omega)$ and  $\chi_{\bar\Omega^c}\, g$ extends to a function in $Lip(\gamma,\, \R^n\setminus\Omega)$. Consider $g_-$ the extension of $\chi_{\bar\Omega^c}\, g$ continued by $0$ to $\bar\Omega^c$ to $\bar\Omega^c$ and $g_+$ the extension of $\chi_{\bar\Omega^c}\, g$ continued by $0$ to $\Omega$ and define for $x\in\R^n$ the function $g(x)=\frac{1}{2}\{g_+(x)+g_-(x)\}$. Then for any $x\in\partial\Omega$ we have $$R_{j,\, i}[g](x)=\frac{1}{2}\biggl\{\lim_{y\to x;\, y\in\Omega} (\chi_\Omega\,  R_{j,\, i}[g])(y)+\lim_{y\to x;\, y\in\R^n\setminus\Omega} (\chi_{\R^n\setminus \bar\Omega}\, R_{j,\, i}[g])(y)\biggr\}
	$$ 
\end{thm}

\begin{rmk} Other kinds of jump formulas for Calder\'on--Zygmund operators in potential theory appear in \cite{HoMiTa} and \cite{Tol}. 
\end{rmk}

\begin{proof}

	As seen in section 4 we have that $\chi_\Omega\, R_{j,\ i}[g]$ extends to a function in $\operatorname{Lip}_{\gamma}(\bar\Omega)$ and  $\chi_{\R^n\setminus \bar\Omega}\, R_{j,\ i}[g]$ extends to a function in $\operatorname{Lip}_{\gamma}(\R^n\setminus\Omega)$ then both the limits $\lim_{y\to x;\, y\in\Omega} \chi_\Omega\, R_{j,\ i}[g](y)$ and \newline $\lim_{y\to x;\, y\in\R^n\setminus\Omega} \chi_{\R^n\setminus \bar\Omega}\, R_{j,\ i}[g](y)$  exist and we can choose $y=x\pm\lambda\eta(x)$, where $\eta(x)$ is the unit vector, normal exterior to $\partial\Omega$ at $x$.
	
	
	Also, if $g_\pm$ are the Lipschitz extensions of $g$ to $\Omega^c$ and $\bar\Omega$, respectively, we have, for $y\in B_{\frac{R_0}{8}}(x)$,and fixed $x\in\partial\Omega$ the following facts	
	\begin{itemize}
		\item If $y\in\Omega$, then 
		\begin{equation*}
			\begin{split}
				R_{j,\ i}[g_-](y)&=\int_\Omega (g_--g_-(y))\, (R_{j,\ i})_y+g_-(y)\, \int_\Omega (R_{j,\ i})_y\\*[5pt]
				&=\int_{\Omega\setminus B_{2\, \|x-y\|}(x)} (g_--g_-(y))\, (R_{j,\ i})_y\\*[5pt]
				&\quad+\int_{\Omega\cap B_{2\, \|x-y\|}(x)} (g_--g_-(y))\, (R_{j,\ i})_y+g_-(y)\, \int_\Omega (R_{j,\ i})_y\\*[5pt]
				&=(I)(y)+(II)(y)+g_-(y)\, (III)(y).
			\end{split}
		\end{equation*}

		For the integral $(I)(y)$, we have immediately that, for any fixed $\zeta$, 
		$$
		\chi_{\Omega\setminus B_{2\, \|x-y\|}(x)}(\zeta)\,  (g_-(\zeta)-g_-(y))\, (R_{j,\ i})_y(\zeta)\to_{y\to x}\chi_\Omega(\zeta)\,  (g_-(\zeta)-g_-(x))\, (R_{j,\ i})_x(\zeta),
		$$ 
		and also 
		\begin{equation*}
			|\chi_{\Omega\setminus B_{2\, \|x-y\|}(x)}(\zeta)\,  (g_-(\zeta)-g_-(y))\, (R_{j,\ i})_y(\zeta)|
			\le \|g_-\|_{\Lip{(\gamma,\, \bar\Omega)}}\frac{2}{\|\zeta-x\|^{n-\gamma}},
		\end{equation*} 
		and by the dominated convergence theorem, 
		$$
		(I)(y)\to_{y\to x}\int_\Omega (g_--g_-(x))\, (R_{j,\ i})_x.
		$$ 
		
		
		The term $(II)(y)$ is controlled by $$\|g_-\|_{\Lip{(\gamma,\, \bar\Omega)}}\, \int_{\Omega\cap B_{2\, \|x-y\|}(x)} \frac{1}{\|\zeta-y\|^{n-\gamma}}\, dm(\zeta)\lesssim\|g_-\|_{\Lip{(\gamma,\, \bar\Omega)}}\, \, \int_0^{3\, \|x-y\|}\frac{1}{r^{1-\gamma}}\, dr$$
		$$\lesssim\|g_-\|_{\Lip{(\gamma,\, \bar\Omega)}}\, \|x-y\|^\gamma,
		\to_{y\to x}0.
		$$

		Also, if $\varrho\in\mathcal{D}(\R^n)$ such that $\varrho\equiv1$ in a ball containing $\bar\Omega$, then
		\begin{equation*}
			\begin{split}
				R_{j,\ i}[g_+](y)&=R_{j,\ i}[(1-\varrho)\, g_+](y)+R_{j,\ i}[\varrho\, g_+](y)\\*[5pt]
				&=(IV)(y)+\int_{\R^n\setminus(\Omega\cup B_{2\, \|x-y\|}(x))} \varrho\, ( g_+-g_+(x))\, (R_{j,\ i})_y\\*[5pt]
				&\quad+\int_{\Omega^c\cap B_{2\, \|x-y\|}(x)} \varrho\, (g_+- g_+(x))\, (R_{j,\ i})_y+g_+(x)\, \int_{\Omega^c} \varrho\, (R_{j,\ i})_y\\*[5pt]
				&=(IV)(y)+(V)(y)+(VI)(y)+g_+(x)\, (VII)(y).
			\end{split}
		\end{equation*}

		It is immediate that 
		$$
		\lim_{y\to x}(IV)(y)=R_{j,\ i}[(1-\varrho)\, g_+](x).
		$$
		
		
		By arguments similar to those used for $(I)(y)$, we have that 
		$$
		\lim_{y\to x}(V)(y)=\int_{\bar\Omega^c} \varrho\, (g_+-g_+(x))\, (R_{j,\ i})_x,
		$$ 
		and, analogously to $(II)(y)$, we have that 
		$$
		\lim_{y\to x}(VI)(y)=0.
		$$

		So, as all limits exist, we have 
		\begin{equation*}
			\begin{split}
				2\, \lim_{y\to x;\, y\in\Omega}& \chi_\Omega\, R_{j,\ i}[ g](y)\\*[5pt]
				&=\int_\Omega (g_--g_-(x))\, (R_{j,\ i})_x+R_{j,\ i}[(1-\varrho)\, g_+](x)\\*[5pt]
				&\quad+\int_{\bar\Omega^c} \varrho\, (g_+-g_+(x))\, (R_{j,\ i})_x+g_-(x) \lim_{y\to x}\!\int_\Omega (R_{j,\ i})_y\!+\!g_+(x) \lim_{y\to x}\!\int_{\bar\Omega^c} \varrho\, (R_{j,\ i})_y\!\\*[5pt]
				&=R_{j,\ i}[g_-](x)+R_{j,\ i}[g_+](x)+g_-(x)\biggl\{\lim_{y\to x}\int_\Omega (R_{j,\ i})_y-\int_\Omega (R_{j,\ i})_x\biggr\}\\*[5pt]
				&\quad+g_+(x)\biggl \{\lim_{y\to x}\int_{\bar\Omega^c} \varrho\, (R_{j,\ i})_y-\int_{\bar\Omega^c} \varrho\, (R_{j,\ i})_x\biggr\}\\*[5pt]
				&=2\, R_{j,\ i}[g](x)+g_-(x)\biggl \{\lim_{y\to x}\int_\Omega (R_{j,\ i})_y-\int_\Omega (R_{j,\ i})_x\biggr\}\\*[5pt]
				&\quad+g_+(x)\biggl\{\lim_{y\to x}\int_{\bar\Omega^c} \varrho\, (R_{j,\ i})_y-\int_{\bar\Omega^c} \varrho\, (R_{j,\ i})_x\biggr\}.
			\end{split}
		\end{equation*}

		\item If $y\notin\bar\Omega$, then, in a similar way, we have
		\begin{equation*}
			\begin{split}
				2\, \lim_{y\to x;\, y\in\bar\Omega^c} \chi_{\bar\Omega^c}\,  R_{j,\ i}[g](y)&=2\, R_{j,\ i}[g](z)+g_-(z)\biggl \{\lim_{y\to x}\int_{\Omega} (R_{j,\ i})_y-\int_\Omega (R_{j,\ i})_x\biggr\}\\*[5pt]
				&\quad+g_+(z)\biggl\{\lim_{y\to x}\int_{\bar\Omega^c} \varrho\, (R_{j,\ i})_y-\int_{\bar\Omega^c} \varrho\, (R_{j,\ i})_x\biggr\}.
			\end{split}
		\end{equation*}

		\item Then, for any $x\in\partial\Omega$ we have 
		\begin{equation*}
			\begin{split}
				&\frac{1}{2}\biggl \{\lim_{y\to x;\, y\in\Omega} \chi_\Omega\,  R_{j,\ i}(y)+\lim_{y\to x;\, y\in\R^n\setminus\Omega} \chi_{\R^n\setminus \bar\Omega}\, R_{j,\ i}[g](y)\biggr\}=R_{j,\ i}[g](x)\\*[5pt]
				&\quad+\frac{1}{2}\, g_-(x)\biggl \{\lim_{y\to x;\, y\in\Omega}\int_\Omega (R_{j,\ i})_y-\int_\Omega\, (R_{j,\ i})_x\biggr\}\\*[5pt]
				&\quad+\frac{1}{2}\, g_+(x)\biggl\{\lim_{y\to x;\, y\in\Omega}\int_{\bar\Omega^c}\varrho\,  (R_{j,\ i})_y-\int_{\bar\Omega^c}\varrho\,  (R_{j,\ i})_x\biggr\}\\*[5pt]
				&\quad+\frac{1}{2}\, g_-(x)\biggl \{\lim_{y\to x;\, x\notin\bar\Omega}\int_{\Omega} (R_{j,\ i})_y-\int_\Omega (R_{j,\ i})_x\biggr\}\\*[5pt]
				&\quad+
				\frac{1}{2}\, g_+(x)\biggl\{\lim_{y\to x;\, x\notin\bar\Omega}\int_{\bar\Omega^c}\varrho\,  (R_{j,\ i})_y-\int_{\bar\Omega^c}\varrho\,  (R_{j,\ i})_x\biggr\}\\*[5pt]
				&=R_{j,\ i}[g](x)+g_-(x)\biggl(\frac{1}{2}\biggl \{\lim_{y\to x;\, y\in\Omega}\int_\Omega (R_{j,\ i})_y+\lim_{y\to x;\, x\notin\bar\Omega}\int_{\Omega} (R_{j,\ i})_y\biggr\}-\int_\Omega (R_{j,\ i})_x\biggr)\\*[5pt]
				&\quad+g_+(x)\biggl (\frac{1}{2}\biggl \{\lim_{y\to x;\, y\in\Omega}\int_{\bar\Omega^c}\varrho\,  (R_{j,\ i})_y+\lim_{y\to x;\, x\notin\bar\Omega}\int_{\bar\Omega^c}\varrho\,  (R_{j,\ i})_y\biggr\}-\int_{\bar\Omega^c}\varrho\,  (R_{j,\ i})_x\biggr).
			\end{split}
		\end{equation*} 
		
		And the lemma below finishes the proof of Theorem \ref{jump}.
	\end{itemize}
	
\end{proof}

\begin{lem} Let $W\subset\R^n$ a domain with compact $\ka{1,\gamma}$ boundary.Let $x\in\partial W$ and $\eta=\eta(x)$ be the normal exterior vector at $x$.  Consider the points $y=x\pm\lambda\, \eta$. 
	
	For $h\in\mathcal{D}(\R^n)$ we have 
	$$
	\frac{1}{2}\biggl \{\lim_{\lambda\to 0;\, y=x-\lambda\, \eta}\int_{W}h\,  (R_{j,\ i})_y+\lim_{\lambda\to 0;\, y=x+\lambda\, \eta}\int_{W}h\,  (R_{j,\ i})_y\biggr\}=\int_{W}h\,  (R_{j,\ i})_x.
	$$
\end{lem}

\begin{proof} 
	First of all, for $y\notin\partial W$, using Stokes theorem we have  
	\begin{equation*}
		\begin{split}
			&\int_{W}h\,  (R_{j,\ i})_y =\int_{W}h(\zeta)\, d\biggl((-1)^{i-1}\, (K_j)_y\, \widehat{d\zeta_i}\biggr)\\*[5pt]
			&=-\int_{W}\frac{\partial h}{\partial\zeta_i}(\zeta)\, (K_j)_y \, V(\zeta)+\int_{\partial W}h(\zeta)\, (-1)^{i-1}\, (K_j)_y\, \widehat{d\zeta_i} =(1)+(-1)^{i-1}\, c_n\  (2).
		\end{split}
	\end{equation*} 
	
	The integral 
	$$(2)=
	\int_{\partial W}h(\zeta)\, \frac{(\zeta-y)_j}{\|\zeta-y\|^n}\, \widehat{d\zeta_i}=\int_{\partial W\setminus B_{2\, \|x-y\|}(x)}h(\zeta)\, \frac{(\zeta-y)_j}{\|\zeta-y\|^n}\, \widehat{d\zeta_i}$$
	$$+\int_{\partial W\cap B_{2\, \|x-y\|}(x)}h(\zeta)\, \frac{(\zeta-y)_j}{\|\zeta-y\|^n}\, \widehat{d\zeta_i}=(21)+(22).
	$$ 
	
	As consequence of Lemma 11 we have that the integral of $K_j$ on an hypersurface near a given point can be approximated locally by the integral over the tangent hyperplane across this point.
		
	\begin{equation*}
		\begin{split}
			&\int_{\partial W\cap B_{2\, \|x-y\|}(x)} \frac{(\zeta-y)_j}{\|\zeta-y\|^n}\, \widehat{d\zeta_i}=\int_{\{\kappa_x=0\}\cap B_{2\, \|x-y\|}(x)}\frac{(\zeta-y)_j}{\|\zeta-y\|^n}\, \widehat{d\zeta_i}\\*[5pt]
			&\quad+\biggl\{\int_{\partial W\cap B_{2\, \|x-y\|}(x)} -\int_{\{\kappa_x=0\}\cap B_{2\, \|x-y\|}(x)}\biggr\}\, \frac{(\zeta-y)_j}{\|\zeta-y\|^n}\, \widehat{d\zeta_i}\\*[5pt]
			&=(3)+(4),
		\end{split}
	\end{equation*} 
	and by the Stokes formula the second term is 
	\begin{equation*}
		\begin{split}
			(4)&=-\biggl\{\int_{\{\kappa_x<0\}\cap W^c\cap B_{2\, \|x-y\|}(x)}+\int_{\{\kappa_x>0\}\cap W\cap B_{2\, \|x-y\|}(x)}\biggr\}\,  \frac{(\zeta-y)_j}{\|\zeta-y\|^n}\, \widehat{d\zeta_i}\\*[5pt]
			&\quad-\biggl\{\int_{\{\kappa_x<0\}\cap W^c\cap \partial B_{2\, \|x-y\|}(x)}+\int_{\{\kappa_x>0\}\cap W\cap \partial B_{2\, \|x-y\|}(x)}\biggr\} \frac{(\zeta-y)_j}{\|\zeta-y\|^n}\, \widehat{d\zeta_i}.
		\end{split}
	\end{equation*}

	All the integrals are well defined for  $y=x\pm\lambda\, \eta$, and the  Riesz-geometric lemma and the facts that both 
	$\{\kappa_x>0\}\cap W\cap \partial B_{2\, \|x-y\|}(x)$ and $\{\kappa_x<0\}\cap W\cap \partial B_{2\, \|x-y\|}(x)$ tend to $0$ as $r$ tends to $0$, imply that 
	$$
	\lim_{y\to x}(4)=0.
	$$ 
	
	After a translation and a rotation and use of lemma \ref{half}, 
	$$
	\lim_{y\to x}(3)=\mathcal{K}_{j,\, i}(\frac{\nabla\rho}{\|\nabla\rho\|})\stackrel{\text{def}}{=}\mathcal{K}_{j,\, i}(\rho).
	$$

	Then, 
	\begin{multline*}
		\frac{1}{2}\, \biggl \{\lim_{\lambda\to 0;\, y=x-\lambda\, \eta}\int_{W}h\,  (R_{j,\ i})_y+\lim_{\lambda\to 0;\, y=x+\lambda\, \eta}\int_{W}h\,  (R_{j,\ i})_y\biggr\}\\*[5pt]
		=-\frac{1}{2}\biggl \{\lim_{\lambda\to 0}\int_{W}\frac{\partial h}{\partial\zeta_i}(\zeta)\, (K_j)_{x-\lambda\, \eta} \, V(\zeta)
		+\lim_{\lambda\to 0}\int_{W}\frac{\partial h}{\partial\zeta_i}(\zeta)\, (K_j)_{x+\lambda\, \eta} \, V(\zeta)\biggr\}\\*[5pt]
		+\frac{(-1)^{i-1}}{2}\, \biggl\{\lim_{\lambda\to 0}\int_{\partial W}h(\zeta)\, (K_j)_{x-\lambda\, \eta}\, \widehat{d\zeta_i}+\lim_{\lambda\to 0}\int_{\partial W}h(\zeta)\, (K_j)_{x+\lambda\, \eta}\, \widehat{d\zeta_i}\biggr\}\\*[5pt]
		=-\int_{W}\frac{\partial h}{\partial\zeta_i}(\zeta)\, (K_j)_x \, V(\zeta)
		+\frac{(-1)^{i-1}}{2}\, \lim_{\lambda\to 0}\int_{\partial W}h(\zeta)\, [(K_j)_{x-\lambda\, \eta}+(K_j)_{x+\lambda\, \eta}]\, \widehat{d\zeta_i}.
	\end{multline*} because $$(K_j)_{x-\lambda\, \eta}+(K_j)_{x+\lambda\, \eta}=c_n\, [\frac{(\zeta-(x-\lambda\, \eta))_j}{\|\zeta-(x-\lambda\, \eta)\|^n}+\frac{(\zeta-(x+\lambda\, \eta))_j}{\|\zeta-(x+\lambda\, \eta)\|^n}]$$

	On the other hand $$\int_{W}h\,  (R_{j,\ i})_x=\lim_{\lambda\to 0}\int_{W\setminus B_{2\, \lambda}(x)}h\,  (R_{j,\ i})_x=\lim_{\lambda\to 0} I_\lambda$$ and using the Stokes theorem we have $$I_\lambda=(-1)^{i-1}\, \int_{\partial(W\setminus B_{2\, \lambda}(x)\, )}h\,  (K_j)_y\, \widehat{d\zeta_j}-\int_{W\setminus B_{2\, \lambda}(x)}\frac{\partial h}{\partial\zeta_i}(\zeta)\, \, (K_j)_y\, V(\zeta)$$
	$$=(-1)^{i-1}\, \int_{(\partial W)\cap B_{2\, \lambda}(x)^c}h\,  (K_j)_y\, \widehat{d\zeta_j}+(-1)^{i-1}\, \int_{(\partial B_{2\, \lambda}(x))\cap W}h\,  (K_j)_y\, \widehat{d\zeta_j}$$
	$$-\int_{W\setminus B_{2\, \lambda}(x)}\frac{\partial h}{\partial\zeta_i}(\zeta)\, \, (K_j)_y\, V(\zeta)$$ and $$\lim_{\lambda\to 0} I_\lambda=-\int_{W}\frac{\partial h}{\partial\zeta_i}(\zeta)\, \, (K_j)_y\, V(\zeta)+(-1)^{i-1}\, \text{v. p.} \int_{(\partial W)\cap B_{2\, \lambda}(x)^c}h\,  (K_j)_y\, \widehat{d\zeta_j},$$ for 
	$$\int_{(\partial B_{2\, \lambda}(x))\cap W}h\,  (K_j)_{y}\, \widehat{d\zeta_j}\simeq\, \int_{\Sigma_\lambda}h(x+2\, \lambda\, \theta)\, \frac{(x+2\, \lambda\, \theta-y)_j}{\|x+2\, \lambda\, \theta-y\|^n} \, \lambda^{n-1} \widehat{d\theta_j}\to0$$ if $\lambda\to0$, 
	where $\Sigma_\lambda(x)$ is an open subset of the sphere $S^{n-1}$.
	
	Also the existence of the principal value upstairs is warranteed by Lemma 11 and a standard argument.
\end{proof}

\subparagraph{Acknowledgemnts:}J. B. and J. M. have been partially
		supported by 2021SGR00071 (Generalitat de Catalunya).
		J. M. has been partially
		supported by PID2020-112881GB-I00 and Severo Ochoa and
		Maria de Maeztu program for centers CEX2020-001084-MMTM2016-75390 (Mineco, Spain).

\bibliography{aggreg}
\bibliographystyle{alpha}

\noindent
{\small
	\begin{tabular}{@{}l}
		J.\ M.\ Burgu\'es\\
		J.\ Mateu\\
		Departament de Matem\`atiques\\
		Universitat Aut\`onoma de Barcelona\\
		08193 Bellaterra, Barcelona, Catalonia\\
		{\it E-mail:} {\tt josepmaria.burgues@uab.cat}\\
		{\it E-mail:} {\tt joan.mateu@uab.cat}
\end{tabular}}

\end{document}